\documentclass[leqno,12pt]{article}

\usepackage{amsthm,amsfonts,amsbsy,amssymb,amsmath,amscd, enumerate}
\usepackage[all]{xy}
\usepackage{xr-hyper}
\usepackage[colorlinks=true]{hyperref}

\usepackage{fullpage}
\usepackage{verbatim}
\usepackage{calrsfs}
\usepackage[dvips]{lscape}

\newcommand{\BA}{{\mathbb {A}}}
\newcommand{\BB}{{\mathbb {B}}}
\newcommand{\BC}{{\mathbb {C}}}

\newcommand{\BF}{{\mathbb {F}}}

\newcommand{\BH}{{\mathbb {H}}}

\newcommand{\BP}{{\mathbb {P}}}
\newcommand{\BQ}{{\mathbb {Q}}}
\newcommand{\BR}{{\mathbb {R}}}

\newcommand{\BZ}{{\mathbb {Z}}}

\newcommand{\CA}{{\mathcal {A}}}

\newcommand{\CH}{{\mathcal {H}}}

\newcommand{\CL}{{\mathcal {L}}}

\newcommand{\CP}{{\mathcal {P}}}

\newcommand{\CR }{{\mathcal {R}}}
\newcommand{\CS}{{\mathcal {S}}}

\newcommand{\CW}{{\mathcal {W}}}

\newcommand{\RN}{{\mathrm {N}}}

\newcommand{\ad}{{\mathrm{ad}}}

\newcommand{\Aut}{{\mathrm{Aut}}}

\newcommand{\Cl}{{\mathrm{Cl}}}

\newcommand{\End}{{\mathrm{End}}}

\newcommand{\Gal}{{\mathrm{Gal}}}
\newcommand{\GL}{{\mathrm{GL}}}

\newcommand{\Hom}{{\mathrm{Hom}}}

\newcommand{\JL}{{\mathrm{JL}}}

\newcommand{\ord}{{\mathrm{ord}}}

\newcommand{\inv}{{\mathrm{inv}}}

\newcommand{\rank}{{\mathrm{rank}}}

\newcommand{\Pet}{{\mathrm{Pet}}}

\newcommand{\Res}{{\mathrm{Res}}}

\newcommand{\Sel}{{\mathrm{Sel}}}

 \font\cyr=wncyr10

    \newcommand{\Sha}{\hbox{\cyr X}}

\newcommand{\Spec}{{\mathrm{Spec}}}
\newcommand{\SO}{{\mathrm{SO}}}

\newcommand{\tor}{{\mathrm{tor}}}
\newcommand{\tr}{{\mathrm{Tr}}}

\newcommand{\vol}{{\mathrm{vol}}}

\newcommand{\ab}{{\mathrm{ab}}}

\renewcommand{\mod}{\, \mathrm{mod}\, }
\renewcommand{\pmod}[1]{\, (\mathrm{mod}\, #1) }

\newcommand{\matrixx}[4]{\begin{pmatrix}
#1 & #2 \\ #3 & #4
\end{pmatrix} }

\newcommand{\wt}{\widetilde}
\newcommand{\wh}{\widehat}

\newcommand{\pair}[1]{\langle {#1} \rangle}

\newcommand{\sfrac}[2]{\left( \frac {#1}{#2}\right)}
\newcommand{\ds}{\displaystyle}
\newcommand{\ov}{\overline}

\newcommand{\lra}{\longrightarrow}
\newcommand{\lla}{\longleftarrow}
\newcommand{\ra}{\rightarrow}

\newcommand{\lto}{\longmapsto}
\newcommand{\bs}{\backslash}

\newtheorem{thm}{Theorem}[section]
\newtheorem{cor}[thm]{Corollary}
\newtheorem{lem}[thm]{Lemma}
\newtheorem{prop}[thm]{Proposition}

\theoremstyle{definition}
\newtheorem{definition}[thm]{Definition}

\theoremstyle{remark}
\newtheorem{remark}[thm]{Remark}

\numberwithin{equation}{subsection}

\begin{document}

\title{Genus Periods, Genus Points and
\\Congruent Number Problem}
\author{Ye Tian, Xinyi Yuan, and Shou-Wu Zhang}
\maketitle

\abstract

In this paper, based on  an ideal of Tian we will  establish a new sufficient condition for a positive integer  $n$ to be congruent 
 in terms of the Legendre symbols $\sfrac pq$, with $p$ and $q$ running over the prime factors of $n$.
 Our criterion generalizes previous criterions of  Heegner,  and Birch--Stephens,  Monsky, and   Tian,
 and conjecturally provides a list of positive density of  congruent numbers.
Our method of proving our criterion is  to give  formulae for the analytic Tate--Shafarevich number 
 $\CL(n)$ in terms of the so-called genus periods and genus points.
These formulae  are  derived from the Waldspurger formula and  the generalized Gross--Zagier formula of Yuan--Zhang--Zhang.

\tableofcontents

\section{Introduction}
A positive integer $n$  is called a {\em  congruent number} if it is the area of a right-angled triangle, all of whose
 sides have rational lengths.  The congruent number problem, which is the oldest unsolved major problem in number theory,  
 is the question of finding an algorithm for deciding in finite number of steps whether or not a given integer is a congruent number.
In this paper, based on  an ideal of Tian \cite{Tian} we will  establish a new {\em sufficient} condition for $n$ to be congruent 
 in terms of the Legendre symbols $\sfrac pq$, with $p$ and $q$ running over the prime factors of $n$.

This type of criterion was first given by Heegner \cite{Heegner} and Birch and Stephens \cite{BS} for some $n$ with a single odd prime factor,
 and by  Monsky \cite{Monsky}  for  some $n$ with two odd prime factors, and finally  Tian \cite{Tian} saw how to extend it to  $n$ with an 
 arbitrary number of  prime factors.
Our criterion generalizes all of this work, and we believe that it has potential applications to  the following  {\em  distribution conjecture of congruent numbers}:
{\em all $n\equiv 5, 6, 7\mod 8$ are congruent  and all but  density $0$ of $n\equiv 1,2,3$ are not congruent.}
Note that in \cite{Tu2}, Tunnell  gave a {\em necessary} condition for $n$ to be
 congruent   in terms of numbers of solutions of some equations  $n=Q(x, y, z)$ with positive  definite quadratic forms $Q(x, y, z)$ over $\BZ$.
Tunnell's criterion is also sufficient if the rank part of the  BSD conjecture is assumed.

In the following, we would like to describe our main results. Let us first consider  the elliptic curve
$$E_n: \quad ny^2=x^3-x,$$
where $n$ is assumed to be a \emph{square-free positive} integer. Then it is well known that $n$ is congruent if  and only if $E_n (\BQ)$ has positive
rank. This is  equivalent to the vanishing of $L(E_n, 1)$ under the rank part of  BSD conjecture
 $$
 \rank\ E_n(\BQ)=\ord_{s=1}L(E_n, s).
$$
By Birch--Stephens \cite{BS}, the root number
$$
\epsilon(E_n)=
\begin{cases}
1
& \mbox{if } n\equiv 1, 2, 3\pmod 8, \\
-1
& \mbox{if } n\equiv 5, 6, 7\pmod 8.
\end{cases}
$$
It follows that $\ord _{s=1}L(E_n, s)$ is even (resp. odd) if and only if $n\equiv 1, 2, 3\mod 8$ (resp. $n\equiv 5, 6, 7\mod 8$).
The density conjecture of congruent numbers follows from rank part of the BSD and the following {\em density conjecture of $L$-functions}:
 $\ord _{s=1}L(E_n, s)\le 1$ for all but density $0$ of $n$'s.

By work of  Coates--Wiles \cite{Coates-Wiles}, Rubin \cite{Rubin}, Gross--Zagier \cite{GZ} and Kolyvagin \cite{K1}, the rank part of the BSD conjecture holds for $E_n$
if  $\ord_{s=1}L(E_n, s)\leq 1$ and the Tate--Shafarevich group $\Sha(E_n)$ is finite.
Thus we define an invariant $\CL(n)$ of $E_n$ as follows:
$$
\CL(n):=
\begin{cases} \left[L(E_n, 1)/(2^{2k(n)-2-a(n)} \Omega _{n, \infty})\right]^{1/2}
& \mbox{if } \ord _{s=1}L(E_n, s)=0, \\
\left[L'(E_n, 1)/(2^{2k(n)-2-a(n)}\cdot \Omega _{n, \infty}R_n)\right]^{1/2}
& \mbox{if } \ord _{s=1}L(E_n, s)=1, \\
0 & \mbox{if } \ord _{s=1}L(E_n, s)>1.
\end{cases}
$$
Here
 \begin{itemize}
 \item $k(n)$ is the number of odd prime factors of $n$;
 \item $a(n)=0$ if $n$ is even, and $1$ if $n$ is odd;
\item the real period
$$ \Omega _{n, \infty}=\frac 2{\sqrt {n}}\int_1^\infty \frac {dx}{\sqrt{x^3-x}},$$
\item $R_n$ is twice of the N\'{e}ron--Tate height of a generator of $E_n(\BQ)/E_n(\BQ)_\tor$ (in the case of rank one).
\end{itemize}
The definition is made so that the full BSD conjecture for $E_n$ in the case $\ord _{s=1}L(E_n, s)\leq 1$ writes as
\begin{equation}\label{bsd}
\#\Sha(E_n)= \CL(n)^2.
\end{equation}
The density conjecture of congruent numbers  is equivalent to  {\em non-vanishing $\CL(n)$ for density one of $n$'s}.

The number $\CL (n)$ is a priori a complex number defined up to a sign.
In this paper, we show that $\CL (n)$ is an integer, and give a criterion for when it is odd  in terms of the parities of the genus class numbers
$$g(d):=\#(2 \Cl (\BQ(\sqrt {-d})))$$
of positive divisors $d$ of $n$.
It is clear that $g(d)$ is odd if and only if $\Cl (\BQ(\sqrt {-d}))$ has no element of exact order $4$.
Thus by R\'edei \cite{Re}, the parity of $g(d)$ can be computed in terms of the R\'edei matrix of the Legendre symbols $\ds\sfrac pq$ of  prime factors $p, q$  of $d$. 
The choice of the sign of $\CL (n)$ is not an issue in this paper since we are mainly interested in its parity.
We divide our results  naturally into two cases by the root number $\epsilon (E_n)$.

\begin{thm}\label{lg}
Let $n\equiv 1, 2, 3\pmod 8$ be a positive and square-free integer.
Then $\CL (n)$ is an integer, and
$$\CL (n)\equiv \sum_{\substack {n=d_0d_1\cdots d_\ell\\
d_i\equiv 1\pmod 8,\ i>0}}
\prod_i g(d_i)
\quad\pmod 2.$$
Here all decompositions $n=d_0\cdots d_\ell$ are non-ordered with $d_i>1$ for all $i\geq 0$. The right-hand side is considered to be 1 if $n=1$.
\end{thm}

For $n\equiv 5, 6, 7$,  we introduce an integer $\rho (n)\geq 0$ by
$$2^{\rho(n)} =[E_n(\BQ): \varphi _n (A_n(\BQ))+E_n[2]],$$
where $\varphi_n: A_n \to E_n$ is a 2-isogeny from
$A_n: 2nv^2=u^3+u$
to $E_n: ny^2=x^3-x$
defined  by
$$\varphi _n (u, v)=\left(\frac 12 \left(u+\frac 1u\right ), \frac v{2u}\left(u-\frac 1u\right)\right).$$

\begin{thm}\label{lg'}
Let $n\equiv 5, 6, 7\pmod 8$ be a positive and square-free integer.
Then $\CL (n)$ is an integer.
If $n\equiv 5, 7\pmod 8$, then $2^{-\rho(n)}\CL(n)$ is even only if
$$\sum_{\substack{{n=d_0\cdots d_\ell}\\ {d_i\equiv 1\pmod 8,\ i>0}}} \prod_i
g(d_i)\quad \equiv \sum_{\substack{{n=d_0\cdots d_\ell,}\\
{d_0\equiv 5, 6, 7\pmod 8}\\ {d_1\equiv 1, 2 ,3 \pmod 8} \\ {d_i\equiv 1\pmod 8,\ i>1}}} \prod_i g(d_i) \ \equiv 0
\quad\pmod 2.$$
If $n\equiv 6\pmod 8$, then $2^{-\rho(n)}\CL(n)$ is even only if
$$
 \sum_{\substack{{n=d_0\cdots d_\ell,}\\
{d_0\equiv 5, 6, 7\pmod 8}\\ {d_1\equiv 1, 2 ,3 \pmod 8} \\ {d_i\equiv 1\pmod 8,\ i>1}}} \prod_i g(d_i) \ \equiv 0
\quad\pmod 2.$$
Here all decompositions  $n=d_0\cdots d_\ell$ are non-ordered with all $d_i>1$.
\end{thm}

Our method of proving these theorems is  to give  formulae of $\CL(n)$ in terms of the so-called genus periods and genus points (cf. Theorems \ref{LQ} and Theorem \ref{m3}).
These formulae  are  derived from the Waldspurger formula in \cite{Wa} and  the generalized Gross--Zagier formula of Yuan--Zhang--Zhang \cite{YZZ} using an induction argument of Tian \cite{Tian}.

\begin{remark}  For each residue class in $\{1, 2, 3, 5, 6, 7\mod 8\}$,
 we believe that our formulae in Theorems \ref{lg}, \ref{lg'} give a positive density of $n$ with $\CL (n)$ odd.
 For $n\equiv 1, 2, 3\mod 8$, this is already implied by the  BSD formula \ref{bsd} modulo $2$ and  the work of Heath-Brown \cite{Heath-Brown}.
 Moreover the BSD formula \ref{bsd} modulo 2 can be checked case by case. In fact, in \cite{Monsky3}, the $\BF_2$-rank of
 $\Sel _2(E_n)/E_n (\BQ)[2]$ can be  also calculated in terms of Legendre symbols for every  $n$. 
\end{remark}

In the following we want to give some some criterions of congruent and non-congruent numbers  extending  Tian
\cite{Tian} in terms of a single genus class number.

\begin{cor}
 Let $n$ be a square-free positive integer   such that $\BQ(\sqrt{-n})$ has no ideal classes of exact order $4$.
For any integer $r$, let $A_r, B_r$ denote the following property of $n$:
\begin{align*}
&A_r(n):  \quad \#\{p\mid n: \ p\equiv 3\pmod 4\}\le r;\\
&B_r(n):\quad \#\{p\mid n: \ p\equiv \pm 3\pmod 8\}\le r.\end{align*}
Then in the following case, $n$ is a non-congruent number:
\begin{itemize}
\item  $n\equiv 1\pmod 8$ with  $A_2(n)$ or $B_2(n)$,
\item  $n\equiv  2\pmod{8}$ with   $A_0(n)$ or $B_2(n)$,
\item $n\equiv  3\pmod 8$ with   $A_1(n)$ or $B_1(n)$.
\end{itemize}
In the following case, $n$ is a congruent number:
\begin{itemize}
\item  $n\equiv 5\pmod 8$ with $A_0(n)$ or $B_1(n)$,
\item  $n\equiv 7\pmod 8$ with  $A_1(n)$ or $B_0(n)$.
\end{itemize}
\end{cor}

\begin{proof}
By R\'edei \cite{Re},   $g(d)$ is even in any  of the following cases:
\begin{itemize}
\item $d=p_1 \cdots p_k\equiv 1\pmod 8$, $p_i\equiv \pm1\pmod 8$, $k>0$;
\item $d=2p_1 \cdots p_k$, $p_i\equiv \pm1 \pmod 8 $, $k>0$;
\item $d=p_1\cdots p_k\equiv 1\pmod 8$ , $p_i\equiv 1 \pmod 4$, $k>0$.
\end{itemize}
It follows that under any of the conditions of the corollary, the following congruence holds:
$$\sum_{\substack {n=d_0d_1\cdots d_\ell\\
d_i\equiv 1 \pmod 8 ,\ i>0}}
\prod_i g(d_i)\equiv g(n)\ \pmod 2.$$
The conclusion follows from  Theorem \ref{lg} and \ref{lg'}.
\end{proof}

  \subsubsection*{Acknowledgements}
Ye Tian  would like to acknowledge the support of the  NSFC grants 11325106 and 11031004.
Xinyi Yuan would like to acknowledge the support of the National Science Foundation
under the award DMS-1330987.
 Shou-Wu Zhang would like to acknowledge the support of the
National Science Foundation under the awards DMS-0970100 and DMS-1065839.

\section{Quadratic periods and genus periods} \label{section 2}

The goal of this section is to prove Theorem \ref{lg}.
Assume that $n\equiv 1,2,3 \pmod 8$ is positive and square-free throughout this section.

 \subsection{Quadratic periods and genus periods} \label{section2.1}

Let $K_n=\BQ(\sqrt {-n})$ be the quadratic imaginary extension.
For any decomposition $n=d_1\cdot d_2$ with $d_2$
\emph{positive and odd}, we have an unramified quadratic extension
$K_n (\sqrt {d_2^*})$ of $K_n$ where $d_2^*=(-1)^{(d_2-1)/2}d_2$.
By the class field theory, the extension gives a quadratic character
$$\chi_{d_1, d_2}:\Cl_n \lra \{\pm1\}$$
on the class group $\Cl_n$ of $K_n$.
In the degenerate case $d_2=1$, we take the convention $\chi_{d_1, d_2}=1$.
Conversely, by Gauss's genus theory, any quadratic character of $\Cl_n$ comes from such a decomposition $n=d_1\cdot d_2$.

The Rankin-Selberg L-series of the elliptic curve $E:y^2=x^3-x$ twisted by $\chi_{d_1, d_2}$ is given by
 $$L(E_{K_n}, \chi_{d_1, d_2},  s)=L (E_{d_1}, s)L(E_{d_2}, s).$$
In the following, we give a formula
for $\CL(d_1)\CL(d_2)$ using the Waldspurger formula. Notice that such formulae concern the quaternion algebra determined by the local root numbers of the L-function $L(E_{K_n}, \chi_{d_1, d_2}, s)$.

Let $B$ be the quaternion algebra over $\BQ$ ramified exactly at 2 and
$\infty$.  In fact, $B$ is the classical
Hamiltonian quaternion (over $\BQ$):
$$B=\BQ+\BQ i+\BQ j+\BQ k, \quad i^2=j^2=-1,\ ij=k=-ji.$$
Let $ O_B$ be the standard  maximal order of $B$:
$$ O_B:= O_B'+\BZ \zeta, \qquad  O_B':=\BZ+ \BZ i+\BZ j+\BZ k, \qquad \zeta =(-1+i+j+k)/2.$$

Fix an embedding $\tau:K_n\hookrightarrow B$ such that the image of $O_{K_n}$ lies in
$O_B$. If $n\equiv 1\pmod8$, we further specify the embedding by
$$
\tau(\sqrt{-n})= ai+bj+ck
$$
where $n=a^2+b^2+c^2$ with $a,b,c\in \BZ$ and $4|c$.
It is a classical result of Legendre that we can find integer solutions $a,b,c$ if $n$ is not of the form $4^e(8m-1)$. The more specific condition $n\equiv 1\pmod8$ implies the existence of a solution with $4|c$. See \cite[Theorem 5]{JP} for example.

By the Jacquet--Langlands correspondence,  the newform $f_E\in S_2(\Gamma _0(32))$ corresponding to the elliptic curve $E: y^2=x^3-x$ defines an automorphic representation $\pi=\otimes_v \pi_v$ of $B^\times(\BA)$.
Note that the central character of $\pi$ and
the infinite part $\pi_\infty$ are trivial, so $\pi=\otimes_v \pi_v$ is naturally
realized as a subspace of $C^\infty(B^\times\bs
\wh{B}^\times/\wh{\BQ}^\times)$.

Denote by $\pi_\BZ$ the $\BZ$-submodule of $\pi$ consisting of elements of $\pi$ which takes integral values on $B^\times\bs\wh{B}^\times/\wh{\BQ}^\times$.
Denote $U_n={\wh R_n^\times\cdot  K_{n, 2}^\times}$, an open subgroup of $\wh B^\times$.
Here $R_n= O_{K_n}+4 O_B$ is an order of $B$ of conductor $32$.
Consider the $U_n$-invariant submodule $\pi^{U_n}_\BZ$ of $\pi_\BZ$.
We will see that $\pi^{U_n}_\BZ$ is free of rank $1$ over $\BZ$,
as the special case of $\chi=1$ in Theorem \ref{thm2.3}.
This is an integral example of the multiplicity one theorem of Tunnell \cite{Tu} and Saito \cite{Sa} reviewed in Theorem \ref{TS1} and Corollary \ref{TS2}.

Fix a $\BZ$-generator $f_n$ of $\pi^{U_n}_\BZ$, which is determined up to multiplication by $\pm1$.
Define the {\em quadratic  period} $P(d_1, d_2)$ by
$$P(d_1, d_2):=\sum _{t\in \Cl_n}f_n(t)\chi_{d_1, d_2}(t).$$

\begin{thm}\label{thm2.1}
The period $P(d_1, d_2)\ne 0$ only if $d_2\equiv 1\pmod 8$.
In that case,
$$P(d_1, d_2)=\pm 2^{k-a}\cdot w_K\cdot \CL(d_1)\CL( d_2),$$
where $2w_K$ is the number of roots of unity in $K$, $k$ is the number of odd prime factors of $n=d_1d_2$, and $a=1$ if $n$ is odd
 and $a=0$ otherwise.
\end{thm}

Now we define the {\em  genus period}  $Q(n)$ by
$$Q(n):=\sum _{t\in 2\Cl_n}f_n(t).$$
Notice that $P(n)$ and $Q(n)$ are well-defined up to signs.

\begin{thm}\label{LQ}
The number $\CL (n)$ is an integer and satisfies
$$\CL(n)\equiv
\sum_{\substack{n=d_0d_1\cdots d_\ell\\ d_i\equiv 1\pmod 8, \ i>0}}
\prod_i Q(d_i)  \quad\pmod 2.$$
Here in the sums, all decompositions  $n=d_0\cdots d_\ell$ are  non-ordered with $d_i>1$ for all $i\geq 0$.
\end{thm}

Now Theorem \ref{lg} follows from Theorem \ref{LQ} and the following result.
\begin{prop}\label{odd}
One has
$$f_n\left( (\wh{B}^{\times})^ 2\right)\subset 1+2\BZ.$$
Therefore,
$$Q(n)\equiv g(n)\quad\pmod 2.$$
\end{prop}

\subsection{Primitive test vectors} \label{section2.2}

In this subsection, we give an explicit construction of the test vector $f_n$, to prepare for the proof of the result in the last subsection.

Resume the above notations related to $K_n, B$ and $\pi$.
The local components of the automorphic representation $\pi=\otimes_v \pi_v$
of $B^\times(\BA)$ has the following properties:
\begin{itemize}
\item $\pi_\infty$ is trivial;
\item $\pi_p$ is unramified if $p\ne \infty, 2$, i.e.,  $\pi^{ O_{B, p}^\times}$ is one-dimensional;
\item $\pi_2$ has a conductor of exponential $4$ (cf. \cite{G2}), i.e., for a uniformizer
$\lambda$ (for example, $1+i$) of $B$ at $2$,
$$\pi_2^{1+\lambda^4  O_{B, 2}}\ne 0, \qquad \pi_2^{1+\lambda^3  O_{B, 2}}=0.$$
\end{itemize}
Let $U=\prod_p U_p$ be the open compact subgroup of
$\wh{ O}_B^\times$ with $U_p= O_{B, p}^\times$ if $p\ne 2$, and
$$U_2=\BZ_2^\times(1+\lambda ^4 O_{B, 2})=\BZ_2^\times(1+4 O_{B, 2}).$$
Then $\pi^U\simeq\pi_2^{U_2}$ is stable under the action of $B_2^\times$, since $U_2$ is normal in $B_2^\times$. By the irreducibility of $\pi_2$, we further have
$\pi_2^{U_2}=\pi_2$.

By definition, $\pi^U$ is a subspace of $C^\infty(B^\times\bs \wh{B}^\times/\wh{\BQ}^\times U)$, the space of maps from (the finite set) $B^\times\bs \wh{B}^\times/\wh{\BQ}^\times U$ to $\BC$. The following is a more detailed description.

\begin{thm} \label{thm2.4}
\begin{enumerate} [(1)]
\item The space $\pi^U$ is a $6$-dimensional irreducible
representation of $B_2^\times$,  with an orthogonal basis
$$f_\delta \in C^\infty(B^\times\bs \wh{B}^\times/\wh{\BQ}^\times U), \qquad
\delta\in \left\{\frac {\pm i\pm j}2, \frac {\pm j \pm k}2,  \frac
{\pm k\pm i}2\right\}\Big/\{\pm 1\}.$$
Here  for each $\delta$, the function $f_\delta$ is determined by its restriction to $1+2 O_{B, 2}$
and
$$f_\delta(1+2x)=(-1)^{\tr(\delta x)}, \qquad \forall x\in
 O_{B,2}.$$
\item The representation $\pi^U$ of $B_2^\times$ has an integral
structure $\pi_\BZ^U$ generated by
$$\qquad\qquad f_{i\pm  j}:= \frac12  (f_{\frac{i+j}{2}}\pm f_{\frac{i-j}{2}}),
\quad f_{j\pm k}:=\frac12 (f_{\frac{j+k}{2}}\pm f_{\frac{j-k}{2}}),
\quad f_{k\pm i}:=\frac 12(f_{\frac{k+i}{2}}\pm
f_{\frac{k-i}{2}}).$$
Moreover,  this $\BZ$-basis is orthonormal with
respect to the Tamagawa measure on $B^\times\bs B^\times(\BA)/\BA^\times$.
\item Let $\chi_0$ be the
character of $B^\times(\BA)$ associated to the quadratic extension
$\BQ(i)$, i.e. the composition
$$B^\times(\BA)\overset {\det} \lra \BA^\times  \simeq \BQ^\times \times
(\wh \BZ^\times\times \BR_+^\times) \lra \wh \BZ^\times\lra (\BZ/4\BZ)^\times \simeq
\{\pm 1\}.$$Then $\pi\simeq \pi\otimes\chi_0$ and
$$\quad\chi _0f_{\frac{i+j}{2}}=f_{\frac{i-j}{2}}, \qquad \chi _0f_{\frac{j+k}{2}}=f_{\frac{j-k}{2}},
\qquad \chi _0f_{\frac{k+i}{2}}=f_{\frac{k-i}{2}}.$$
\end{enumerate}
\end{thm}

To deduce the theorem, we first need the following precise description of $B^\times\bs \wh{B}^\times/\wh{\BQ}^\times U$.

\begin{lem} The following natural maps are bijective:
$$ O_{B, 2}/(\BZ_2+2 O_{B, 2})\stackrel{\sim}{\lra}
(1+2 O_{B, 2})/\BZ_2^\times(1+4 O_{B, 2})\stackrel{\sim}{\lra}
B^\times\bs \wh{B}^\times/\wh{\BQ}^\times U,$$
where the first map is
defined by $x\mapsto 1+2x$, and the second one is given by the natural inclusion
 $B_2^\times \subset \wh{B}^\times$.

Moreover, under the composition
$$
B^\times\bs B^\times(\BA)/\wh{\BQ}^\times
\lra
B^\times\bs \wh{B}^\times/\wh{\BQ}^\times U
\stackrel{\sim}{\lra}
 O_{B, 2}/(\BZ_2+2 O_{B, 2}),$$
the Tamagawa measure on $B^\times\bs B^\times(\BA)/\BA^\times$ transfers to the Haar measure of (the finite abelian group) $ O_{B, 2}/(\BZ_2+2 O_{B, 2})$ of total volume 2.
\end{lem}
\begin{proof}
We first prove the bijectivity.
The first map is clearly a group isomorphism. For the second
map, we use the class number one
property of $B$, i.e.,
$$\wh B^\times =B^\times\cdot \wh  O_B^\times=B^\times\cdot B_2^\times \cdot U^{(2)}.$$
It follows that
$$B^\times \bs \wh B^\times /\wh\BQ^\times U\simeq H\bs B_2^\times/\BQ_2^\times U_2, \quad H=B^\times \cap (U^{(2)}\cdot B_2^\times).$$
It is easy to see
that $H$ is a semi-product of $\lambda^\BZ$, where $\lambda\in  O_B$ is an element
with reduced norm $2$,  and the subgroup
$$ O_B^\times =\left\{\pm1, \quad \pm i, \quad \pm j, \quad \pm k, \quad \frac {\pm 1\pm i\pm j\pm k}2\right\}.$$
The group $ O_B^\times$ is a semi-product of $\mu_3$ generated by
$\zeta=(-1+i+j+k)/2$ and
$$( O_B' )^\times=\{\pm1, \quad \pm i, \quad \pm j, \quad \pm k\}.$$
Consider the filtration of $B_2^\times$ given by
$$B_2^\times \supset  O_{B, 2}^\times \supset 1+\lambda
 O_{B, 2}^\times \supset 1+2 O_{B, 2},$$
and its induced filtration
$$H\supset  O_B^\times \supset ( O_B')^\times\supset \mu_2.$$
It is straight forward to check that these two exact sequences have
isomorphic sub-quotients. It follows that the map $H\to B_2^\times$
induces an exact sequence
$$1\lra \mu_2\lra H\lra B_2^\times /(1+2 O_{B, 2})\lra 1.$$
In other words, the $B_2^\times$ is generated by $H$ and the  normal
subgroup $1+2 O_{B, 2}$ with intersection $H\cap (1+2 O_{B,
2})=\mu_2.$ Thus
$$H\bs B_2^\times/\BQ_2^\times U_2\stackrel{\sim}{\lla} \mu_2\bs (1+2 O_{B, 2}/1+4 O_{B, 2}
)\stackrel{\sim}{\lla}   (1+2 O_{B, 2})/\BZ_2^\times (1+4 O_{B,
2}).$$ The other two relations can be verified similarly.

Now we treat the measure. Note that the Tamagawa measure gives
$B^\times\bs B^\times(\BA)/\BA^\times$ total volume 2.
Then the induced measure on $ O_{B, 2}/(\BZ_2+2 O_{B, 2})$ also has total volume 2.
It suffices to check that the induced measure is uniform.
Equivalently, we need to show that
$\vol(B^\times\bs B^\times g \wh{\BQ}^\times U)$
is constant in $g\in \wh B^\times$.
By the first part of the lemma, we can always take a representative
$g\in 1+2 O_{B, 2}$ for the double coset $B^\times g \wh{\BQ}^\times U$. The key is that $g_p=1$ for $p\neq 2$. It follows that
$$B^\times\bs B^\times g \wh{\BQ}^\times U
=B^\times\bs B^\times \wh{\BQ}^\times U g,$$
whose volume is independent of $g$ since the measure is invariant under the right translation.
\end{proof}

In the lemma, the right
multiplication action of $B_2^\times=H\cdot (1+2 O_{B, 2})$ on
$\wh{B}^\times$ induces its action on $ O_{B, 2}/(\BZ_2+2 O_{B,
2})$ given by right conjugation of $H$ and translation of $x$, for
elements $1+2x\in 1+2 O_{B, 2}$.

Consider the space $\CA_0\subseteq C^\infty(B^\times\bs
\wh{B}^\times/\wh{\BQ}^\times)$ of forms perpendicular to forms
$\chi\circ \det$ where $\chi$ runs over all characters of
$\BQ^\times\bs \wh{\BQ}^\times$, then $\pi\subset \CA_0$. Let
$\CA_0^U$ be the subspace of $\CA_0$ of forms invariant under $U$,
then $\pi^U\subset \CA_0^U$.

The restriction map
$$\CA_0^U\lra \BC[ O_{B, 2}/(\BZ_2+2 O_{B, 2})], \qquad f\lto (\phi_f: x\mapsto f(1+2x))$$
define an isomorphism between $\CA_0^U$ and the space $\CA_1$ of
functions $\phi $ on $ O_{B, 2}/(\BZ_2+2 O_{B, 2})$ perpendicular
to the characters $1$ and $(-1)^\tr$ on $ O_{B, 2}/(\BZ_2+2 O_{B, 2})$.
Here the trace map $$\tr:  O_{B,
2}/(\BZ_2+2 O_{B, 2})\lra \BZ_2/2\BZ_2$$
is induced form the reduced trace.
The vector space $\CA_1$
is decomposed into the direct sum
$$\CA_1=\sum _{\psi\in \Psi} \BC \psi,$$
where $$\Psi=\left\{\psi\in \Hom (  O_B/(2 O_B+\BZ), \mu_2), \quad
\psi\ne 1,  (-1)^{\tr }\right\}$$
is a set of quadratic characters $\psi$ of
$ O_{B, 2}/(\BZ_2+2 O_{B, 2})$.

We have an explicit description of forms in $\CA_0^U$ corresponding to $\Psi$. Let $\wp=\lambda  O_{B, 2}$
be the maximal ideal of $ O_{B, 2}$. Then the trace map defines a
perfect pairing
$$ O_{B, 2}\otimes \wp^{-1}\lra \BZ_p, \qquad (x, y)\lra \tr(xy).$$
It induces a perfect pairing
$$( O_{B, 2}/2 O_{B, 2})\otimes (\wp^{-1}/2\wp^{-1})\lra \mu_2, \qquad (x, y)\mapsto (-1)^{\tr(xy)}.$$
It is easy to see that  $\Psi$ corresponds to the subset $\bar\Delta$
of elements $\bar \delta\in \wp^{-1}/2\wp^{-1}$ with the following
properties:
$$\tr(\delta)=0\pmod 2, \qquad \delta\ne 0, 1\pmod 2.$$
Note that the set
$$\Delta= \left\{\frac {\pm i\pm j}2, \frac {\pm j \pm k}2,  \frac
{\pm k\pm i}2\right\}$$
in Theorem \ref{thm2.4} is contained in
$\wp^{-1}$.
Thus we can identify $\bar\Delta=\Delta/\{\pm1\}$.
For each $\delta \in \Delta/\{\pm1\}$, the corresponding form
$f_\delta$ is given by
$$f_\delta (g)=(-1)^{\tr (\delta x)},$$for any $g=bh(1+2x)u  \in \wh B^\times$
with $b\in B^\times$, $h\in H$, $x\in O_{B, 2}$, and $u\in
\wh{\BQ}^\times U.$
Hence, the space $\CA_0^U$ is $6$-dimensional with the
explicit decomposition
$$\CA_0^U=\sum_{\delta\in \Delta/\{\pm1\}} \BC f_\delta$$
into characters of
$(1+2 O_{B,2})/(1+4 O_{B, 2})$.

\begin{proof}[Proof of Theorem \ref{thm2.4}]

For (1), it suffices to prove that $\CA_0^U$ is irreducible as a representation of $G:=B_2^\times/(1+4 O_{B,
2})$. Note that $G$ contains a normal and commutative finite
subgroup $C=(1+2 O_{B, 2})/(1+4 O_{B, 2})$. Thus any invariant subspace $V$ of
$\CA_0^U$ is a direct sum
$$V=\oplus _{\chi \in X}V_\chi$$
over some multiset $X$ of characters of $C$.
The multiset $X$ is stable under the conjugation of
$G$. We have seen that $V_\chi$ are all one-dimensional, and $X$ is
included into $\Psi$, the set of characters induced by elements in
$\Delta$. Thus we need only prove that $G$ acts transitively on
$\Delta$ by conjugations. In fact,  $\Delta$ is a principle
homogenous space of $ O_B^\times/\mu_2$ under conjugation.

Now we treat (2).
Note that for any $h\in H$, $x\in  O_{B, 2}$, and $\delta\in
\Delta$, we have that $h\delta h^{-1}\in \Delta$ and
$$\pi(h(1+2x))f_\delta=\psi_\delta(x)f_{h\delta h^{-1}}
=\pm f_{h\delta h^{-1}},$$
where
$\psi_\delta\in \Psi$ denotes the character $x\mapsto (-1)^{\tr
(\delta x)}$ on $ O_{B, 2}/(\BZ_2+2 O_{B, 2})$. Therefore, the action
of $B_2^\times$ on $f_{i\pm j}$ is given by
$$\pi(h)f_{i\pm j}=f_{hih^{-1}\pm hjh^{-1}},\qquad \pi(1+2x)f_{i\pm j}\in
\{\pm f_{i+j}, \pm f_{i-j}\}.$$
Similar results hold for $f_{j\pm k}$ and
$f_{k\pm i}$.
Thus $\pi_\BZ^U$ is an integral structure on $\pi^U$.
The orthonormality of the basis is a simple consequence of the previous result on the measures.

For (3), it is clear that $\chi_0$ is invariant under the left action of
$B^\times \cdot H$ and the right action of $U$ and its restriction on
$1+2 O_{B, 2}$ is given by $\chi_0(1+2x)=(-1)^{\tr x}$ for any
$x\in  O_{B, 2}$.  Thus for any $x\in  O_{B, 2}$,
$$\chi_0f_{\frac{i+j}{2}}(1+2x)=(-1)^{\tr(x+\frac{i+j}{2}x)}=(-1)^{\tr(x\frac{i-j}{2})}
(-1)^{\tr((1+j)x)}=f_{\frac{i-j}{2}}(1+2x).$$
\end{proof}

\begin{thm}\label{thm2.3}
Let $K$ be an imaginary quadratic field and $\chi$ a
quadratic character of $\wh{K}^\times/K^\times\wh{ O}_K^\times$
such that $L(E_{K},\chi,s)$ has root number $+1$ (so that $2$ cannot split
in $K$).  Let $\varpi$ be a uniformizer of $K_2$ and $\chi_2$ the
2-component of $\chi$. Fix a $\BQ$-embedding of $\tau:K\hookrightarrow B$ such
that $ O_K$ is contained in $ O_B$. Then the vector space
$$\pi^{U, \chi_2}:=\{f\in \pi^U, \ \pi(t)f=\chi_2(t)f, \ \forall t\in
K_2^\times\}$$ is one-dimensional. All the possible cases of
$(K_2, \chi_2(\varpi))$ are listed below:
$$(\BQ_2(\sqrt{-3}), 1), \quad (\BQ_2(\sqrt{-1}), \pm 1), \quad
(\BQ_2(\sqrt{-2m}), (-1)^{\frac{m-1}{2}}),\  m\equiv 1, 3, 5, 7\pmod
8.$$ Let $g\in B^\times_2$ be such that $\tau_2':=g^{-1}\tau_2 g$ is
given by
$$\begin{aligned}
&(-1+\sqrt{-3})/2\longmapsto \zeta, \quad \sqrt{-1}\longmapsto
k,\qquad\qquad\
\sqrt{-2}\longmapsto i+j,\\
&\sqrt{-10}\longmapsto i-3j,\qquad \sqrt{-6}\longmapsto i+j+2k,\quad
\sqrt{-14}\longmapsto -3i+j-2k,\end{aligned}$$
respectively in the above cases.  Then the vector
$f=\pi(g)f_0$ lies in $\pi^{U, \chi_2}$, where
$$f_0=\begin{cases}
f_{\frac{i-j}{2}}+f_{\frac{j-k}{2}}+f_{\frac{k-i}{2}},\quad
&\text{if $K_2=\BQ_2(\sqrt{-3})$},\\
f_{i\pm j}=\frac{1}{2}\left(f_{\frac{i+j}{2}}\pm
f_{\frac{i-j}{2}}\right), \quad &\text{if $(K_2,
\chi_2(\varpi))=(\BQ_2(\sqrt{-1}), \pm 1)$},\\
f_{\frac{i+j}{2}}, &\text{if $K_2=\BQ_2(\sqrt{-2m})$ with $m$ odd.}
\end{cases}$$
Moreover, $\pi_\BZ^{U, \chi_2}:=\pi_\BZ^U\cap \pi^{U, \chi_2}=\BZ f$.
\end{thm}

\begin{definition} \label{def primitive}
The automorphic forms $f$ and $-f$ described in the theorem are called \emph{primitive test vectors} for $(\pi, \chi)$.
\end{definition}

The theorem can be interpreted by the multiplicity one theorem of Tunnell \cite{Tu} and Saito \cite{Sa} reviewed in Theorem \ref{TS1} and Corollary \ref{TS2}.
In fact, the space
$$\pi^{\wh O_B^{2,\times}, \chi_2}:=\{f\in \pi^{\wh O_B^{2,\times}}, \ \pi(t)f=\chi_2(t)f, \ \forall t\in
K_2^\times\}$$
is at most one-dimensional by the multiplicity one theorem. The theorem confirms that it is one-dimensional and constructs an explicit generator of the integral structure.

\begin{proof}[Proof of Theorem \ref{thm2.3}]
It suffices to show that $f_0$ is $\chi_2$-invariant under the embedding
$\tau_2':K\hookrightarrow B$.

First consider the case $K_2=\BQ_2(\sqrt{-3})$, where
$$K_2^\times/\BQ_2^\times(1+4 O_2)= O_2^\times/\BZ_2^\times(1+4 O_2)$$
is cyclic of order $6$ and generated by $\zeta$ and $1+2\zeta$. Note
that $$\zeta^{-1}i\zeta=j, \quad \zeta^{-1}j \zeta=k, \quad
\zeta^{-1}k\zeta=i.$$ Thus the subspace of $\pi^U$ of forms fixed by
$\zeta\in H$ is of 2-dimensional with basis
$$f_{\frac{i-j}{2}}+f_{\frac{j-k}{2}}+f_{\frac{k-i}{2}}, \quad
f_{\frac{i+j}{2}}+f_{\frac{j+k}{2}}+f_{\frac{k+i}{2}}.$$  Moreover,
note that $\psi_\delta(\zeta)=1$ for $\delta=\frac{i-j}{2},
\frac{j-k}{2}, \frac{k-i}{2}$ and $\psi_\delta(\zeta)=-1$ otherwise.
Thus $\pi^{U, \chi_2}$ is one-dimensional with basis
$f_{\frac{i-j}{2}}+f_{\frac{j-k}{2}}+f_{\frac{k-i}{2}}$.

In the case $K_2=\BQ(\sqrt{-1})$, let $\varpi=k-1$, we have that
$$K_2^\times/\BQ_2^\times(1+4 O_2)=\varpi^{\BZ/4\BZ}\times \langle
1+\varpi, 1+2\varpi\rangle.$$ Note that $1+2\varpi=2k-1$, and
$\psi_\delta(k)=1$ if $\delta=\frac{i\pm j}{2}$ and
$\psi_\delta(k)=-1$ otherwise. Thus the subspace of $\pi^U$ of forms
fixed by $1+2\varpi$ is of 2-dimensional with basis
$$f_{i+j}, \quad f_{i-j},$$where $1+\varpi=k$ acts trivially
since $k^{-1}ik=-i, k^{-1}jk=-j$. Finally, since $\varpi=k-1\in H$
and
$$\varpi^{-1} i\varpi=j, \quad \varpi^{-1}j \varpi=-i, \varpi^{-1} k
\varpi=k,$$ we have that $\pi^{U, \chi_2}$ is of one-dimensional
with base $f_{i\pm j}$ for $\chi_2(\varpi)=\pm 1$.

For the case $K_2=\BQ_2(\sqrt{-2n})$ with $n=1, 3, 5, 7$, let
$\varpi=\sqrt{-2n}$, we have that
$K_2^\times/\BQ_2^\times(1+4 O_2)$ is generated by the order $2$
element $\varpi$ and order 4 element $1+\varpi$.   The embedding
$\tau_2'$ maps $1+\varpi$ to $k(1+2(\zeta +i))\mod
\BZ_2^\times(1+4 O_{B, 2})$. Note that $kik^{-1}=-i, kjk^{-1}=-j$
and
$$\psi_\delta(\zeta+i)=\begin{cases}1, \quad &\text{if $\delta=\frac{i+j}{2},
\frac{i+k}{2}, \frac{j-k}{2}$},\\
-1, &\text{otherwise}.\end{cases}$$ It follows that the subspace of
$\pi^U$ fixed by $\tau_2'(1+\varpi)$ is of one-dimensional with
basis $f_{\frac{i+j}{2}}$.  We have the following decompositions of
$\tau_2'(\varpi)\in B_2^\times=H\cdot (1+2 O_{B, 2}) \mod
\BZ_2^\times(1+4 O_{B, 2})$:
$$\begin{aligned}
&i+j\in H, \ &&  i-3j\equiv (i+j)(1+2k), \\
&i+j+2k\equiv (j-i)(1+2(\zeta-k)), \ &&
-3i+j-2k=(i-j)(1+2(\zeta-k)).\end{aligned}$$ Note that
$\psi_{\frac{i+j}{2}}(k)=1$ and $\psi_{\frac{i+j}{2}}(\zeta)=-1$, we
know that $f_{\frac{i+j}{2}}$ is $\chi_2$-invariant.
\end{proof}

\subsection{Proofs of Theorems \ref{thm2.1}, \ref{LQ}
and Proposition \ref{odd} }

Resume the notations in \S \ref{section2.1}.
Especially, $f_n$ is a basis of $\pi_\BZ^{U_n}$ with
$$U_n={\wh R_n^\times\cdot  K_{n, 2}^\times}, \quad R_n= O_{K_n}+4 O_B.$$
We first connect it to the primitive test vectors in \S \ref{section2.2}.

Recall that in \S \ref{section2.2}, we have introduced
$$U= \wh O_{B}^{2,\times}\cdot U_2, \quad U_2=\BZ_2^\times(1+4 O_{B, 2}).$$
In Theorem \ref{thm2.3} and Definition \ref{def primitive}, we have introduced the primitive test vectors for $(\pi,\chi)$.
For the connection, it is easy to verify $U_n=U\cdot  K_{n, 2}^\times$.
Hence, $f_n$ is a primitive test vector for $(\pi,\chi)$ if and only if $\chi_2=1$.

\subsubsection*{Proof of Theorem \ref{thm2.1}}

Write $K=K_n$ for simplicity. The goal is to treat
$$P(d_1, d_2)=\sum_{t\in \Cl_n} f_n(t) \chi_{d_1, d_2}(t).$$
The tool is the Waldspurger formula.

By $\Cl_n=K^\times\bs \wh K^\times/\wh O_K^\times$, the summation is essentially an integration on $K^\times\bs \wh K^\times$.
Since $f_n$ is invariant under the action of $K_2^\times$, the integration is nonzero only if
$\chi_{d_1, d_2}$ is trivial on $K_{2}^\times$.
In other words, $K_2(\sqrt{d_2^*})$ splits into two copies of $K_2=\BQ_2(\sqrt{-n})$.
This is equivalent to $d_2^*\equiv 1\pmod 8$. Then $d_2\equiv \pm 1\pmod 8$.
We will exclude the case $d_2\equiv -1\pmod 8$ later.

Assume $d_2^*\equiv 1\pmod 8$.
Then $\chi_{d_1, d_2}$ is trivial on $K_{2}^\times$, and $f_n$ is a primitive test vector for
$(\pi,\chi_{d_1,d_2})$ as described in Theorem \ref{thm2.3}.
In particular,
$$(f_n, f_n)_\Pet=\begin{cases}
6, \ &\text{if $K_2=\BQ_2(\sqrt{-3})$},\\
1, &\text{if $K_2=\BQ_2(\sqrt{-1})$},\\
2, &\text{if $K_2=\BQ_2(\sqrt{-2m})$ with $m=1, 3, 5,
7$}.\end{cases}$$

Apply the explicit Waldspurger formula in Theorem \ref{thm4.2}.
We have
$$|P(d_1, d_2)|^2=\frac{w_K^2}{2^5 3^b \pi^3}\cdot \frac{(f_n, f_n)_\Pet}{(f', f')_\Pet}\cdot L(E_{d_1},1)L(E_{d_2},1),$$
where $b=1$ if $2$ is inert in $K$ and $b=0$ otherwise, and  $f'$ is the normalized new form in the automorphic representation of $\GL_2(\BA)$ associated to $E$.

We claim that
$$(f' f')_\Pet=\Omega_{d_1, \infty}\Omega_{d_2, \infty}\cdot \frac{|D|^{1/2}}{2^8 \pi^3 e},$$
where $D$ is the discriminant of $K$.
Let $\phi=\sum_{n=1}^\infty a_n q^n$ be the corresponding newform of weight $2$. Note that
$$(\phi, \phi)_{\Gamma_0(32)}=\iint_{\Gamma_0(32)\bs \CH} |\phi(z)|^2 dxdy$$
and
$$(f', f')_\Pet=\int_{\GL_2(\BQ)\bs \GL_2(\BA)/\BQ^\times } |f'(g)|^2 dg
$$
 are related by
 $$\frac{(\phi, \phi)_{\Gamma_0(32)}}{\vol(X_0(32))}=\frac{(f', f')_\Pet}{2}, \qquad \text{where $\vol(X_0(32))=16\pi$}.$$
Let $\varphi: X_0(32)\ra E$ be a modular parametrization of degree $2$, and $\omega$  the N\'{e}ron differential on $E$, and $\Omega=\int_{E(\BR)} \omega$.
Note that
$$\varphi^*\omega=4\pi i \phi(z) dz, \quad
2^{-1}\Omega^2=\iint_{E(\BC)} |\omega \wedge \ov{\omega}|,$$
and thus
$$\Omega^2=32\pi^2 (\phi, \phi)_{\Gamma_0(32)}.$$
By definition,
$$\Omega_{d_1, \infty}\Omega_{d_2, \infty}=\Omega^2/\sqrt{d_1d_2}=2^{e-1}\Omega^2/\sqrt{|D|},$$
where $e=1$ if $2\nmid D$ and $e=2$ otherwise. Put all these together, we have the formula for $(f' f')_\Pet$.

Hence, we have
$$|P(d_1, d_2)|^2=2^{4+c} w_K^2 \frac{L(E_{d_1},1)}{\Omega_{d_1, \infty}}\cdot \frac{L(E_{d_2},1)}{\Omega_{d_2, \infty}},$$
where $c=0$ if $8\nmid D$ and $c=1$ otherwise.
It gives the formula of the theorem.

It remains to prove that $d_2\equiv -1\pmod 8$ implies $P(d_1, d_2)=0$.
This is a direct consequence from the formula we just proved, since $L(E_{d_1},1)=0$ by considering the root number in this case.

\subsubsection*{Proof of Theorem \ref{LQ}}

Let $h_2(n)=\dim _{\BF_2}\Cl_n/2\Cl_n$.
By Gauss's genus theory, $h_2(n)+1$ is exactly equal to the number of prime factors of the discriminant of $K_n$, and any character of $\Cl_n/2\Cl_n$ is of the form $\chi_{d_1,d_2}$ for some decomposition $d=d_1d_2$ with $d_2$ positive and odd.
Moreover, a repetition $\chi_{d_1',d_2'}=\chi_{d_1,d_2}$ occurs only if
$(d_1',d_2')=(d_2,d_1)$ and $n\equiv3\pmod8$.

Hence, we have the character formula
$${\sum_{n=d_1d_2}} '\chi_{d_1, d_2}(t)=2^{h_2(n)}\delta _{2\Cl (K_n)}(t), \qquad t\in \Cl_n,$$
where the sum is over ordered (resp. non-ordered) decompositions $d=d_1d_2$ if $n\equiv 1\pmod 8$ (resp. $n\equiv 3\pmod 8$), and requires $d_2$ to be odd if $n\equiv 2\pmod 8$.
As a result, we have
\begin{equation} \label{PQ}
{\sum_{n=d_1d_2}} ' P(d_1, d_2)=2^{h_2(n)} Q(n).
\end{equation}
The summation follows the same rule as above.

The following lemma shows the symmetry in the case $n\equiv 1\pmod 8$.

\begin{lem}\label{lem2.6}
Assume $n\equiv 1\pmod 8$. Then
for any decomposition $d=d_1d_2$ with $d_1,d_2>0$,
$$P(d_1, d_2)= P(d_2, d_1).$$
\end{lem}
\begin{proof} Let $\chi_0$ be the character on $\wh{B}^\times$
corresponding to the extension $\BQ(i)$ over $\BQ$,  defined in
Theorem \ref{thm2.4}. The two quadratic characters are related by
$$\chi_{d_1, d_2}=\chi_{d_2, d_1} \cdot \chi_0.$$
In fact, for any $t\in \wh{K}^\times$, we have
$$\chi_{d_1, d_2}(t)
\chi_{d_2, d_1}(t)=\frac{\sigma_t(\sqrt{d_1})}{\sqrt{d_1}}\frac{\sigma_t(\sqrt{d_2})}
{\sqrt{d_2}}=\frac{\sigma_t(i)}{i}=\chi_0(t).$$

In the notation of Theorem \ref{thm2.3}, the primitive test vector is given by
$f_n=\pi(g)f_0$ (up to $\{\pm1\}$) with $g\in B_2^\times$ and
$$f_0=\frac 12 (f_{\frac {i+j}2}+f_{\frac {i-j}2})=f_{\frac {i+j}2}\cdot (\frac {1+\chi_0}2).$$
We claim that $\chi_0(g)=1$ by our special choice of $\tau:K_n\hookrightarrow B$ at the beginning.

Assuming $\chi_0(g)=1$, then
$$\begin{aligned}
P(d_2, d_1)&=\sum_t f_{\frac{i+j}{2}}(tg)\frac {1+\chi_0(t)}2\chi_{d_2, d_1}(t)\\
&=\sum_t f_{\frac{i+j}{2}}(tg)\frac {1+\chi_0(t)}2 \chi_0(t)\chi_{d_1, d_2}(t)\\
&= \sum_t f_{\frac{i+j}{2}}(tg)\frac {1+\chi_0(t)}2\chi_{d_1, d_2}(t) \\
&=P(d_1, d_2).
\end{aligned}$$

It remains to check $\chi_0(g)=1$.
Recall that $g\in B^\times_2$ is an element such that $\tau_2'=g^{-1}\tau_2 g: K_{n,2} \hookrightarrow B_2$ gives $\tau_2'(\sqrt{-1})=k$.
Recall that the embedding $\tau:K_n\hookrightarrow B$ is defined by
$$
\tau(\sqrt{-n})= ai+bj+ck
$$
where $n=a^2+b^2+c^2$ with $a,b,c\in \BZ$ and $4|c$.
Thus the equation for $g$ is just
$$g^{-1}\cdot \frac{1}{\sqrt n}(ai+bj+ck)\cdot  g=k.$$
Here $\sqrt n$ denotes a square root of $n$ in $K_2$.
Explicit computation gives a solution
$$
g_0= ai+bj+(c+\sqrt n)k.
$$
For this solution, we have
$$
\det(g_0)=a^2+b^2+(c+\sqrt n)^2=2(n+c\sqrt n).
$$
Note that $n+c\sqrt n\equiv 1 \pmod 4$ by the condition $4|c$, and thus $\chi_0(g_0)=1$.
It is easy to see that any other solution is of the form $g=g_0 (u+v k)$ for $u,v\in \BQ_2$. Then we have $\chi_0(u+v k)=1$ and thus $\chi_0(g)=1$.
\end{proof}

\begin{lem}
One has
$$\sum _{n=d_1d_2}\epsilon (d_1, d_2)\CL (d_1)\CL(d_2)=Q(n),$$
where $\epsilon (d_1, d_2)=\pm 1$, and the sum is over non-ordered  decompositions $n=d_1d_2$ such that $d_1,d_2>0$ and $d_2\equiv 1\pmod 8$.
\end{lem}

\begin{proof}
Writing equation (\ref{PQ}) in terms of non-ordered decompositions, we have
\begin{equation*}
 \sum_{n=d_1d_2} P(d_1, d_2)=2^{h_2(n)-\delta} Q(n),
\end{equation*}
where $\delta =1$ if $n\equiv 1\pmod 8$ and $\delta=0$ otherwise.
Here in the case $n\equiv 1\pmod 8$, we have used the symmetry
$P(d_1, d_2)=P(d_2, d_1)$.
Apply Theorem \ref{thm2.1}.
\end{proof}

Finally, we are ready to derive Theorem \ref{LQ}.

\begin{proof}[Proof of Theorem \ref{LQ}]
Since $\CL (1)=1$, the above lemma gives a recursive formula
$$\pm \CL (n)=Q(n)-
\sum_{\substack{n=d_1d_2 \\ d_2\equiv 1\pmod 8, \ d_2>1}}
\epsilon (d_1, d_2)\CL (d_1)\CL(d_2).$$
Here the sum is over non-ordered decompositions.
This formula determines $\CL(n)$ uniquely.
In particular, $\CL(n)$ is an integer.

Now we prove the congruence formula
$$\CL(n)\equiv
\sum_{\substack{n=d_0d_1\cdots d_\ell\\ d_i\equiv 1\pmod 8, \ i>0\\
d_i>1,\ i\geq0}}
\prod_i Q(d_i)  \quad\pmod 2.$$
It suffices to prove that the congruence formula (applied to every $P(d_1)$ and $P(d_2)$ below) satisfies the recursive formula
$$Q(n) \equiv
\sum_{\substack{n=d_1d_2\\d_2\equiv 1\pmod 8, \ d_2>0}}
\CL (d_1)\CL(d_2)
\quad\pmod 2.$$
Namely, we need to check that
\begin{multline*}
Q(n)\equiv
\sum_{\substack{n=d_1d_2\\d_2\equiv 1\pmod 8, \ d_2>0}}
\left(\sum_{\substack {d_1=d_0'd_1'\cdots d_{\ell'}'\\
d_j'\equiv 1\pmod 8, \ j>0\\
d_j'>1,\ j\geq0}}
\prod_{j\geq 0} g(d_j') \right)
\left(\sum_{\substack{d_2=d_0''d_1''\cdots d_{\ell''}''\\
d_k''\equiv 1 \pmod 8, \ k>0\\
d_k''>1,\ k\geq0}}
\prod_{k\geq 0} g(d_k'')\right)
\quad\pmod 2.
\end{multline*}
The right-hand side is a $\BZ$-linear combination of
$$
\prod_{j=0}^{\ell'} g(d_j')
\prod_{k=0}^{\ell''} g(d_k'').
$$
We consider the multiplicity of this term in the sum.
Each appearance of such a term gives a partition
$$
\{d_1',\cdots, d_{\ell'}', d_0'',\cdots, d_{\ell''}''\}
=\{d_1',\cdots, d_{\ell'}'\}\coprod \{d_0'',\cdots, d_{\ell''}''\}.
$$
If the set on the left-hand side is non-empty, the number of such partitions is even. Then the contribution of this set in the sum is zero modulo 2.
Thus, we are only left with the empty set, which corresponds to the unique term $g(n)$ on the right. This proves the formula.
\end{proof}

\subsubsection*{Proof of Proposition \ref{odd}}

The proof easily follows from the explicit result in Theorem \ref{thm2.3}.
In fact, take $K=K_n$ and $\chi=1$ in the theorem. We see that the primitive test vector
$f_n=\pi(g)f_0$ for some $g\in B_2^\times$, where
$$f_0=\begin{cases}
f_{\frac{i-j}{2}}+f_{\frac{j-k}{2}}+f_{\frac{k-i}{2}},\quad
&\text{if $K_2=\BQ_2(\sqrt{-3})$},\\
f_{i+ j}=\frac{1}{2}\left(f_{\frac{i+j}{2}}+
f_{\frac{i-j}{2}}\right), \quad &\text{if $K_2=\BQ_2(\sqrt{-1})$},\\
f_{\frac{i+j}{2}}, &\text{if $K_2=\BQ_2(\sqrt{-2m})$ with $m$ odd.}
\end{cases}$$
Note that the case $(K_2, \chi_2(\varpi))=(\BQ_2(\sqrt{-1}), -1)$ does not occur here.
It is immediate that $f_0$ and $f$ take odd values everywhere in the first and the third cases.

Assume that we are in the case $K_2=\BQ_2(\sqrt{-1})$.
By Theorem \ref{thm2.4}, $\ds\chi _0f_{\frac{i+j}{2}}=f_{\frac{i-j}{2}}$.
For any $h\in \wh{B}^{\times}$, we have
$$
f_n(h^2)=f_0(h^2g)
=\frac{1}{2} f_{\frac{i+j}{2}} (h^2g)(1+\chi_0(h^2g))
=f_{\frac{i+j}{2}} (h^2g)
=\pm 1,
$$
which is odd. Here we have used the fact $\chi_0(g)=1$, which has been treated in the proof of Lemma \ref{lem2.6}.

\section{Quadratic points and genus points}

This section treats $\CL(n)$ for $n\equiv 5,6,7 \pmod 8$.
The goal is to prove Theorem \ref{lg'}.
We assume  $n\equiv 5,6,7 \pmod 8$ throughout this section.
The method is to construct rational points using the tower $X=\lim _UX_U$
of modular curves $X_U$.

\subsection{Quadratic points and genus points} \label{section3.1}

In the following, we will mainly work on the elliptic curve $A: 2y^2=x^3+x$ (instead of $E:y^2=x^3-x$), which is isomorphic to $(X_0(32), \infty)$.
Fix an identification $i_0:(X_0(32), \infty)\to A$.
We will introduce a morphism $f_n:X_V\to A$ from certain modular curve $X_V$ to $A$, and use this morphism to produce Heegner points on $A$.

\subsubsection*{Test vector}
Recall that the open compact subgroup $U_0(32)$ of $\GL_2(\wh\BQ)$ is given by
$$
U_0(32)= \left\{ \matrixx{a}{b}{c}{d}\in \GL_2(\wh\BZ): 32|c \right\}.
$$
Define another open compact subgroup
$$
U= \left\{ \matrixx{a}{b}{c}{d}\in U_0(32): 4|(a-d) \right\}.
$$
Then $U$ is a normal subgroup of $U_0(32)$ of index two.

Denote by $f_0: X_U\to A$ the natural projection map $X_U\to X_0(32)$.
It is finite and \'etale of degree 2.
Note that the geometrically connected components of $X_U$ are parametrized by $\Spec\, \BQ(i)$. Then it is easy to figure out that $X_U\cong {X_0(32)}_{\BQ(i)}$ over $\BQ$, and under this identification $f_0$ is the natural map by the base change. Then
$$\Aut_\BQ(X_U)=\Aut_{\BQ(i)}({X_0(32)}_{\BQ(i)})\rtimes \{1, \epsilon\},$$
where $\epsilon$ is the Hecke operator given by $\matrixx{1}{}{}{-1}$, which is also the automorphism coming from the non-trivial automorphism of $\Spec\, \BQ(i)$.

For $n\equiv 7\pmod 8$, let $f_n:X_0(32)\to A$ be the identity map $i_0:X_0(32)\to A$.
For $n\equiv 5,6\pmod 8$, define $f_n: X_U\ra A$ by
$$f_n:=\begin{cases} f_0-f_0\circ [i], \qquad &\text{if $n\equiv 5\pmod 8$},\\
f_0\circ [i], &\text{if $n\equiv 6\pmod 8$}.
\end{cases}$$

Denote $K_n=\BQ(\sqrt {-n})$ as before.
Embed $K_n$ into $M_2(\BQ)$ by
$$\sqrt{-n} \longmapsto \matrixx{-1}{1/4}{-4(n+1)}{1}, \qquad \matrixx{}{1/4}{-4n}{}, \qquad \matrixx{\delta}{2}{-(n+\delta^2)/2}{-\delta},$$
according to $n\equiv 5, 6, 7\pmod 8$ respectively. Here $\delta$ is an integer such that $\delta^2\equiv -n\pmod {128}$ in the case $n\equiv 7\pmod 8$.

The embeddings look arbitrary, but they are chosen on purpose. For $n\equiv 5, 6\pmod 8$, the embeddings make $K_{n,2}^\times$ normalize $U_2$ in $\GL_2(\BQ_2)$ at the place $2$, which is the basis of our treatment. For $n\equiv 7\pmod 8$, the embedding gives $\wh{O}_K^\times\subset U_0(32)$, which makes the easiest calculation.

Similarly, the choices of $f_n$ seem artificial and technical here. However, they are obtained by some prescribed representation-theoretical properties below.
Following \cite[\S 1.2]{YZZ},  consider the representation
  $$\pi=\Hom ^0_\infty(X, A)=\varinjlim_V \Hom^0_\infty (X_V, A)$$
of $\GL_2(\wh \BQ)$.
Here for any open compact subgroup $V$ of $\GL_2(\wh \BQ)$,
$$ \Hom^0_\infty (X_V, A)=\Hom_\infty (X_V, A)\otimes_\BZ\BQ,$$
where
$$\Hom_\infty(X_V, A)=\{f\in \Hom (X_V, A): \ f(\infty)\in A(\overline \BQ)_{\tor}\}.$$
Here $\infty$ denotes the cusp of $X_V$.

\begin{prop} \label{choice}
\begin{enumerate} [(1)]
\item
If $n\equiv 5,6\pmod 8$, the space $\pi^{\GL_2(\wh\BZ^{(2)})\cdot K_{n,2}^\times}$ is one-dimensional and contains $f_n$.
\item
If $n\equiv 7\pmod 8$,  the space
$\pi^{U_0(32)}$ is one-dimensional and contains $f_n$.
\end{enumerate}
\end{prop}

The proposition explains that $f_n$ is an explicit vector in a one-dimensional space in the framework of the multiplicity one theorem of Tunnell \cite{Tu} and Saito \cite{Sa}. See Theorem \ref{TS1} and Corollary \ref{TS2}.
One can also define an integral structure $\pi_\BZ$ of $\pi$ as the subgroup of elements of $\pi$ coming from $\Hom_\infty (X_V, A)$ for some $V$.
Then one can consider the primitivity of $f_n$ under this integral structure as in \S \ref{section 2}. However, this is too involved in the current setting, so we will only consider the behavior of $f_n$ in the rational structure $\pi$.

\subsubsection*{CM point2}
Note that we have chosen an explicit embedding of $K_n$ in $M_2(\BQ)$, which induces an action of $K_n^\times$ on the upper half plane $\CH$.
Let
$$P_n=[h, 1]\in X_U(\BC)$$
be the CM point, where $h\in \CH^{K_n^\times}$ is the unique fixed point of $K_n^\times$ in
$\CH$.
Let
$$z_n=f_n(P_n) \in A(K_n^\ab).$$

Note that $z_n$ is not necessarily defined over the Hilbert class field $H_n$ of $K_n$.
Denote by $H_n'=H_n(z_n)$ the extension of $H_n$ generated by the residue field of $z_n$.
The following result is a precise description of the field of definition of $z_n$. In the following, denote by
$$
\sigma: K_n^\times \bs \wh K_n^\times \lra \Gal(K_n^\ab/K_n)
$$
the geometric Artin map, normalized by sending the uniformizers to the geometric Frobenii.
So it is the reciprocal of the usual Artin map.

\begin{prop}  \label{CM point1}
\begin{enumerate}[(1)]
 \item Assume that $n\equiv 5\pmod 8$. Then $\Gal(H_n'/H_n)\simeq \BZ/2\BZ$ is generated by $\sigma_{\varpi}^2$. Here $\varpi=(\sqrt{-n}-1)_2\in K_{n,2}^\times$.  The field $H_n'(\sqrt{2})$ is the ring class field of conductor $4$ over $K_n$.
 The Galois group $\Gal(H_n'(\sqrt{2})/H_n)\simeq (\BZ/2\BZ)^2$ is generated by $\sigma_{1+2\varpi}$ and $\sigma_{\varpi}^2$, and $H'_n$ is the subfield of $H_n'(\sqrt{2})$ fixed by $\sigma_{1+2\varpi}$.

 \item Assume that $n\equiv 6\pmod 8$. Then $\Gal(H_n'/H_n)\simeq \BZ/4\BZ$ is generated by $\sigma_{1+\varpi}$. Here $\varpi=(\sqrt{-n})_2\in K_{n,2}^\times$.  The subfield of $H_n'$ fixed by $\sigma_{1+\varpi}^2$ is $H_n(i)$. The field $H_n'$ is exactly the ring class field of conductor $4$ over $K_n$.

\item Assume that $n\equiv 7\pmod 8$. Then $H_n'=H_n$.

\item For any $n\equiv 5,6,7\pmod 8$, $2z_n$ is defined over $H_n$.
\end{enumerate}
\end{prop}

Denote $K_n'=K_n, K_n(i), K_n$ according to $n\equiv 5, 6, 7\pmod 8$ respectively. Set
$\Cl_n=\Gal(H_n/K_n)$ and $\Cl'_n=\Gal(H_n'/K_n')$.
Let $\sigma$ be the unique order-two element of $\Gal(H_n'/H_n)$ in the case $n\equiv 5,6\pmod 8$, and set $\sigma=1$ in the case $n\equiv 7\pmod 8$.
Then the natural map $\Cl_n'\to \Cl_n$ induces two isomorphisms
$$
\Cl'_n/\langle\sigma\rangle\cong\Cl_n, \quad
(2\Cl'_n)/\langle\sigma\rangle\cong2\Cl_n.
$$
The least obvious case is the second isomorphism for $n\equiv 6\pmod 8$. For that,  it suffices to check that $\sigma=\sigma_{1+\varpi}^2$ lies in $2\Cl'_n$.
Note that $\sigma_{1+\varpi}\notin \Cl'_n$, but we use the relations
$\sigma=(\sigma_{1+\varpi}\sigma_{\varpi})^2$ and $\sigma_{1+\varpi}\sigma_{\varpi}\in \Cl'_n$ instead. In fact, an easy calculation shows $\sigma_{\varpi}^2=1$ (on $H_n'$) and $\sigma_{\varpi}(i)=-i$, which give the new relations.

\subsubsection*{Quadratic point}

Fix a set $\Phi\subset \Cl'_n$ of representatives of $\Cl'_n/\langle\sigma\rangle\cong \Cl_n$. Let $\chi:\Cl_n\to \{\pm1\}$ be a character.
Define the quadratic point $P_\chi$ associated to $\chi$ by
$$P_\chi:=\sum_{t\in \Phi} f_n(P_n)^{t} \chi(t)\in A(H_n').$$
Here $\chi$ is also viewed as a function on $\Phi$ via the bijection $\Phi\to \Cl_n$.

To give a formula for $P_\chi$, we need to describe another algebraic point on the elliptic curve.
Recall that $\CL(n)$ and  $\rho(n)$ are defined in the introduction of this paper.
We will see that $\CL(n)$ is a rational number.
Define
 $$\CP(n):=2^{-1-\rho(n)}\CL(n) \alpha_n \in A(K_n)^-\otimes_\BZ\BQ,$$
where
$$A(K_n)^-:=\{\alpha\in  A(K_n): \bar\alpha=-\alpha\}\ \subset\ A(K_n),$$
and
$\alpha_n\in A(K_n)^-$ is any point which generates the free part $A(K_n)^-/A(K_n)^-_\tor$ if $\CL(n)\neq 0$. Note that $\CP(n)=0$ if $\CL(n)=0$.

\begin{thm}[Gross-Zagier formula]\label{GZ}
Let $\chi:\Cl_n\to \{\pm1\}$ be a character.
The point $P_\chi$ is non-torsion only if $\chi$ is of the form
$$\chi_{d_0, d_1},\quad  n=d_0d_1,\ 0<d_0\equiv 5, 6, 7\pmod 8,\ 0<d_1\equiv 1, 2, 3\pmod 8, $$ where $\chi_{d_0, d_1}$ is the unique Hecke character over $K_n$ associated to the extension $K_n(\sqrt{d_1})$ for $n\equiv 5,6\pmod 8$ or $K_n(\sqrt{d_1^*})$ for $n\equiv 7\pmod 8$.
Here $d_1^*=(-1)^{(d_1-1)/2}d_1$ as before.

In that case, in the vector space
$A(H_n'(i))\otimes_{\BZ} \BQ=A(H_n'(i))\otimes_{\BZ[i]} \BQ[i]$,
$$P_{\chi}=\epsilon (d_0, d_1) 2^{h_2(n)}\CL(d_1) \CP(d_0),$$
where $\epsilon(d_0, d_1)=\pm i$ if $(d_0, d_1)\equiv (5, 3)\pmod 8$ and $\epsilon(d_0, d_1)=\pm 1$ otherwise.
\end{thm}

\subsubsection*{Genus point}

Set $\Phi_0=\Phi\cap (2\Cl_n')$ as a subset of $\Cl_n'$.
Then $\Phi_0\subset 2\Cl_n'$ is a set of representatives of $(2\Cl'_n)/\langle\sigma\rangle\cong 2\Cl_n$ in $2\Cl_n'$. Define
 $$Z(n):=\sum_{t\in \Phi_0} f_n(P_n)^{t}\in A(H_n').$$

To compare $Z(d_0)$ for different divisors $d_0$ of $n$, we introduce the composite field
$$\BH_n':=L_n(i)\cdot\prod_{\substack{d_0|n, \ d_0>0 \\ d_0\equiv 5, 6\pmod 8}} H_{d_0}'\subset \overline\BQ.$$
Here $L_n(i)=\BQ(i, \sqrt{d}:d|n)$. The field seems to be very large, but we will see that
$A(\BH_n')_\tor\subset A[4]$ in Lemma \ref{torsion2}, which is a key property in our treatment.
Note that $A(\BH_n')$ is a $\BZ[i]$-module.
Define $P(n)\in A(\BH_n')$ inductively by
$$P(n):=Z(n)-\sum_{\substack{n=d_0d_1\\ d_0\equiv 5, 6, 7\pmod 8\\
d_1\equiv  1,2 ,3 \pmod 8,\ d_1>1}} \epsilon(d_0, d_1)\CL(d_1)P(d_0),$$
where $\epsilon(d_0, d_1)\in \mu_4$ is as in Theorem \ref{GZ}.
Note that $\CL(d_1)\in \BZ$ by Theorem \ref{lg}.
By definition, it is easy to verify the following congruence formula.

\begin{prop} \label{congruence}
In $A(\BH_n')$,
\begin{multline*}
P(n) \equiv
\sum_{\substack{n=d_0d_1\cdots d_\ell\\
d_0\equiv 5,6, 7\pmod 8\\
d_1\equiv 1, 2, 3\pmod 8\\
d_i\equiv 1 \pmod 8, \ i>1}}\epsilon(d_0, d_1) \left(\prod_{i\geq 1} g(d_i)\right) Z(d_0)
\\
+i\sum_{\substack{n=d_0d_1\cdots d_\ell\\
(d_0,d_1,d_2)\equiv (5,3,2)\pmod 8\\
d_i\equiv 1 \pmod 8, \ i>2}} \left(\prod_{i\geq 1} g(d_i)\right) Z(d_0)
\qquad\mod\ 2A(\BH_n').
\end{multline*}
\end{prop}

The main result of this section is as follows, which is an enhanced version of Theorem \ref{lg'}.

\begin{thm}\label{m3}
The vector $\CP(n)\in A(K_n)^-\otimes_\BZ\BQ$ is represented by the point $P(n)\in A(\BH_n')$ in the sense that they are equal in $A(\BH_n')\otimes_\BZ\BQ$.
Moreover,
 \begin{enumerate}[(1)]
 \item The image of $2P(n)$ under any $2$-isogeny from $A$ to $E$ belongs to $E(K_n)^-$, i.e. $\CL(n)$ is integral.
 \item Assume that $P(n)\in A(K_n)^-+A[4]$, i.e. $2^{-\rho(n)}\CL(n)$ is even.
If $n\equiv 5, 7\pmod 8$, then
 $$\sum_{\substack{{n=d_0\cdots d_\ell}\\ {d_i\equiv 1\pmod 8,\ i>0}}} \prod_i
g(d_i)\equiv \sum_{\substack{{n=d_0\cdots d_\ell,}\\
{d_0\equiv 5, 6, 7\pmod 8}\\ {d_1\equiv  1,2 ,3 \pmod 8} \\ {d_i\equiv 1\pmod 8,\ i>1}}} \prod_i g(d_i)  \equiv 0 \quad
\pmod 2.$$
If $n\equiv 6\pmod 8$, then
 $$ \sum_{\substack{{n=d_0\cdots d_\ell,}\\
{d_0\equiv 5, 6, 7\pmod 8}\\ {d_1\equiv  1,2 ,3 \pmod 8} \\ {d_i\equiv 1\pmod 8,\ i>1}}} \prod_i g(d_i)  \equiv 0\quad
\pmod 2.$$
\end{enumerate}
\end{thm}

\subsection{Test vectors}

Recall that in Proposition \ref{CM point1} we have described the field $H_n'=H_n(z_n)$.
The major goal of this section is to prove some results about Galois actions on $z_n$. We will also prove Proposition \ref{choice} and Proposition \ref{CM point1}.

To describe the results about Galois actions on $z_n$, we recall some basic facts about $X_0(32)$ and $A$, which are basic facts or results proved in \cite{Tian}.
\begin{enumerate}[(1)]
\item
There is an analytic isomorphism
$$\tau: \BC/(1+i)\BZ[i]\lra A(\BC).$$
The map $\tau$ is unique up to multiplication by $\mu_4=\{\pm1, \pm i\}$.
We can adjust $\tau$ such that $\BR/2\BZ$ maps onto $A(\BR)$ and
$$A[2^\infty]\ =\ \BQ_2(i)/(1+i)\BZ_2[i]\ \subset\ \BQ(i)/(1+i)\BZ[i] \ =\ A(\BC)_{\tor}. $$
\item
Under the uniformization $\tau$, the Galois group $G_\BQ=G_{\BQ(i)}\rtimes \{1, c\}$ acts on $A[2^\infty]=\BQ_2(i)/(1+i)\BZ_2[i]$ as follows.
The induced action of $c$ on $\BQ_2(i)/(1+i)\BZ[i]$ is still given by the conjugation $i\mapsto -i$, and the induced action of $G_{\BQ(i)}$ on $\BQ_2(i)/(1+i)\BZ[i]$ is given by multiplying by the composition
$$G_{\BQ(i)}\ra \Gal(\BQ(i)^{\ab}/\BQ(i))\stackrel{\sigma_{\BQ(i)}^{-1}}{\lra} \BQ(i)^\times \bs \wh{\BQ(i)}^\times \cong (1+(1+i)^3\wh{\BZ[i]})^\times\ra (1+(1+i)^3\BZ_2[i])^\times,$$
where $\sigma_{\BQ(i)}$ denotes the Artin map and the last map is the natural projection.
\item
The identification $i_0:X_0(32)\ra A$ (mapping $\infty$ to $0$) identifies the set
$\CS=\Gamma_0(32)\bs \BP^1(\BQ)$ of cusps with $A[(1+i)^3]=A(\BQ(i))$.
Replacing $\tau$ by $-\tau$ if necessary, we can (and we will) assume that the induced bijection
$$\tau: \frac12 \BZ[i]/(1+i)\BZ[i] \lra A[(1+i)^3]= \Gamma_0(32)\bs \BP^1(\BQ)$$
gives
$$\tau(0)=[\infty], \ \quad \tau(1/2)=[0], \ \quad \tau(-1/2)=[1/2],$$
$$\tau(1)=[1/16], \ \quad \tau(\pm i/2)=[\pm 1/4], \ \quad \tau((1\pm i)/2)=[\pm 1/8].$$
\end{enumerate}

Now we are ready to state the main result of this subsection.

\begin{thm}\label{CM point2}
Resume the notations in Proposition \ref{CM point1}. The following are true:
\begin{enumerate}[(1)]
 \item Assume that $n\equiv 5\pmod 8$. Then
$$ z_n^{\sigma_\varpi}=z_n+\tau(\frac{1+i}2), \qquad
z_n^{\sigma_{1+2\varpi}}=z_n, \qquad \bar z_n=-z_n+\tau(1).$$
Thus $z_n^{\sigma_{\varpi^2/2}}=z_n^{\sigma_{\varpi^2}}=z_n+\tau(1)$.

 \item Assume that $n\equiv 6\pmod 8$. Then
$$z_n^{\sigma_{\varpi}}=z_n+\tau(-i/2),\qquad z_n^{\sigma_{1+\varpi}}=z_n+\tau(\frac{1-i}2), \qquad \bar z_n=-z_n.$$
Thus $z_n^{\sigma_{1+\varpi}^2}=z_n+\tau(1)$.

\item Assume that $n\equiv 7\pmod 8$.
Let $v_2$ and $v_2'$ be the two places of $K_n$ above $2$ such that $v_2(\sqrt{-n}-\delta)\geq 6$. Let $\varpi\in K_{n,2}$ be an element with $v_2(\varpi)=1$ and $v_2'(\varpi)=0$.
Then
$$\bar z_n+z_n^{\sigma_{\varpi^5}}=\tau(1/2).$$
 \end{enumerate}
\end{thm}

Here $\bar z_n$ denotes the complex conjugate of $z_n$.
The results will be treated case by case in the following.
For simplicity, we write $K$ for $K_n$ (so that $K_2$ means the local field $K_{n,2}$ of $K_n$ at $2$).

\

\subsection*{Case $n\equiv 5\pmod 8$}

In this case, $f_n:X_U\to A$ is given by $f_n=f_0-f_0\circ [i]$, and the embedding of $K$ into $M_2(\BQ)$ is given by
$$\sqrt{-n} \longmapsto \matrixx{-1}{1/4}{-4(n+1)}{1}.$$
The embedding gives
$(\wh{\BZ}+4\wh{O}_{K})^\times \subset U$.

\begin{lem} \label{coset5}
Assume  $n\equiv 5\pmod 8$.
\begin{enumerate}[(1)]
\item The quotient
$K_{2}^\times/\BQ_2^\times (1+4 O_{K,2})$
is isomorphic to $\BZ/4\BZ\times \BZ/2\BZ$, and generated by the order-four element
$\varpi=(\sqrt{-n}-1)_2$ and the order-two element $1+2\varpi$.
\item
The multiplicative group $K_2^\times$ normalizes $U_2$.
\end{enumerate}
\end{lem}
\begin{proof}
We first check (1). Note that $K_2$ is ramified over $\BQ_2$.
Then $\BQ_2^\times O_{K,2}^\times$ has index two in $K_{2}^\times$.
Then $K_{2}^\times/\BQ_2^\times (1+4 O_{K,2})$ has an index-two subgroup
$$\BQ_2^\times O_{K,2}^\times /\BQ_2^\times (1+4 O_{K,2})
=O_{K,2}^\times /\BZ_2^\times (1+4 O_{K,2})
=(O_{K,2}/4 O_{K,2})^\times /\{\pm 1\}
\simeq \BZ/2\BZ\times \BZ/2\BZ.
$$
It follows that $K_{2}^\times/\BQ_2^\times (1+4 O_{K,2})$ is isomorphic to $\BZ/4\BZ\times \BZ/2\BZ$.
Now it is easy to check that $\varpi$ and $1+2\varpi$ generate the group.

For (2), since $1+4 O_{K,2} \subset U$, we see that $\BQ_2^\times (1+4 O_{K,2})$ normalizes $U_2$. By (1),  it suffices to check that $\varpi$ and $1+2\varpi$ normalize $U_2$, which can be done by explicit calculations.
\end{proof}

By the lemma, $K_2^\times$ normalizes $U_2$, and thus it acts on $X_U$ by the right multiplication. The subgroup $\BQ_2^\times(1+4 O_{K_n,2})$ acts trivially and induces a homomorphism
$$K_2^\times/\BQ_2^\times
(1+4 O_{K_n,2})\lra \Aut_\BQ(X_U).$$
We will describe this homomorphism explicitly.

The following result contains a lot of identities in
$$\Aut _\BQ(X_U)=\Aut_{\BQ(i)}(X_0(32)_{\BQ(i)})\rtimes \{1,
\epsilon\}\simeq \left( A(\BQ(i))\rtimes \mu _4\right) \rtimes \{1,
\epsilon\}.$$
Here $\epsilon=\matrixx{1}{}{}{-1}\in \GL_2(\BQ)$ also normalizes $U$, and its Hecke action on $X_U$ gives the non-trivial element of $\Gal(X_U/X_0(32))$. In particular, $\epsilon$ acts on $\Aut_{\BQ(i)}(X_0(32)_{\BQ(i)})$ by sending $i$ to $-i$.
Recall that we have also identified
$$
A(\BQ(i))= \frac12 \BZ[i]/(1+i)\BZ[i]
$$
with the set $\CS$ of cusps of $X_0(32)$.

\begin{prop}\label{heckeaction5}
Assume  $n\equiv 5\pmod 8$.
\begin{enumerate}[(1)]
\item For any
$Q\in X_U(\BC)$,
$$Q^{\varpi}
=[i] Q^\epsilon+\tau(\frac12), \quad\
Q^{1+2\varpi}=Q+\tau(1).$$
\item  The order-two element $j=\matrixx{1}{0}{8}{-1} \in
\GL_2(\BQ)$ normalizes $K^\times$ such that $jxj=\ov{x}$ for all
$x\in K^\times$ and normalizes $U$ with the  induced  action on
$X_U$ given by
$$Q^{j}=[-i] Q^\epsilon+\tau(\frac {1+i}2), \quad \forall\ Q\in X_U(\BC).$$
\end{enumerate}
\end{prop}

\begin{proof}
The right translation by an element $g\in \GL_2(\BA_f)$ switches the two geometric
components of $X_U$ if and only if the image of $g$ under the
composition
$$\GL_2(\wh{\BQ})\stackrel{\det}{\lra}
\wh{\BQ}^\times=\BQ^\times_+\cdot \wh{\BZ}^\times \lra
(\BZ_2/4\BZ_2)^\times \cong \{\pm 1\}$$ is $-1$. For example, all
$\epsilon=\matrixx{1}{0}{0}{-1},\  j=\matrixx{1}{0}{8}{-1}$,\ and
$\varpi=\matrixx{-2}{1/4}{-4(n+1)}{0}$ are such elements,  but
$1+2\varpi$ is not.

Hence, the actions of $\varpi \epsilon$,
$1+2\varpi$ and $j\epsilon$ take the form
$$Q^{\varpi \epsilon}=\alpha Q+R, \quad Q^{1+2\varpi}=\beta Q +S,
\quad Q^{j\epsilon}=\gamma Q+T$$ where $\alpha, \beta, \gamma\in
\mu_4$ and $R, S, T\in\CS$ are cusps of $X_0(32)$. Here the right sides belong to
$$\Aut_{\BQ(i)}(X_0(32)_{\BQ(i)})=\Aut(X_U/\BQ(i))\subset
\Aut_\BQ(X_U).$$

To compute $R, S, T$, we take $Q=[\infty]$. In terms of the complex uniformization
$$X_0(32)(\BC)=\GL_2(\BQ)_+ \bs (\CH\cup \BP^1(\BQ)) \times
\GL_2(\wh{\BQ})/U_0(32),$$
we have
$$
R=[\infty, \varpi \epsilon], \quad
S=[\infty, 1+2\varpi], \quad
T=[\infty, j \epsilon].
$$
We need to convert them to expressions of the form $[\theta]=[\theta,1]$ with $\theta\in  \BP^1(\BQ)$.
By the complex uniformization,
$$\begin{aligned}
\CS&=\GL_2(\BQ)_+\bs \BP^1(\BQ)\times
\GL_2(\wh{\BQ})/U_0(32)=P(\BQ)_+\bs \GL_2(\wh{\BQ})/U_0(32)\\
&=P(\BQ)_+\bs P(\wh{\BQ})\cdot \GL_2(\wh{\BZ})/U_0(32)=N(\BZ_2)\bs
\GL_2(\BZ_2)/U_0(32)_2. \end{aligned}$$
For the truth of the last identity, we refer to \cite[Lemma 4.12(2)]{YZZ}.
We have decompositions in
$\GL_2(\BQ_2)$ as follows:
$$\begin{aligned}
\varpi\epsilon&=\matrixx{1/4}{0}{0}{8}\matrixx{-8}{-1}{1}{0}\matrixx
{-(n+1)/2}{0}{4(n+3)}{1},\\
1+2\varpi&=\matrixx{1}{1/2}{0}{1}\matrixx{1}{0}{-16}{1}\matrixx{4n+1}{0}{8(7n+1)}{1}, \\
j\epsilon&=\matrixx{1}{0}{8}{1}.
\end{aligned}$$
It follows that
$$
R=\left[\infty, \matrixx{-8}{-1}{1}{0}_2\right]=\left[\matrixx{-8}{-1}{1}{0}^{-1}\infty, 1\right]=[0]=\tau(1/2).
$$
Similarly, $S=[1/16]=\tau(1)$ and $T=[-1/8]=\tau((1-i)/2)$.

To find $\alpha, \beta, \gamma$, we only need
check the action on the cusp $[0]=\left[\infty,\matrixx{0}{1}{-1}{0}\right]$. We have  decompositions:
$$\begin{aligned}
\matrixx{0}{1}{-1}{0}\varpi\epsilon
&=\matrixx{8}{0}{0}{1/4}\matrixx{1}{0}{8}{1}\matrixx{(n+1)/2}{0}{-4(n+3)}{-1},\\
\matrixx{0}{1}{-1}{0}(1+2\varpi)&=\matrixx{-2}{+1}{0}{-1/2}\matrixx{7}{-3}{-2}{1}
\matrixx{4n-17}{3}{8(n-5)}{7},\\
\matrixx{0}{1}{-1}{0}j\epsilon&=\matrixx{8}{1}{-1}{0}.
\end{aligned}$$
It follows that
$$[0]^{\varpi\epsilon}=[-1/8]=\tau((1-i)/2), \quad
[0]^{1+2\varpi}=[1/2]=\tau(-1/2), \quad
[0]^{j\epsilon}=[0]=\tau(1/2).$$
We then have the equations
$$\tau((1-i)/2)=\alpha\tau(1/2)+\tau(1/2), \quad \tau(-1/2)=\beta
\tau(1/2)+\tau(1), \quad \tau(1/2)=\gamma\tau(1/2)+\tau((1-i)/2),$$
which give $\alpha=-i$, $\beta=1$ and $\gamma=i$.

Therefore, we have
$$Q^{\varpi\epsilon}=[-i]Q+\tau(1/2), \quad\
Q^{1+2\varpi}=Q+\tau(1), \quad
Q^{j\epsilon}=[i]Q+\tau(\frac {1-i}2).$$
For the first and the the third equations, we take a further $\epsilon$-action on both sides. Then
$$
Q^{\varpi}=([-i]Q+\tau(1/2))^\epsilon
=[i](Q^\epsilon)+\tau(1/2)
$$
and
$$
Q^{j}=([i]Q+\tau(\frac {1-i}2))^\epsilon=[-i](Q^\epsilon)+\tau(\frac {1+i}2).$$
It finishes the proof.
\end{proof}

The map $K_2^\times \to \Aut(X_U)$ induces an action of $K_2^\times$ on $\Hom(X_U,A)$.
Still use $\pi$ denote this action. Now it is easy to have the action on $f_n=f_0\circ [1-i].$

\begin{cor} \label{integral5}
In $\Hom(X_U,A)$,
$$\pi(\varpi)f_n=f_n+\tau(\frac{1+i}2), \quad\ \pi(1+2\varpi)f_n=f_n.$$
\end{cor}

\begin{proof}
For any $Q\in X_U(\BC)$,
$$
(\pi(\varpi)f_n)(Q)= f_0([1-i] Q^\varpi).
$$
By the proposition,
$$
[1-i] Q^\varpi=[1-i]([i]Q^\epsilon+\tau(1/2)) =[1+i]Q^\epsilon+\tau(\frac{1-i}2)
=( [1-i]Q+\tau(\frac{1+i}2) )^\epsilon.$$
Note that $f_0$ is invariant under $\epsilon.$
Thus
$$
(\pi(\varpi)f_n)(Q)= f_0([1-i]Q+\tau(\frac{1+i}2))
= f_0([1-i]Q)+\tau(\frac{1+i}2)
= f_n(Q)+\tau(\frac{1+i}2).
$$
The second equality is proved similarly.
\end{proof}

The corollary is an integral version of Lemma \ref{choice} for $n\equiv5\pmod8$.
Now $f_n$ lies in the space $\pi^{\GL_2(\wh\BZ^{(2)})\cdot K_{2}^\times}\simeq \pi_2^{K_{2}^\times}$, which is one-dimensional by Theorem \ref{TS1} and Theorem \ref{TS2}.

Now we are ready to prove Theorem \ref{CM point2} for $n\equiv5\pmod8$, i.e.,
$$ z_n^{\sigma_\varpi}=z_n+\tau(\frac{1+i}2), \qquad
z_n^{\sigma_{1+2\varpi}}=z_n, \qquad \bar z_n=-z_n+\tau(1).$$
For the first two equalities, the key is that the Galois action of $K_2^\times$ on $P$ (via the Artin map $\sigma$) is the same as the Hecke action of $K_2^\times$, by the special form of
$P_n=[h,1]$.
Then by Corollary \ref{integral5},
$$
z_n^{\sigma_\varpi}=f_n(P_n^{\sigma_\varpi})
=f_n(P_n^{\varpi})
=(\pi(\varpi)f_n)(P_n)
=f_n(P_n)+\tau(\frac{1+i}2)
=z_n+\tau(\frac{1+i}2).
$$
The second equality is similarly obtained.

For the third equality, the Hecke action of the element $j$ in Proposition \ref{heckeaction5} (2) gives the complex conjugation of $P_n=[h,1]$ by the condition $jxj^{-1}=\ov{x}$ for all
$x\in K^\times$.
In fact,
$$
\bar P_n=[\bar h,1], \quad P_n^j=[h,j]=[j(h),1].
$$
It suffices to check $\bar h=j(h)$. Note that $\{h, \bar h\}$ is the set of fixed points of $K^\times$ in $\CH^\pm$. By $jK^\times j=K^\times$, we see that $\{h, \bar h\}=\{j(h), j(\bar h)\}$ as sets. Since $\det(j)=-1<0$, we have $j(h)\in \CH^-$ and thus $j(h)=\bar h$.

Hence,
$$
\bar z_n=f_n(\bar P_n)
=f_n(P_n^{j})
=f_n([-i] P_n^\epsilon+\tau(\frac {1+i}2))
=f_0([-1-i] P_n^\epsilon+\tau(1)).
$$
By $[-1-i] P_n^\epsilon+\tau(1)=([-1+i] P_n^\epsilon+\tau(1))^\epsilon$, we have
$$
\bar z_n
=f_0([-1+i] P_n+\tau(1))
=-f_0(P_n)+\tau(1)
=-z_n+\tau(1).
$$
This proves the theorem in the current case.

Finally, we prove Proposition \ref{CM point1} for $n\equiv 5\pmod 8$.
By the reciprocity law, the point $P_n$ is defined over the abelian extension of $K$ with Galois group $\wh K^\times/ K^\times (\wh K^\times\cap U)$.
It is easy to see $(\wh{\BZ}+4\wh{O}_{K})^\times \subset U$.
Then $P_n$ is defined over the ring class field $H_{n,4}$ of $K$
with Galois group $\wh K^\times/ K^\times (\wh{\BZ}+4\wh{O}_{K})^\times$.
We have
\begin{multline*}
\Gal(H_{n,4}/H_n)
\cong K^\times \wh{O}_{K}^\times/
K^\times (\wh{\BZ}+4\wh{O}_{K})^\times
= \wh{O}_{K}^\times/ (\wh{\BZ}+4\wh{O}_{K})^\times \\
= O_{K,2}^\times/ (\BZ_2+4 O_{K,2})^\times
= O_{K,2}^\times/ \BZ_2^\times(1+4 O_{K,2})
= (O_{K,2}/4 O_{K,2})^\times/ \{\pm1\}.
\end{multline*}
As in the proof of Lemma \ref{coset5}, the right-hand side is isomorphic to $\BZ/2\BZ\times \BZ/2\BZ$, and generated by $\frac12\varpi^2$ and $1+2\varpi$.
Consider $z_n=f_n(P_n)$ and $H_n'=H_n(z_n) \subset H_{n,4}$.
By Theorem \ref{CM point2},
$$z_n^{\sigma_{\varpi^2/2}}=z_n+\tau(1), \qquad
z_n^{\sigma_{1+2\varpi}}=z_n.$$
It follows that $H_n'$ is the index-two subfield of $H_{n,4}$ fixed by $\sigma_{1+2\varpi}$.
Note that $\sqrt2\in H_{n,4}$ but $\sqrt2\notin H_{n}'$ by
$\sigma_{1+2\varpi}(\sqrt2)=-\sqrt2$.
The equations also indicate that $2z_n$ is invariant under both $\sigma_{\varpi^2/2}$ and $\sigma_{1+2\varpi}$, and thus it is defined over $H_n$.
The proposition is proved in this case.

\

\subsection*{Case $n\equiv 6\pmod 8$}

Now we consider the case $n\equiv 6\pmod 8$.
The exposition is very similar to the previous case $n\equiv 5\pmod 8$, and the calculations are slightly simpler.
We still follow the process of the previous case, but only sketch some of the proofs.

In this case, $f_n:X_U\to A$ is given by $f_n=f_0\circ [i]$, and the embedding of $K$ into $M_2(\BQ)$ is given by
$$\sqrt{-n}\lto {\matrixx{}{1/4}{-4n}{}}.$$
The embedding still gives
$(\wh{\BZ}+4\wh{O}_{K})^\times \subset U$.

\begin{lem} \label{coset6}
Assume  $n\equiv 6\pmod 8$.
\begin{enumerate}[(1)]
\item The quotient
$K_{2}^\times/\BQ_2^\times (1+4 O_{K,2})$
is isomorphic to $\BZ/4\BZ\times \BZ/2\BZ$, and generated by the order-two element
$\varpi=(\sqrt{-n})_2$ and the order-four element $1+\varpi$.
\item
The multiplicative group $K_2^\times$ normalizes $U_2$.
\end{enumerate}
\end{lem}
\begin{proof}
The proof is similar to that of Lemma \ref{coset5}.
\end{proof}

Now we describe the homomorphism
$$K_2^\times/\BQ_2^\times
(1+4 O_{K_n,2})\lra \Aut_\BQ(X_U).$$

\begin{prop}\label{heckeaction6}
Assume  $n\equiv 6\pmod 8$.
\begin{enumerate}[(1)]
\item For any
$Q\in X_U(\BC)$,
$$Q^{\varpi}=-Q^\epsilon+\tau(\frac12), \quad Q^{1+\varpi}=-Q^\epsilon+\tau(\frac{1-i}2).$$
\item  The order-two element $\epsilon=\matrixx{1}{}{}{-1}\in
\GL_2(\BQ)$ normalizes $K^\times$ such that $\epsilon x \epsilon=\ov{x}$ for all
$x\in K^\times$.
\end{enumerate}
\end{prop}

\begin{proof}
Follow the strategy of Proposition \ref{heckeaction5}.
The Hecke operators
$\varpi\epsilon $ and $(1+\varpi)\epsilon$ do not switch the two geometric
components of $X_U$.
We have the decompositions
$$\begin{aligned}
\varpi\epsilon&=\matrixx{1/4}{0}{0}{8}\matrixx{0}{1}{-1}{0}\matrixx{n/2}{0}{0}{-1},\\
(1+\varpi)\epsilon&=\matrixx{1}{1/4}{0}{1}\matrixx{-1}{0}{8}{-1}\matrixx{-(1+n)}
{0}{-8(1+n/2)}{1}.\end{aligned}$$
It follows that $\varpi\epsilon $ and $(1+\varpi)\epsilon$ maps $[\infty]$ to $[0]$ and
$[1/8]$, respectively. Thus they acts on $X_U$ take form
$$Q^{\varpi\epsilon}=\alpha Q+\tau(1/2), \quad Q^{(1+\varpi)\epsilon}=\beta Q
+\tau((1+i)/2)$$for some $\alpha, \beta\in \mu_4$.
Use the decompositions
$$\begin{aligned}
\matrixx{0}{1}{-1}{0}\varpi\epsilon &
=\matrixx{-8}{0}{0}{-1/4}\matrixx{n/2}{0}{0}{-1},\\
\matrixx{0}{1}{-1}{0}(1+\varpi)\epsilon&
=\matrixx{4}{-1}{0}{1/4}\matrixx{1}{0}{-4}{1}\matrixx{-(1+n)}{0}{-8(1+n/2)}{1}.
\end{aligned}$$
Setting $Q=[0]=\tau(1/2)$, we have the equations
$$\tau(0)=\alpha\tau(1/2)+\tau(1/2), \quad
\tau(i/2)=\beta\tau(1/2)+\tau((1+i)/2).$$ It follows that $\alpha=-1$ and
$\beta=-1$. Hence, we have
$$Q^{\varpi\epsilon}=-Q+\tau(\frac12), \quad Q^{(1+\varpi)\epsilon}=-Q+\tau(\frac{1+i}2).$$
Further actions by $\epsilon$ gives the results.
\end{proof}

We have the following integral version of Lemma \ref{choice} for $n\equiv5\pmod8$.

\begin{cor} \label{integral6}
Assume  $n\equiv 6\pmod 8$.
In $\Hom(X_U,A)$,
$$\pi(\varpi)f_n=f_n+\tau(-\frac{i}2), \quad\ \pi(1+\varpi)f_n=f_n+\tau (\frac {1-i}2).$$
\end{cor}

\begin{proof}
The proof is similar to that of Corollary \ref{integral5}.
\end{proof}

Now we can prove Theorem \ref{CM point2} and Proposition \ref{CM point1} for $n\equiv6\pmod8$ similarly.
For example, the proof of $\bar z_n=-z_n$ is given by:
$$
\bar z_n=f_n(\bar P_n)
=f_n(P_n^{\epsilon})
=f_0([i] P_n^\epsilon)
=f_0(([-i] P_n)^\epsilon)
=f_0([-i] P_n)
=-z_n.
$$

\

\subsection*{Case $n\equiv 7\pmod 8$}

Now we consider the case $n\equiv 7\pmod 8$. Then $2$ is split over $K$. It is the simplest case since $f_n:X_0(32)\to A$ is just the identity map $i_0:X_0(32)\to A$. The theory does not involve the more complicated curve $X_U$. For example, Proposition \ref{choice} is true in this case since $\dim \pi^{U_0(32)}=1$ by the newform theory.

The embedding of $K$ into $M_2(\BQ)$ is given by
$$\sqrt{-n}\longmapsto
{\matrixx{\delta}{2}{-(n+\delta^2)/2}{-\delta}},$$
where $\delta\in \BZ$ satisfies $\delta^2\equiv -n\pmod {128}$.
It is easy to check that the embedding gives $\wh{O}_K^\times\subset U_0(32)$.
Then Proposition \ref{CM point1} is automatic in this case.

The following is devoted to prove Theorem \ref{CM point2} in this case.
Recall from the theorem that $v_2$ and $v_2'$ are the two places of $K_n$ above $2$ such that $v_2(\sqrt{-n}-\delta)\geq 6$, and that $\varpi\in K_{n,2}$ is an element with $v_2(\varpi)=1$ and $v_2'(\varpi)=0$.

\begin{prop}\label{prop3.9}
Assume  $n\equiv 7\pmod 8$.
Let
$$W=\matrixx{0}{1}{-32}{0}, \qquad
j=\matrixx{1}{}{-\delta}{-1}$$
be elements of $\GL_2(\BQ)$.
Then
\begin{enumerate}[(1)]
\item
The element $W$ normalizes $U_0(32)$, and
$$Q^W=-Q+\tau (1/2), \quad\ \forall\ Q\in X_0(32)(\BC).$$
\item
One has $j^2=1$, $jxj=\bar x$ for any $x\in K$, and $j\varpi^5\in W\cdot U_0(32)_2$.
Therefore,
$$Q^{j\varpi^5}=-Q+\tau (1/2), \quad\ \forall\ Q\in X_0(32)(\BC).$$
\end{enumerate}
\end{prop}

\begin{proof}
For (1), consider the Atkin--Lehner operator $\pi(W)$.
Note that $\pi(W)f_n=-f_n$ in the $\BQ$-space $\pi$ since $A$ has root number 1.
It follows that, in the $\BZ$-module $\Hom(X_0(32),A)$, the sum
 $\pi(W)f_n+f_n$ is a torsion point of $A$.
To figure out the torsion point, evaluate at $[\infty]$. We have
$$(\pi(W)f_n+f_n)([\infty])=f_n([\infty]^W)+f_n([\infty])=\tau (1/2).$$
This proves (1).

For (2), consider the $\BQ_2$-algebra $K_2\cong K_{v_0}\times K_{v_0'}=\BQ_2\times \BQ_2$.
Let $\alpha\in \BZ_2^\times$ be such that
$$(\alpha, -\alpha)=\sqrt{-n}\lto
\matrixx{\delta}{2}{-(n+\delta^2)/2}{-\delta}.$$ Then
$64|(\alpha-\delta)$ and $2\| (\alpha+\delta)$. We may take
$\varpi=(2, 1)$. Then
$$\varpi^5
=(32, 1)= \frac{31}{2\alpha} (\alpha,-\alpha)+ \frac{33}{2} (1,1)
$$
corresponds to the matrix
$$
\frac{31}{2\alpha} \matrixx{\delta}{2}{-(n+\delta^2)/2}{-\delta}
+\frac{33}{2} \matrixx{1}{}{}{1}
=\frac{1}{2\alpha} \matrixx{31\delta+33\alpha}{62}{-31(n+\delta^2)/2}{-31\delta+33\alpha}.
$$
It is now straight forward to check
$W^{-1} j\varpi^5\in U_0(32)$.
\end{proof}

Now it is easy to obtain Theorem \ref{CM point2} for $n\equiv 7\pmod 8$ which asserts
$$\bar z_n+z_n^{\sigma_{\varpi^5}}=\tau(1/2).$$
In fact, the proposition gives
$$P_n^{j\varpi^5}=-P_n+\tau (1/2).$$
As before, $j$ computes the complex conjugate of $P_n$. Then the above becomes
$$(\bar P_n)^{\varpi^5}=-P_n+\tau (1/2).$$
The Hecke action is defined over $\BQ$ and thus commutes with the complex conjugation.
This finishes the proof.

\subsection{Proofs of Theorems \ref{GZ}, Proposition \ref{congruence} and Theorem \ref{m3}}

In this section, we prove our main theorems in the case $n\equiv 5, 6, 7\pmod 8$.

\subsubsection*{Proof of Theorem \ref{GZ}}

We first prove the following result, which gives the first statement of the theorem.
\begin{lem}\label{char'}
Let $\chi:\Cl_n\to \{\pm1\}$ be a character satisfying the following conditions:
 \begin{enumerate} [(1)]
 \item The root number of $L(A_{K_n},\chi,s)$ is $-1$;
 \item If $2$ is not split in $K_n$, then the $2$-component $\chi_2:K_{n,2}^\times \to \{\pm1\}$ of $\chi$ is trivial.
\end{enumerate}
Then $\chi$ is exactly of the form
$$\chi_{d_0, d_1},\qquad  n=d_0d_1,\ 0<d_0\equiv 5, 6, 7\pmod 8,\ 0<d_1\equiv 1, 2, 3\pmod 8, $$ where $\chi_{d_0, d_1}$ is the unique Hecke character over $K_n$ associated to the extension $K_n(\sqrt{d_1})$ for $n\equiv 5,6\pmod 8$ and $K_n(\sqrt{d_1^*})$ for $n\equiv 7\pmod 8$.
\end{lem}

\begin{proof}
The character $\chi$ corresponds to an extension over $K_n$ of degree dividing $2$ and inside the genus field $L_n$ of $K_n$, which must be of form $K_n(\sqrt{d})=K_n(\sqrt{-n/d})$ for some integer $d|n$ with $\sqrt{d}\in L_n$.

First, the L-function $L(A_{K_n},\chi,s)=L(A_d,s)L(A_{n/d},s)$ has root number $-1$ if and only if exactly one element of $\{|d|, \ n/|d|\}$ is congruent to $1, 2, 3$ modulo $8$ and the other one is congruent to $5, 6, 7$ modulo $8$.  Thus we may assume that the extension corresponding to $\chi$ is of the form $K_n(\sqrt{\pm d_1})\subset L_n$, $0<d_1\equiv 1, 2, 3\pmod 8$, such that $d_0:=n/d_1\equiv 5, 6, 7\pmod 8$.

If $n\equiv 7\pmod 8$,  then $2$ is split in $K$ and the second condition is empty.  Note that $\sqrt{d_1^*}\in L_n$ but $\sqrt{-d_1^*}\notin L_n$. Thus $\chi$ is exactly of desired form.

If $n\equiv 5\pmod 8$,  then $2$ is ramified in $K_n$. Both $\sqrt{d_1}$ and
$\sqrt{-d_1}$ are in $L_n$. We have $(d_0,d_1)\equiv (5,1), (7,3) \pmod8$.
The restriction that $\chi_2$ is trivial implies that the extension corresponding to $\chi$ is $K_n(\sqrt{d_1})$ in the first case or $K_n(\sqrt{-d_0})$ in the second case. Then $K_n(\sqrt{d_1})$ is a uniform way to write down the field.

If $n\equiv 6\pmod 8$,  then $2$ is ramified in $K_n$.
We have $(d_0,d_1)\equiv (6,1), (7,2) \pmod8$.
Then the extension corresponding to $\chi$ is $K_n(\sqrt{d_1})$ in the first case or $K_n(\sqrt{-d_0})$ in the second case, and $K_n(\sqrt{d_1})$ is still a uniform way.
\end{proof}

Now we prove Theorem \ref{GZ}. It is an example of Theorem \ref{EGZ},
an explicit version of the Gross--Zagier formula proved by Yuan--Zhang--Zhang \cite{YZZ}. Recall that
$$
P_\chi= \sum_{t\in \Phi} f_n(P_n)^t \chi(t).
$$
The summation on $\Phi$ is not canonical, so the expression is not the exact case to apply the formula.
However, by Proposition \ref{CM point1}, $2z_n=2f_n(P_n)$ is defined over $H_n$ and thus
$$
2P_\chi= \sum_{t\in \Phi} (2f_n(P_n))^t \chi(t)
= \sum_{t\in \Cl_n} (2f_n(P_n))^t \chi(t).
$$
This is the situation to apply the Gross--Zagier formula (to the test vector $2f_n$).

First, we see that the point $P_\chi$ is non-torsion only if $\chi$ satisfies the two conditions of Lemma \ref{char'}. The first condition holds by considering the Tunnell--Saito theorem (cf. Theorem \ref{TS1}). See the remarks after \cite[Theorem 1.2]{YZZ} for example.
For the second condition, assume that $2$ is not split in $K=K_n$. By $\Cl_n=K^\times\bs \wh K^\times/\wh O_K^\times$, the summation for $2P_\chi$ is essentially an integration on $K^\times\bs \wh K^\times$.
By Proposition \ref{choice}, $f_n$ is invariant under the action of $K_2^\times$ up to torsions, so the integration is non-torsion only if
$\chi$ is trivial on $K_{2}^\times$.

Hence, Lemma \ref{char'} implies the first statement of the theorem.
Next, assume $\chi=\chi_{d_0, d_1}$ as in the theorem.
Denote $P(d_0,d_1)=P_{\chi_{d_0,d_1}}$.
We first have the following basic result.

\begin{lem} \label{multiples1}
If $(d_0, d_1)\equiv (5, 3)\pmod 8$, then $4P(d_0, d_1)\in A(\BQ(\sqrt{d_0}))^-=[i] A(\BQ(\sqrt{-d_0}))^-$.
Otherwise, $4P(d_0, d_1)\in A(\BQ(\sqrt{-d_0}))^-.$
\end{lem}
\begin{proof}
Recall that
$$
2P(d_0, d_1)= \sum_{t\in \Cl_n} (2z_n)^t \chi_{d_0, d_1}(t).
$$
Then $2P(d_0, d_1)$ is invariant under the action of
$\ker(\chi_{d_0,d_1})=\Gal(H_n/K_{d_0,d_1})$.
Then $2P(d_0, d_1)$ is defined over $K_{d_0,d_1}$.
Here $K_{d_0,d_1}=K_n(\sqrt{-d_1})$ if $(d_0, d_1)\equiv (5, 3)\pmod 8$, and $K_{d_0,d_1}=K_n(\sqrt{d_1})$ otherwise.

First, assume that $d_1\neq 1$, so that $[K_{d_0,d_1}:\BQ]=4$.
Consider the action of $\Gal(K_{d_0,d_1}/\BQ)$ on $2P(d_0, d_1)$.
The group $\Gal(K_{d_0,d_1}/\BQ)$ has two explicit elements: the complex conjugation $c$ and the unique nontrivial element
$\tau$ of $\Gal(K_{d_0,d_1}/K_n)$.
By definition, $\chi$ takes $-1$ on any lifting of $\tau$ in $\Cl_n$. It follows that
$$
(2P(d_0, d_1))^\tau= -2P(d_0, d_1).
$$
On the other hand, the complex conjugate
$$
(2P(d_0, d_1))^c= \sum_{t\in \Cl_n} (2\bar z_n)^{1/t}
= \sum_{t\in \Cl_n} (2\bar z_n)^{t} .
$$
By Theorem \ref{CM point2}, if $n\equiv5, 6\pmod 8$, then $2\bar z_n=-2z_n$. It follows that $(2P(d_0, d_1))^c=-2P(d_0, d_1)$.
If $n\equiv7\pmod 8$, we only have
$2\bar z_n=-2z_n^{\sigma_{\varpi^5}}+\tau(1)$, which gives
$$(2P(d_0, d_1))^c=-2\chi_{d_0, d_1}(\sigma_\varpi) P(d_0, d_1)+ |\Cl_n|\tau(1).$$
Here
$$\chi_{d_0, d_1}(\sigma_\varpi)=-1
\ \Longleftrightarrow \
\sigma_\varpi(\sqrt{d_1^*})=-\sqrt{d_1^*}
\ \Longleftrightarrow \
(d_0, d_1)\equiv (5, 3)\pmod 8.$$

In summary, if $(d_0, d_1)\not\equiv (5, 3)\pmod 8$ (and $d_1\neq1$), then
$$
(4P(d_0, d_1))^\tau= -4P(d_0, d_1), \quad
(4P(d_0, d_1))^c= -4P(d_0, d_1).
$$
It follows that $4P(d_0, d_1)$ is invariant under $c\tau$, and thus defined over
$$
K_{d_0,d_1}^{c\tau}=\BQ(\sqrt{-d_0}, \sqrt{d_1})^{c\tau}=\BQ(\sqrt{-d_0}).
$$
The action of $\tau$ further gives $4P(d_0, d_1)\in A(\BQ(\sqrt{-d_0}))^-$.
If $(d_0, d_1)\equiv (5, 3)\pmod 8$, then
$$
(4P(d_0, d_1))^\tau= -4P(d_0, d_1), \quad
(4P(d_0, d_1))^c= 4P(d_0, d_1).
$$
It follows that $4P(d_0, d_1)$ is invariant under $c$, and thus defined over
$$
K_{d_0,d_1}^{c}=\BQ(\sqrt{d_0}, \sqrt{-d_1})^{c}=\BQ(\sqrt{d_0}).
$$
The action of $\tau$ further gives $4P(d_0, d_1)\in A(\BQ(\sqrt{d_0}))^-$.

In the last case $d_1=1$, we have $K_{d_0,d_1}=K_n$. Then $(4P(d_0, d_1))^c= 4P(d_0, d_1)$ implies $4P(d_0, d_1)\in A(\BQ(\sqrt{-d_0}))^-$.
\end{proof}

Go back to the proof of the theorem. Now we are ready to prove the formula
$$P_{\chi}=\epsilon (d_0, d_1) 2^{h_2(n)}\CL(d_1) \CP(d_0) \ \in\ A(H_n'(i))\otimes_\BZ\BQ.$$
The formula is equivalent to
$$
2 \bar\epsilon (d_0, d_1) P_{\chi}= 2^{h_2(n)+1}\CL(d_1) \CP(d_0).
$$
By Lemma \ref{multiples1}, this is an identity in $A(K_{d_0})^-\otimes_\BZ\BQ$.

We first claim that the equality is true up to a multiple in $\BQ^\times$.
In fact, if $L'(A_{K_n},\chi,1)=L'(A_{d_0},1)L(A_{d_1},1)$ is zero, then the right-hand side is zero by definition, and $P_\chi$ is zero since the canonical height $\wh h(P_\chi)=0$ by the Gross--Zagier formula (in either \cite[Theorem 1.2]{YZZ} or the explicit version Theorem \ref{EGZ}).
If $L'(A_{K_n},\chi,1)\neq 0$, then by the theorems of Gross--Zagier and Kolyvagin, $E(K_{d_0})^-\otimes_\BZ\BQ$ is one-dimensional, and the thus two sides of the equality are proportional.

To finish the proof, it suffices to check that the two sides of the equality have the same canonical heights.
One can do the whole computation on $A$, but we will carry it out on $E$ to be compatible with our original framework.

Let $\varphi: A\ra E$ be the isogeny of degree $2$.
The desired formula becomes
$$ R_\chi= \epsilon(d_0, d_1) 2^{h_2(n)} \CL(d_1)\CR(d_0) \ \in\ E(H_n'(i))\otimes_\BZ\BQ.$$
Here $R_\chi=\varphi(P_\chi)$ and $\CR(d_0)=\varphi(\CP(d_0))$.
The vector $\CR(d_0)\in E(K_{d_0})^-\otimes_\BZ\BQ$ has an independent description. If $\CL(d_0)=0$, then $\CR(d_0)=0$. If
$\CL(d_0)\neq 0$, then the theorems of Gross--Zagier and Kolyvagin imply that $E(K_{d_0})^-$ is of rank one.
In this case, $\CR(d_0)=2^{-1}\CL(d_0)\beta_{d_0}\in E(K_{d_0})^-_\BQ$,
where $\beta_{d_0}\in E(K_{d_0})^-$ is any $\BZ$-basis of the free part of $E(K_{d_0})^-$.

The height identity we need to check is
$$\wh{h}(R_\chi)=4^{h_2(n)-1} \CL(d_1)^2 \CL(d_0)^2 \wh{h}(\beta_{d_0}).$$
Assuming $L'(E_{K_n},\chi,1)\neq 0$. By the definitions of $\CL(d_1)$ and $\CL(d_0)$ in the introduction, the identity becomes
$$\wh{h}(R_\chi)=
L'(E_{K_n},\chi,1)/(2^{2k(n)-2h_2(n)-2-a(n)} \Omega _{d_0, \infty} \Omega _{d_1, \infty}).
$$

Apply Theorem \ref{EGZ}, the explicit Gross-Zagier formula in the appendix, for $(E_{K_n},\chi_{d_0, d_1})$ and the morphism $\varphi \circ f_n$.
The proof is finished by computations similar to that in the proof of Theorem \ref{thm2.1}.

In the proof, we also see that $\CL(n)\in \BQ$. For example, the height formula
$$\wh{h}(R_\chi)=4^{h_2(n)-1} \CL(d_1)^2 \CL(d_0)^2 \wh{h}(\beta_{d_0})$$
actually implies that $\CL(d_1) \CL(d_0)\in \BQ$. Setting $d_1=1$, we see that $\CL(n)\in \BQ$.

\subsubsection*{Proof of Proposition \ref{congruence}}

Here we prove Proposition \ref{congruence} which asserts that
\begin{multline*}
P(n) \equiv
\sum_{\substack{n=d_0d_1\cdots d_\ell\\
d_0\equiv 5,6, 7\pmod 8\\
d_1\equiv 1, 2, 3\pmod 8\\
d_i\equiv 1 \pmod 8, \ i>1}}\epsilon(d_0, d_1) \left(\prod_{i\geq 1} g(d_i)\right) Z(d_0)
\\
+i\sum_{\substack{n=d_0d_1\cdots d_\ell\\
(d_0,d_1,d_2)\equiv (5,3,2)\pmod 8\\
d_i\equiv 1 \pmod 8, \ i>2}} \left(\prod_{i\geq 1} g(d_i)\right) Z(d_0)
\qquad\mod\ 2A(\BH_n').
\end{multline*}

It suffices to prove that the above formula
(applied to every $P(d_0)$ below) and the formula in Theorem \ref{lg} (applied to every $\CL(d_1)$ below) satisfies
$$Z(n)\equiv \sum_{\substack{n=d_0d_1\\ d_0\equiv 5, 6, 7\pmod 8\\
d_1\equiv  1,2 ,3 \pmod 8}} \epsilon(d_0, d_1)\CL(d_1)P(d_0)\qquad
\mod\ 2A(\BH_n').$$

We first treat the case $n\equiv 5,7\pmod8$. Then the formula simplifies as
\begin{multline*}
P(n) \equiv
\sum_{\substack{n=d_0d_1\cdots d_\ell\\
d_0\equiv 5, 7\pmod 8\\
d_1\equiv 1, 3\pmod 8\\
d_i\equiv 1 \pmod 8, \ i>1}}\epsilon(d_0, d_1) \left(\prod_{i\geq 1} g(d_i)\right) Z(d_0)
\qquad\mod\ 2A(\BH_n').
\end{multline*}
We need to check that
\begin{multline*}
Z(n)\equiv \sum_{\substack{n=d_0d_1\\ d_0\equiv 5, 7\pmod 8\\
d_1\equiv  1 ,3 \pmod 8}} \epsilon(d_0, d_1)
\left(\sum_{\substack {d_1=d_0'd_1'\cdots d_{\ell'}'\\
d_j'\equiv 1\pmod 8, \ j>0}}
\prod_{j\geq 0} g(d_j') \right)
\\
\left(\sum_{\substack{d_0=d_0''d_1''\cdots d_{\ell''}''\\
d_0''\equiv 5,7\pmod 8\\
d_1''\equiv 1, 3\pmod 8\\
d_k''\equiv 1 \pmod 8, \ k>1}}\epsilon(d_0'', d_1'')
\prod_{k\geq 1} g(d_k'')Z(d_0'')\right)
\qquad
\mod\ 2A(\BH_n').
\end{multline*}
The right-hand side is a $\BZ$-linear combination of
$$
 \epsilon(d_0, d_1) \epsilon(d_0'', d_1'')
\prod_{j=0}^{\ell'} g(d_j')
\prod_{k=1}^{\ell''} g(d_k'') Z(d_0'').
$$
Consider the multiplicity of this term in the sum.
Each appearance of such a terms gives a partition
$$
\{d_1',\cdots, d_{\ell'}', d_2'',\cdots, d_{\ell''}''\}
=\{d_1',\cdots, d_{\ell'}'\}\cup \{d_2'',\cdots, d_{\ell''}''\}.
$$
If this set is non-empty, the number of such partitions is even, and thus the contribution is zero in the congruence equation. Moreover, if $d_0'\equiv 1\pmod 8$ or $d_1''\equiv 1\pmod 8$, then we can put also put it into the partition deduce that the contribution of such terms is still zero.

Note that the contribution by $d_0=d_0''=n, d_1=1$ is the single term $Z(n)$.
Therefore, it is reduced to check
$$0\equiv \sum_{\substack{n=d_0d_1\\ d_0\equiv 5, 7\pmod 8\\
d_1\equiv  3 \pmod 8}} \epsilon(d_0, d_1) g(d_1)
\sum_{\substack{d_0=d_0''d_1''\\
d_0''\equiv 5, 7\pmod 8\\
d_1''\equiv 3\pmod 8}}\epsilon(d_0'', d_1'') g(d_1'') Z(d_0'')
\qquad
\mod\ 2A(\BH_n').$$
Rewrite it as
$$0\equiv {\sum_{n=d_0''d_1''d_1}}'
\epsilon(d_0''d_1'', d_1)
\epsilon(d_0'', d_1'') g(d_1) g(d_1'') Z(d_0'')
\qquad
\mod\ 2A(\BH_n').$$
Here the sum is over ordered decompositions $n=d_0''d_1''d_1$ which satisfy the original congruence conditions (with $d_0=d_0''d_1''$).
The ordered decomposition
$n=d_0''d_1''d_1$ corresponds to the ordered decomposition $n=d_0''d_1d_1''$ uniquely.
One checks in this case
$$
\epsilon(d_0''d_1'', d_1)  \epsilon(d_0'', d_1'')
=\pm \epsilon(d_0''d_1, d_1'')  \epsilon(d_0'', d_1).
$$
Then the sum is divisible by $2$.

Now we treat the case $n\equiv 6\pmod8$.
We need to check
\begin{multline*}
Z(n)\equiv \sum_{\substack{n=d_0d_1\\ d_0\equiv 6, 7\pmod 8\\
d_1\equiv  1, 2 \pmod 8}}
\left(\sum_{\substack {d_1=d_0'd_1'\cdots d_{\ell'}'\\
d_j'\equiv 1\pmod 8, \ j>0}}
\prod_{j\geq 0} g(d_j') \right)\cdot
\\
\left(\sum_{\substack{d_0=d_0''d_1''\cdots d_{\ell''}''\\
d_0''\equiv 5,6,7\pmod 8\\
d_1''\equiv 1, 2,3\pmod 8\\
d_k''\equiv 1 \pmod 8,  \ k>1}}
 \epsilon(d_0'', d_1'')
\prod_{k\geq 1} g(d_k'')Z(d_0'')
\,+\,
i\sum_{\substack{d_0=d_0''d_1''\cdots d_{\ell''}''\\
(d_0'',d_1'',d_2'')\equiv (5,3,2)\pmod 8\\
d_k''\equiv 1 \pmod 8, \ k>2}}
\prod_{k\geq 1} g(d_k'')Z(d_0'')
\right) \\
\qquad
\mod\ 2A(\BH_n').
\end{multline*}
Split the outer sum $d=d_0d_1$ into the case $(d_0,d_1)\equiv (6,1) \pmod8$
and the case $(d_0,d_1)\equiv (7,2) \pmod8$. We obtain three triple sums, since the conditions $(d_0,d_1)\equiv (7,2) \pmod8$ and $(d_0'',d_1'',d_2'')\equiv (5,3,2)\pmod 8$ do not hold simultaneously.
Similar to the case $n\equiv 5,7\pmod8$, the contribution of the terms with some $d_j'\equiv 1\pmod 8$ or some $d_k''\equiv 1\pmod 8$ is divisible by
$2$. In particular, for the case $d_1\equiv 1\pmod 8$, we are only left with $d_1=1$.
Then it is reduced to check
\begin{multline*}
0\equiv \sum_{\substack{n=d_0d_1\\ d_0\equiv 7\pmod 8\\
d_1\equiv  2 \pmod 8}}
\left( g(d_1) Z(d_0)
+\sum_{\substack{d_0=d_0''d_1''\\
d_0''\equiv 5\pmod 8\\
d_1''\equiv 3\pmod 8}}i\, g(d_1) g(d_1'') Z(d_0'')
\right)\\
+
\sum_{\substack{n=d_0''d_1''\\
d_0''\equiv 7\pmod 8\\
d_1''\equiv 2\pmod 8}}   g(d_1'') Z(d_0'')
\, +\,
i \sum_{\substack{n=d_0''d_1''d_2''\\
(d_0'',d_1'',d_2'')\equiv (5,3,2)\pmod 8}}
g(d_1'')g(d_2'') Z(d_0'')
\qquad
\mod\ 2A(\BH_n').
\end{multline*}
This is true by obvious cancellations, which finishes the proof of the proposition.

\subsubsection*{Torsion points}

To prepare the proof of Theorem \ref{m3}, we present some results on torsion points of $A$. They will be the key to lower multiples of algebraic points.

Denote $F=\BQ(i)$.
Recall that we have fixed an identification
$A(\BC)\cong\BC/(1+i)O_F$,
which gives $A(\BC)_\tor=A(F^\ab)_\tor\cong F/(1+i)O_F$.
Under the identification, the complex conjugation on $A(F^\ab)_\tor$ is given by the conjugation $i\mapsto -i$ on $F$,
The induced action of the Galois group $\Gal(F^\ab/F)$ on
$F/(1+i)O_F$ is given by multiplying by the composition
$$\Gal(F^\ab/F)\stackrel{\sigma_{F}^{-1}}{\lra}
F^\times \bs \wh{F}^\times \cong (1+(1+i)^3\wh O_F)^\times.$$

\begin{lem} \label{torsion1}
Over $F$, the elliptic curve $A_F$ is isomorphic to $E_F$. Moreover,
$$\BQ(A[4])=\BQ(\sqrt 2, i),\quad A(\BQ(i))=A[(1+i)^3].$$
\end{lem}
\begin{proof}
The results can be checked by explicit computations, but we include a theoretical proof.
For the first statement, consider the two 2-isogenies
$$
\varphi_F: A_F\lra E_F, \quad [1+i]: A_F\lra A_F.
$$
One checks that these two morphisms have the same kernel $\{0,\tau(1)\}$. It
follows that there is an isomorphism $A_F\to E_F$ carrying $[1+i]$ to
$\psi_F$.

Now we treat $\BQ(A[4])$. It is easy to have $F=\BQ(A[2])\subset \BQ(A[4])$, and thus $\BQ(A[4])=F(A[4])$.
The Galois action of $\Gal (F^\ab/F)$ on $A[4]$ is given by
$$s_4:(1+(1+i)^3\wh O_F)^\times
\lra (1+(1+i)^3O_{F_2})/(1+4O_{F_2}).$$
The field $F(A[4])$ is given by the subfield of $F^\ab$ fixed by
$\ker(s_4)=(1+4\wh O_F)^\times$, which is the ring class field of $F$ of conductor $4$. The norm map
$$(1+(1+i)^3O_{F_2})/(1+4O_{F_2}) \simeq (1+4\BZ_2)/(1+8\BZ_2)$$
implies that $F(A[4])$ is equal to the ring class field $\BQ(\zeta_8)$ of $\BQ$.

For $A(\BQ(i))$, we first see that it is torsion since $A_1(\BQ)\simeq A_{-1}(\BQ)$ are torsion. We also have $A(\BQ(i))[p]=0$ for any odd prime $p$.
In fact, we can show that any non-trivial element of $A[p]$ has a residue field ramified above $p$ and cannot be defined over $\BQ(i)$.
This argument will be used in Lemma \ref{torsion2} in a more complicated situation, so we omit it here.

Finally, we show that $A(\BQ(i))[2^\infty]=A[(1+i)^3]$.
Note that the stabilizer of any element $x_4$ of $A[4]\setminus A[(1+i)^3]$ is still $\ker(s_4)$. Then the residue field $F(x_4)$ is still $\BQ(\zeta_8)$, and thus $x_4\notin A(\BQ(i))$. It follows that  $A[4](\BQ(i))=A[(1+i)^3]$.
\end{proof}

\begin{lem} \label{torsion3}
Let $\kappa\in \Gal(\BQ(\zeta_8)/\BQ)$ be the element sending $\zeta _8$ to $\zeta _8^5$. Then
$$(\kappa+1)A[4]=A[4][\kappa+1]=A[2], \qquad (\kappa+1)E[4]=E[4][\kappa+1]=E[2].$$
Here $(\kappa+1)A[4]$ and $A[4][\kappa+1]$ are respectively the image and the kernel of the map
$$\kappa+1: A[4]\lra A[4], \quad x\longmapsto x^\kappa+x.$$
\end{lem}
\begin{proof}
The results for $A$ and $E$ are equivalent since they are isomorphic over $F=\BQ(i)$.
Note that $\kappa$ acts on $\BQ(\zeta_8)\subset F^\ab$ as $\sigma_{F,2}(\pm 1\pm 2i)=\sigma_{F,2}(\pm 2\pm i)$.
In terms of the CM theory, $\kappa$ acts on $A[4]\cong O_F/4O_F$ by multiplication by $-1\pm 2i\in 1+(1+i)^3O_{F_2}$. Then $\kappa+1$ acts by multiplication by $\pm 2i$.
The results are true.
\end{proof}

\begin{lem} \label{torsion2}
The torsion subgroup $A(\BH_n')_\tor=A[(1+i)^3]$ if $n$ is odd, and $A(\BH_n')_\tor=A[4]$ if $n$ is even.
\end{lem}

\begin{proof}
We prove the results by three steps.
\medskip

\noindent \emph{Step 1}. The group $A(\BH_n')[p]=0$ for any odd prime $p$.
Let $\wp$ be a prime ideal of $F$ above $p$.
The action of the Galois group on $A[\wp]$ gives a homomorphism
$$
\Gal(F^\ab/F)\lra \Aut_{O_F}(A[\wp])=(O_F/\wp)^\times.
$$
This map is surjective since it is given by
$$s_\wp:(1+(1+i)^3\wh O_F)^\times
\lra O_{F_\wp}^\times \lra (O_F/\wp)^\times.$$
As a consequence, we have the following two properties:
\begin{enumerate}[(1)]
\item For any nonzero $x\in A[\wp]$, the residue field $F(x)=F(A[\wp])$ has degree $N(\wp)-1\geq 4$ over $F$.
\item The prime $\wp$ is totally ramified in $F(x)$.
\end{enumerate}
On the other hand, we claim that the ramification index of $\wp$ in $\BH_n'$ is at most $2$.
In fact, $\BH_n'$ is the composite of $L_n(i)$ and $H_{d_0}'$ for different $d_0$, where the extensions $L_n(i)/K_n$ and $H_{d_0}'/K_{d_0}$ do not involve ramification above
$p$. If follows that we only need to consider the ramification index of $p$ in
the composite of $K_n$ and $K_{d_0}$ for different $d_0$, which is at most $2$.

Combining the claim and the properties (1) and (2),
we see that $F(x)$ cannot be contained in $\BH_n'$. In other words,  $A(\BH_n')[\wp]=0$. Then $A(\BH_n')[p]=0$.
Hence, $A(\BH_n')_\tor=A(\BH_n')[2^\infty]$.
\medskip

\noindent \emph{Step 2}.
For any $n\equiv 5,6,7\pmod8$, $A(\BH_n')[2^\infty]\subset A[4]$.
Note that $A(\BH_n')[2^\infty]$ is a finite $O_F$-module, so it must be of the form $A[(1+i)^e]$ for some positive integer $e$. Thus it suffices to prove $|A(\BH_n')[2^\infty]|\leq 16$.

The idea is to use the reduction map to obtain the bound.
Take a prime number $p\nmid (2n)$, and let $v$ be a place of $\BH_n'$ above $p$. Denote by $k(v)$ the residue field of $v$. The reduction map gives an injection
$$
A(\BH_n')[2^\infty] \longrightarrow A(k(v))[2^\infty].
$$
We will choose $p$ carefully to get an easy bound on the right-hand side.
In fact, we choose $p$ satisfying the following properties:
\begin{enumerate}[(1)]
\item $p\equiv 3\pmod8$.
\item $p$ is inert in $K_{d_0}$ for any positive factor $d_0$ of $n$ with $d_0\equiv 5,6\pmod8$.
\end{enumerate}

Assuming the existence of such $p$, we first see how it implies the desired bound. The proof consists of two steps.
The first step is to show that $k(v)=\BF_{p^2}$.
Denote by
$w$ the restriction of $v$ to $L_n(i)$, and
$v_{d_0}$ the restriction of $v$ to $H_{d_0}'$. It is easy to see that the residue field $k(w)=\BF_{p^2}$. To prove $k(v)=\BF_{p^2}$, it suffices prove that $k(v_{d_0})\subset \BF_{p^2}$ for any $d_0\equiv 5,6\pmod8$.
Note that $p$ is inert in $K_{d_0}$. Then it suffices to check that $pO_{K_{d_0}}$ is totally split in $H_{d_0}'$.
By Lemma \ref{CM point1},  $H_{d_0}'$ is contained in the ring class field $H_{d_0,4}$ of conductor $4$.
We claim that $pO_{K_{d_0}}$ is totally split in $H_{d_0,4}$.
In fact, by the class field theory, it is equivalent to the easy fact that the image of $p$ under the composition
$$
K_{d_0, p}^\times \lra \wh K_{d_0}^\times \lra
K_{d_0}^\times\bs \wh K_{d_0}^\times/(\wh \BZ+4 \wh O_{K_{d_0}})^\times=\Gal(H_{d_0,4}/K_{d_0})
$$
is trivial.

The second step is to show that  $|A(\BF_{p^2})[2^\infty]|\leq 16$.
This is done by explicit computation. In fact, by the choice $p\equiv 3\pmod8$, we see that $A$ has supersingular reduction at $p$.
Then the eigenvalues of the absolute Frobenius $\varphi_p$ on the Tate modules of $A$ are $\pm\sqrt{-p}$, so the eigenvalues of $\varphi_p^2$
are $-p, -p$. It follows that
$$
|A(\BF_{p^2})|=p^2+1-(-p-p)=(p+1)^2.
$$
By the choice $p\equiv 3\pmod8$, we have $|A(\BF_{p^2})[2^\infty]|= 16$.
This finishes the second step.

Finally, we check the existence of the prime $p$ satisfying the two conditions.
The second condition is equivalent to $(-d_0/p)=-1$, which becomes
$(d_0/p)=1$ by the first condition. Then we choose $p$ satisfying:
\begin{enumerate}[(a)]
\item $p\equiv 3\pmod8$.
\item $(\ell/p)=1$ for any prime factor $\ell$ of $n$ with $\ell\equiv 1\pmod4$.
\item $(\ell/p)=-1$ for any prime factor $\ell$ of $n$ with $\ell\equiv -1\pmod4$.
\end{enumerate}
It is easy to check that it gives $(d_0/p)=1$ for any $d_0\equiv 5,6\pmod8$.
Now the existence of $p$ satisfying (a), (b) and (c) is just a combination of the quadratic reciprocity law, the Chinese remainder theorem, and Dirichlet's density theorem.

\medskip

\noindent \emph{Step 3}.
If $n$ is odd, then $A(\BH_n')[2^\infty]= A[(1+i)^3]$.
We will prove $\sqrt2\notin \BH_n'$, which implies
$A(\BH_n')_\tor=A[(1+i)^3]$ by Lemma \ref{torsion1}.

To prove $\sqrt2\notin \BH_n'$, note that $\BH_n'$ is the composite of $L_n(i)$ and $H_{d_0}'$ for some $d_0\equiv 5\pmod8$. Let $v$ be a place of $\BH_n'$ above 2, and $v_{d_0}$ the restriction to $H_{d_0}'$. It suffices to show  $\sqrt2\notin (\BH_n')_v$.
Consider the ramification of $v$ above $2$.
Note that $(\BH_n')_v$ is the composite of $\BQ_2(i)$ and $(H_{d_0}')_{v_{d_0}}$ for all related $d_0$. By Proposition \ref{CM point1},
$(H_{d_0}')_{v_{d_0}}$ is unramified over
$N_{d_0,4}=(M_{d_0,4})^{\sigma_{1+2\varpi_{d_0}}}$, where $\varpi_{d_0}=(\sqrt{-d_0}-1)_2$ and $M_{d_0,4}$ is the ring class field  of $\BQ_2(\sqrt{-d_0})$ of conductor $4$.

We claim that $N_{d_0,4}$ is independent of
$d_0$.
In fact, fix an isomorphism $\BQ_2(\sqrt{-d_0})\cong\BQ_2(\sqrt{-5})$, which induces an isomorphism $M_{d_0,4}\cong M_{5,4}$. Note that $1+2\varpi_{d_0}$ and $1+2\varpi_{5}$ have the same image in $K_{5,2}^\times/(\BZ_2+4O_{K_{5,2}})^\times$, so their actions on $M_{5,4}$ are the same. It follows that $N_{d_0,4}=N_{5,4}$.

Note that $i\in N_{5,4}$ and $\sqrt2\notin N_{5,4}$.
Therefore, $(\BH_n')_v$ is unramified over $N_{5,4}$. To prove
$\sqrt2\notin (\BH_n')_v$, it suffices to prove that $N_{5,4}(\sqrt2)=M_{5,4}$ is ramified over $N_{5,4}$. This is clear since $\Gal(M_{5,4}/N_{5,4})$ is generated by $\sigma_{1+2\varpi_{5}}$ with $1+2\varpi_{5}\in O_{K_{5,2}}^\times$.

\end{proof}

\subsubsection*{Proof of Theorem \ref{m3}: representative}

By definition, $\Phi_0\subset \Phi$. Recall that
$$P_\chi=\sum_{t\in \Phi} f_n(P_n)^{t} \chi(t).$$
Summing over all characters $\chi: \Cl_n\cong \Cl_n'/\langle\sigma\rangle\ra \{\pm 1\}$.
We have
$$\sum_{\chi: \Cl_n\ra \{\pm 1\}} P_\chi
=\sum_{t\in \Phi} f_n(P_n)^{t} \sum_{\chi: \Cl_n\ra \{\pm 1\}} \chi(t).$$
As in the case $n\equiv 1,2,3\pmod8$, apply the character formula
$$ \sum_{\chi: \Cl_n\ra \{\pm 1\}} \chi(t)
=2^{h_2(n)}\delta _{2\Cl_n}(t), \qquad t\in \Cl_n.$$
Here $h_2(n)=\dim _{\BF_2}\Cl_n/2\Cl_n$.
Then we obtain
$$\sum_{\chi: \Cl_n\ra \{\pm 1\}} P_\chi
=2^{h_2(n)} \sum_{t\in \Phi_0} f_n(P_n)^{t}
=2^{h_2(n)} Z(n).$$
This is an equality in $A(H_n')$.

By Theorem \ref{GZ}, the equality gives
$$\sum_{\substack{n=d_0d_1\\ d_0\equiv 5, 6, 7\pmod 8\\
d_1\equiv  1,2 ,3 \pmod 8,\ d_1>0}}
\epsilon (d_0, d_1) 2^{h_2(n)}\CL(d_1) \CP(d_0)
=2^{h_2(n)} Z(n)\
\in\ A(H_n'(i))\otimes_{\BZ} \BQ.$$
We end up with
$$\sum_{\substack{n=d_0d_1\\ d_0\equiv 5, 6, 7\pmod 8\\
d_1\equiv  1,2 ,3 \pmod 8,\ d_1>0}}
\epsilon (d_0, d_1)  \CL(d_1) \CP(d_0)
= Z(n)\  \in\ A(H_n'(i))\otimes_{\BZ} \BQ.$$
Then we have
$$\CP(n)=Z(n)-\sum_{\substack{n=d_0d_1\\ d_0\equiv 5, 6, 7\pmod 8\\
d_1\equiv  1,2 ,3 \pmod 8,\ d_1>1}} \epsilon(d_0, d_1)\CL(d_1)\CP(d_0).$$
It follows that $\CP(n)$ and $P(n)$ satisfy the same iteration formula (in different groups).
Therefore, $P(n)$ represents $\CP(n)$.
This proves the first statement of the theorem.

\subsubsection*{Proof of Theorem \ref{m3}: part (1)}

Here we prove part (1) of the theorem.
Let $R(n)$
(resp. $R(d_0, d_1)$)
be the image of $P(n)$
(resp. $P(d_0, d_1)=P_{\chi_{d_0, d_1}}$)
under the $2$-isogeny from $A$ to $E$. Then $R(n)\in E(\BH_n')$ and
$R(d_0, d_1)\in E(H_n')$.
We need to prove that $2R(n)\in E(K_n)^-$.

Note that in Lemma \ref{multiples1} we have already checked $4P(n,1)\in A(K_n)^-$ and thus $4R(n,1)\in E(K_n)^-.$ To relate to $2R(n)$, we have the following simple connection.

\begin{lem} \label{multiples2}
$$4 P(n,1)=\pm 2^{2+h_2(n)}  P(n), \qquad
4 R(n,1)=\pm 2^{2+h_2(n)}  R(n).$$
\end{lem}
\begin{proof}
By Theorem \ref{GZ},
$$P(n,1)=\pm 2^{h_2(n)}  P(n)\
\in\ A(\BH_n')\otimes_{\BZ} \BQ.$$
Then
$$P(n,1)\mp 2^{h_2(n)}  P(n)\
\in\ A(\BH_n')_\tor=  A(\BH_n')[4].$$
Here the last identity follows from Lemma \ref{torsion2}.
\end{proof}

Before proving part (1) of the theorem,
we introduce some notations on fields.
Recall that $H_n$ is the Hilbert class field of $K_n=\BQ(\sqrt{-n})$ and $H_n'=H_n(f_n(P_n))$. Recall that $K_n'=K_n, K_n(i), K_n$ for $n\equiv 5, 6,7\pmod 8$ respectively.
Let $L_n\subset H_n$ be the genus field of $K_n$; that is, $L_n$ is subfield of $H_n$ fixed by the subgroup $2\Cl_n$ of $\Cl_n=\Gal(H_n/K_n)$. Define
$L_n'=L_n, L_n(i), L_n$ for $n\equiv 5, 6,7\pmod 8$ respectively.
Then $L_n'=L_nK_n'$.
Set $K_n''=K_n(E[4](L_n'))$, i.e. $K_n''=K_n(i), K_n(\sqrt{2}, i), K_n$ for $n\equiv 5, 6,7\pmod 8$ respectively.

First, we prove $2R(n)\in E(K_n'')$.
Consider the image of $4R(n, 1)=\pm2^{h_2(n)+2}R(n)$ under the (injective)  Kummer map $$\delta: E(K_n'')/2^{h_2(n)+2}E(K_n'') \lra H^1(K_n'', E[2^{h_2(n)+2}]),$$ and the inflation-restriction exact sequence
$$1\lra \Hom (\Gal(L_n'/K_n''), E[4](K_n''))\lra H^1(K_n'', E[2^{h_2(n)+2}])
\lra H^1(L_n', E[2^{h_2(n)+2}]).$$
(Note that $E[2^\infty](L_n')=E[4](K_n'')$.)
The image of $\delta(2^{h_2(n)+2}R(n))$ in $H^1(L_n', E[2^{h_2(n)+2}])$ is 0, since it is 0 in $E(L_n')/2^{h_2(n)+2}E(L_n')$.
Then $\delta(2^{h_2(n)+2}R(n))$ lies in $\Hom(\Gal(L_n'/K_n''), E[4](K_n''))$, which has exponent 2 since $\Gal(L_n'/K_n'')$ has exponent 2.
It follows that $\delta(2^{h_2(n)+3}R(n))=0$. Thus
$$2^{h_2(n)+3}R(n)\in 2^{h_2(n)+2}E(K_n''), \qquad
2R(n)\in E(K_n'')+E[2^\infty](L_n')=E(K_n'').$$

Second, we prove $2R(n)\in E(K_n)^-$ for $n\equiv 7\pmod 8$.
This is the simplest case, but it illustrates the key idea.
In this case, we already have $2R(n)\in E(K_n)$, and we need to prove
$2\overline{R(n)}=-2R(n)$.
By Lemma \ref{multiples1} and Lemma \ref{multiples2},
$$2^{h_2(n)+2}(R(n)+\overline{R(n)})=\pm (4R(n, 1)+4\overline{R(n, 1)})=0.$$
Then $R(n)+\overline{R(n)}\in E(K_n)[2^\infty]=E[2]$ is killed by 2.
The result follows.

Third, we prove $2R(n)\in E(K_n)^-$ for $n\equiv 5\pmod 8$.
It suffices to prove $2R(n)\in E(K_n)$, since the process from $E(K_n)$ to $E(K_n)^-$ is the same as the case  $n\equiv 7\pmod 8$.
We already know $2R(n)\in E(K_n(i))$.
Denote by $\xi\in \Gal(K_n(i)/K_n)$ the unique non-trivial element, and take a lifting of $\xi$ to $\Gal(\BH_n'/K_n)$, which we still denote by $\xi$.
By Lemma \ref{multiples1} and Lemma \ref{multiples2},
$$2^{h_2(n)+2}(P(n)^\xi-P(n))=\pm (4P(n, 1)^\xi-4 P(n, 1))=0.$$
Then $P(n)^\xi-P(n)\in A(K_n(i))[2^\infty]=A[(1+i)^3]$. Note that $A[(1+i)^3]$ is
exactly killed by $2\varphi:A\to E$. We have $2R(n)^\xi-2R(n)=0$, and thus $2R(n)\in E(K_n)$.

For the case $n\equiv 6\pmod 8$, we need the following simple result.

\begin{lem}
For any $n\equiv 5,6,7\pmod 8$,  $R(n)\in E(L_n(i))$.
\end{lem}
\begin{proof}
By the recursion formula, it suffices to prove $\varphi(Z(n))\in E(L_n')$ for any $n\equiv 5,6,7\pmod 8$.
By Theorem \ref{CM point2}, $\varphi(z_n)$ is invariant under the action of
$\sigma$. Here $\sigma$ is described right after Proposition \ref{CM point1}.
If $n\equiv 5,7\pmod 8$, then $\varphi(z_n)$ is defined over $H_n$, and thus $\varphi(Z(n))$ is defined over $L_n$.
If $n\equiv 6\pmod 8$, then $\varphi(z_n)$ is defined over $H_n(i)$, and thus $\varphi(Z(n))$ is defined over $L_n(i)$.
\end{proof}

Finally, we prove $2R(n)\in E(K_n)^-$ for $n\equiv 6\pmod 8$.
We already know $2R(n)\in E(K_n(\sqrt{2}, i))$.
It suffices to prove $2R(n)\in E(K_n(i))$, since the process from $E(K_n(i))$ to $E(K_n)^-$ is the same as that for the case $n\equiv 5\pmod 8$.

Let $\kappa \in \Gal(K_n(\sqrt{2}, i)/K_n(i))$ be the unique non-trivial element, and take any lifting of $\kappa$ in $\Gal(L_n(i)/K_n(i))$, still denoted by $\kappa$.
We need to show that $(2R(n))^{\kappa}=2R(n)$.
Note that $\kappa^2=1$ since $\Gal(L_n(i)/K_n(i))$ has exponent 2.
By Lemma \ref{multiples1} and Lemma \ref{multiples2},
$$2^{h_2(n)+2}(R(n)^\kappa-R(n))=\pm (4R(n, 1)^\kappa-4R(n, 1))=0,$$
so $R(n)^\kappa-R(n)$ lies in
$E[4][\kappa+1]=\{x\in E[4]:x^\kappa+x=0  \}$.
By Lemma \ref{torsion3},
$E[4][\kappa+1]=E[2]$. It follows that $2(R(n)^\kappa-R(n))=0$.
The proof of part (1) is complete.

\subsubsection*{Proof of Theorem \ref{m3}: part (2)}

We start with some Galois-theoretic preparation.
Denote by
$$r:   \BQ^\times \bs \BA^\times \lra \Gal (\BQ^\ab/\BQ)$$
the Artin map over $\BQ$.
Then $c=r_\infty (-1)$ is the complex conjugation.
Define $\beta_1,\beta_2\in \Gal(\BQ^\ab/\BQ)$ by
$$\beta_1=r _\infty (-1)r _2(-2), \qquad
\beta_2=r _\infty (-1)r _2 (6).
$$
Let $\beta_1', \beta_2'\in \Gal(\overline\BQ/\BQ)$ be any liftings of $\beta_1,\beta_2$.

In the following, we take the convention that $(\gamma+1)R$ means $\gamma(R)+1$ for any $\gamma\in \Gal(\overline\BQ/\BQ)$.
The key of the proof is the following lemma.
\begin{lem} \label{betaaction}
\begin{enumerate}[(1)]
\item For any $n\equiv 5\pmod 8$,
$$ Z(n)^{\beta_1'+1}=Z(n)^{\beta_2'+1} \ \in \  g(n)\ \tau(\frac{1-i}2)+\BZ\, \tau(1).$$
\item For any $n\equiv -2\pmod{16}$,
$$
Z(n)^{\beta_1'+1} \ \in \  g(n)\ \tau(\frac{i}2)+\BZ\, \tau(1).$$
\item For any $n\equiv 6\pmod{16}$,
$$
Z(n)^{\beta_2'+1} \ \in \  g(n)\ \tau(\frac{i}2)+\BZ\, \tau(1).$$
\item For any $n\equiv 7\pmod 8$,
$$Z(n)^{\beta_1'+1}=Z(n)^{\beta_2'+1}= g(n)\ \tau(\frac12).$$
\end{enumerate}
\end{lem}
\begin{proof}
Recall that after Proposition \ref{CM point1} we have introduced $\sigma\in 2\Cl_n'$ which gives
$$
\Cl'_n/\langle\sigma\rangle\cong\Cl_n, \quad
(2\Cl'_n)/\langle\sigma\rangle\cong2\Cl_n.
$$
Note that the genus field $L_n$ is the subfield of $H_n$ fixed by $2\Cl_n$.
It follows that the subfield of $H_n'$ fixed by $2\Cl_n'$ is $L_n, L_n(i), L_n$ according to $n\equiv 5,6,7\pmod 8$.

The field $L_n(i)=\BQ(i, \sqrt{d}:d|n)$ is a subfield of $\BQ^\ab$. It is easy to check that the action of the involved $\beta_j'$ on $L_n(i)$ is the same as that of
$\sigma_{\varpi}\circ c$ in all the four cases of the lemma.
For example, if $n\equiv 5\pmod 8$, then
$\sigma_{\varpi}$ acts on $L_{n}(i)$ as
$r_2(N_{K_{n}/\BQ}(\varpi))=r_2(n+1)=r_2(-2)$.
As a consequence, we claim that
$$
Z(n)^{\beta_j'}-Z(n)^{\sigma_{\varpi}\circ c} \in \BZ\, \tau(1)
$$
in all four cases.

In fact, denote
$\alpha=\sigma_{\varpi}\circ c \circ \beta_j'^{-1}$, viewed as an element of
$\Cl_n'=\Gal(H_n'/K_n')$.
It suffices to show
$$Z(n)^{\alpha}-Z(n) \in \BZ\, \tau(1).$$
Since $\alpha$ acts trivially on $L_n(i)$, we see that $\alpha\in 2\Cl_n'$.
Recall the definition
$$Z(n)=\sum_{t\in \Phi_0} z_n^{t}, \qquad
Z(n)^\alpha=\sum_{t\in \alpha\Phi_0} z_n^{t}.$$
Here $\Phi_0$ is a set of representatives of $2\Cl_n=(2\Cl'_n)/\langle\sigma\rangle$ in $2\Cl_n'$. Then $\alpha\Phi_0$ is also a set of representatives of $2\Cl_n$ in $2\Cl_n'$.
Write $\Phi_0=\{t_i:i=1,\cdots, g(n)\}$.
Then $\alpha\Phi_0=\{\sigma_i t_i: i=1,\cdots, g(n)\}$,
where each $\sigma_i\in  \langle\sigma\rangle$.
By Theorem \ref{CM point2}, we see that $z_n^\sigma=z_n$ or $z_n^\sigma=z_n+\tau(1)$. It follows that
$$
Z(n)^\alpha-Z(n)=\sum_{t_i\in \Phi_0} (z_n^{\sigma_i}-z_n)^{{t_i}}
\in \BZ\, \tau(1).$$

Therefore, the result for $Z(n)^{\beta_j'+1}$ becomes that for
$Z(n)^{\sigma_{\varpi}\circ c+1}$, which can be checked easily by Theorem \ref{CM point2} for $n\equiv 5,6\pmod 8$.
In the case $n\equiv 7\pmod 8$, $H_n'=H_n$ and thus $Z(n)$ is already defined over $L_n$. Then
$$
Z(n)^{\beta_j'}=Z(n)^{\sigma_{\varpi}\circ c}
=Z(n)^{\sigma_{\varpi^5}\circ c}.
$$
Here the last identity holds since the Galois group $\Gal(L_n(i)/\BQ)$ has exponent 2. Then the result for $Z(n)^{\beta_j'+1}$ still follows from Theorem \ref{CM point2}.
\end{proof}

Now we prove part (2) of Theorem \ref{m3}.
Assume that $P(n)=P+t$ for some $P\in A(K_n)^-$ and $t\in A[4]$.
Define $\beta\in \Gal(\BQ^\ab/\BQ)$ by
$$\beta=\begin{cases}
r _\infty (-1)r _2(-2)&\text{if $n\equiv 5,7\pmod 8$ or $n\equiv -2 \pmod{16}$},
\\
r _\infty (-1)r _2 (6)&\text{if $n\equiv 6\pmod{16}$}.
\end{cases}
$$
Let $\beta'\in \Gal(\overline\BQ/\BQ)$ be any liftings of $\beta$.
Explicit calculation shows that $\beta$ acts on $K_n$ by $\sqrt{-n}\mapsto-\sqrt{-n}$. It follows that $P(n)^{\beta}+P(n)=t^\beta+t$.

We first treat the case $n\equiv 5,7\pmod 8$.
Then  $t\in A(\BH_n')[4]=A(\BQ(i))$ by Lemma \ref{torsion2}. Note that $\beta$ acts on $\BQ(i)$ trivially. Then
$P(n)^\beta+P(n)=2t \in \BZ\tau(1).$
Apply  $\beta'+1$ to both sides of Proposition \ref{congruence}.
By Lemma \ref{betaaction},
we have
$$\left(i^{\frac{n-1}{2}}
\sum_{\substack{n=d_0d_1\cdots d_\ell\\
d_0\equiv 5\pmod 8\\
d_1\equiv 1, 3\pmod 8\\
d_i\equiv 1 \pmod 8,\ i>1}} \prod_i g(d_i) \right)\tau(\frac{1-i}{2}) +\left(\sum_{\substack{n=d_0d_1\cdots d_\ell\\
d_0\equiv 7\pmod 8\\
d_1\equiv 1, 3\pmod 8\\
d_i\equiv 1 \pmod 8,\ i>1}}\prod_i g(d_i) \right) \tau(\frac12)\in 2A(\BH_n')^{\beta'+1}+\BZ \tau(1).$$
It follows that the contribution from $2A(\BH_n')^{\beta'+1}$ is torsion, which is contained in
$$2A(\BH_n')_{\tor}= 2 A(\BQ(i)) =\BZ\, \tau(1).$$
Then the left-hand side lies in $\BZ\, \tau(1)$.
Thus the coefficients in both of the brackets must be even.

\

Now we treat the case $n\equiv 6\pmod 8$.
In this case we can only have the weaker result
$$P(n)^\beta+P(n)\in (\beta+1)A[4]=A[2]$$
by Lemma \ref{torsion3}.
Apply $\beta'+1$ to Proposition \ref{congruence} again.
We get
\begin{multline*}
\left(\sum_{\substack{n=d_0d_1\cdots d_\ell\\
d_0\equiv 6\pmod 8\\
d_i\equiv 1 \pmod 8,\ i>0}} \prod_i g(d_i)\right)\tau(\frac i2)+ \left(\sum_{\substack{n=d_0d_1\cdots d_\ell\\
d_0\equiv 7\pmod 8\\
d_1\equiv 2\pmod 8\\
d_i\equiv 1 \pmod 8,\ i>1}} \prod_i g(d_i)\right) \tau(\frac12)
\\
+ \left(i\sum_{\substack{n=d_0d_1\cdots d_\ell\\
(d_0,d_1,d_2)\equiv (5,3,2)\pmod 8\\
d_i\equiv 1 \pmod 8, \ i>2}}\prod_{i\geq 1} g(d_i)\right)
\tau(\frac{1-i}{2})
\ \in\ 2A(\BH_n')^{\beta'+1}+A[2].
\end{multline*}
The contribution of $2A(\BH_n')^{\beta'+1}$ is a torsion point, and thus lies in $2A[4]=A[2]$.
Then the left-hand side lies in $A[2]$.
It follows that the first two coefficients have the same parity, which is the same as the assertion of the theorem in this case.
This finishes the proof of Theorem \ref{m3}.

\appendix

\section{Explicit Formulae}

In this appendix, we prove an explicit Waldspurger formula and an explicit Gross--Zagier formula in the case that the character $\chi$ on the quadratic extension is unramified. The results are derived from the original Waldspurger formula (cf. \cite[Theorem 1.4]{YZZ}) and the Yuan--Zhang--Zhang version of the Gross--Zagier formula proved in \cite[Theorem 1.2]{YZZ}.

All global L-functions in this section are complete L-functions with archimedean components normalized to have center $s=1/2$. To avoid confusion, we use $L(s,1_F)$ to denote the complete Dedekind zeta functions of a number field $F$, which is the product of the usual Dedekind zeta function $\zeta_F(s)$ with the gamma factors.

\subsection{Theorem of multiplicity one}

As in \cite[Chapter 1]{YZZ}, the Waldspurger formula and the Gross--Zagier formula can be interpreted as identities of certain one-dimensional spaces of functionals. In this section, we briefly recall the local results about this space of functionals.

Let $F$ be a local field and $B$ a quaternion algebra over $F$. Then $B$ is
isomorphic to either $M_2(F)$ or the unique division quaternion algebra over
$F$. The Hasse invariant $\epsilon(B)=1$ if $B\simeq M_2(F)$, and
$\epsilon(B)=-1$ if $B$ is the division algebra.

Let $K$ be either $F\oplus F$ or a quadratic field extension over $F$, with a
fixed embedding $K\hookrightarrow B$ of algebras over $F$. Let $\eta:F^\times\to \BC^\times$ be the (quadratic or trivial) character associated to the extension $K/F$.

Let $\pi$ be an irreducible admissible representation of $B^\times$ with central
character $\omega_\pi$, and let $\chi:K^\times\to \BC^\times$ be a character of
$K^{\times}$ such that
$$\omega _\pi\cdot \chi|_{F^\times }=1.$$
Define the co-invariant space
$$(\pi\otimes\chi)_{K^\times}:=\{\ell\in \Hom_{\BC}(\pi, \BC): \ell(\pi(t)v)= \chi^{-1}(t) \ell(v), \ \forall\ v\in \pi,\ t\in K^\times
\}.$$

The following result asserts that the dimension of this space is determined by
the local root number of the Rankin--Selberg L-function $L(\frac 12,\pi, \chi)$.

\begin{thm}[Tunnell \cite{Tu}, Saito \cite{Sa}]\label{TS1}
The dimension $\dim\ (\pi\otimes\chi)_{K^\times} \leq 1$, and the equality holds if and only if
 $$\epsilon (B)=\chi(-1)\eta(-1)\epsilon (\frac 12,\pi, \chi).$$
\end{thm}

We also consider the invariant subspace
$$(\pi\otimes \chi)^{K^\times}=\{v \in \pi: \pi(t)v=\chi^{-1}(t)v, \ \forall\ t\in K^\times \}.$$

\begin{cor} \label{TS2}
One has
$$\dim\ (\pi\otimes \chi)^{K^\times} \leq \dim\ (\pi\otimes \chi)_{K^\times}.$$
If $K$ is a field, then the equality holds.
\end{cor}

\begin{proof}
Denote by $\pi^\vee$ the contragredient of $\pi$.
The natural inclusion $\pi\hookrightarrow \Hom_{\BC}(\pi^\vee, \BC)$ induces an injection
$$(\pi\otimes \chi)^{K^\times} \lra (\pi^\vee\otimes\chi^{-1})_{K^\times}.$$
It follows that
$$\dim\ (\pi\otimes \chi)^{K^\times} \leq \dim\ (\pi^\vee\otimes\chi^{-1})_{K^\times}= \dim\ (\pi\otimes \chi)_{K^\times}.$$
Here the last equality follows from the theorem since
$\epsilon (\frac 12,\pi^\vee, \chi^{-1})=\epsilon (\frac 12,\pi, \chi)$.
This proves the first assertion.

For the second assertion, assuming $\dim\ (\pi\otimes \chi)_{K^\times}=1$, we need to construct a nonzero element of $(\pi\otimes \chi)^{K^\times}$. Take $\ell\in (\pi\otimes \chi)_{K^\times}$ and $v\in \pi$ such that $\ell(v)\neq 0$.
Since $K$ is a field, the quotient $K^\times/F^\times$ is compact. Fix a Haar measure on
$K^\times/F^\times$.
Then
$$w=\int_{K^\times/F^\times} \chi(t) \pi(t)vdt$$
is an element of $\pi$.
Furthermore,
$$
\ell(w)= \int_{K^\times/F^\times} \chi(t) \ell(\pi(t) v) dt
= \int_{K^\times/F^\times}  \ell(v) dt
= \vol(K^\times/F^\times)\ \ell(v)
\neq 0.
$$
It follows that $w\neq 0$, which finishes the proof.
\end{proof}

\subsection{Explicit Waldspurger formula}
Let $F$ be a number field and $\BA$ its ring of adeles. Let $B$ be a
quaternion algebra over $F$ and $G$  the algebraic group $B^\times$
over $F$. Denote by  $Z\cong F^\times$ the
center of $G$. Let $\pi$ be a unitary cuspidal automorphic representation
of $G(\BA)$ and $\omega_\pi$ its central character. Let $K$ be a  quadratic field extension over $F$,  $T$ the algebraic group $K^\times$ over $F$,
$\eta$ its associated quadratic Hecke character on $\BA^\times$. Let $\chi:
K^\times \bs \BA_K^\times \ra \BC^\times$ be a Hecke character of finite order.  Assume that
\begin{itemize}
\item $\omega_\pi \cdot  \chi|_{\BA^\times}=1$.
\item For each places $v$ of $F$,
$\epsilon(1/2, \pi_v, \chi_v) \cdot \eta_v(-1)\chi_v(-1)=\mathrm{inv} (B_v)$.
\end{itemize}
It follows that the global root number $\epsilon(1/2, \pi, \chi)$ of  $L(s, \pi,\chi)$ is $+1$ and there is an $F$-embedding $K\subset B$, which we fix once for all and via which
$T$ is viewed as a sub-torus of $G$. By Theorem \ref{TS1}, the space
$$(\pi\otimes\chi)_{T}:=\{\ell\in \Hom_{\BC}(\pi, \BC): \ell(\pi(t)f)= \chi^{-1}(t) \ell(f), \ \forall\ f\in \pi,\ t\in T(\BA)
\}$$
is one-dimensional.

Let $P_\chi:\pi\to\BC$ be the period functional defined by
$$P_\chi(f)=\int_{ T(F)Z(\BA)\bs T(\BA)} f(t) \chi(t) dt,
\quad \forall f\in \pi.$$
Here the Haar measures is normalized by $\vol(Z(\BA)T(F)\bs T(\BA), dt)=2L(1,
\eta)$.
Note that $P_\chi$ is a natural element of $(\pi\otimes\chi)_{T}$.
The Waldspurger formula tells when it is non-zero.

\begin{thm}[Waldspurger formula, \cite{YZZ}, Theorem 1.4] \label{thm4.1}
For any non-zero pure tensor $f=\otimes_v f_v\in \pi$,
$$\frac{|P_\chi(f)|^2}{\quad (f, f)_\Pet\ }=
\frac{1}{2}\frac{ L(1/2, \pi, \chi)}{ L(1, \pi, \ad)L(2,
1_F)^{-1}}\cdot \beta (f).$$
The notations are explained as follows:
\begin{enumerate}[(1)]
\item $\beta(f)=\prod_v \beta_v(f_v)$ is a product over all places $v$ of $F$
and for each $v$,
$$\beta_v (f_v):=\frac{L(1, \eta_v)L(1, \pi_v,
\ad)}{L(2, 1_v)L(1/2, \pi_v, \chi_v)} \int_{Z(F_v)\bs
T(F_v)} \frac{(\pi_v(t_v) f_{v}, f_{v})_v}{(f_{v}, f_{
v})_v} \chi_v(t_v) dt_v,$$
where  $(\ ,\ )_v$  is any non-trivial
$B_v^\times$-invariant Hermitian pairing on $\pi_v$. The Haar measures are normalized by
$\otimes_v dt_v=dt$ and $\vol(Z(\BA)T(F)\bs T(\BA), dt)=2L(1,
\eta)$.
\item $(f, f)_\Pet$ is the Peterson norm of $f\in \pi$ defined by
$$(f, f)_\Pet=\int_{G(F) Z(\BA)\bs
G(\BA)} |f(g)|^2 dg,$$
where the Haar measure $dg$
is the Tamagawa measure such that the volume of $G(F) Z(\BA)\bs
G(\BA)$ is $2$.
\end{enumerate}
\end{thm}

The goal of this subsection is to give an explicit form of Waldspurger's formula under the
following assumptions:
\begin{enumerate}[(a)]
\item $F$ is totally real and  $K$ is quadratic and totally imaginary
over $F$;
\item $\chi_v$ is unramified for each place $v\nmid \infty$;
\item for any $v|\infty$,  the Jacquet-Langlands correspondence $\pi_v^\JL$ is a
discrete series of weight $k_v$ on $\GL_2(\BR)$ with $k_v\geq 2$
even integer.
\end{enumerate}
It follows that the central character $\omega_\pi$ of $\pi$ is unramified everywhere.

Let $O_F$ be the ring of integers in $F$ and $ O_v$ be the ring of integers in $F_v$ for any finite place $v$ of $F$. For any $a\in \BA^\times$, let $|a|$ denote its adelic absolute valuation such that $dax=|a|dx$ for any Haar measure $dx$ on $\BA^\times$. We view $F_v^\times$ and the finite part $\BA_f^\times$ of $F$ as subrings of $\BA^\times$.

Let $N, D, d\in \BA_f^\times$ be such that for any finite place $v$
of $F$, $N_v$ generates the conductor of $\pi_v^\JL$ of $\pi$,  $D_v$ generates the relative
discriminant of $K_v/F_v$, and $d_v$ generates the different of $F_v$.
For each $v\nmid \infty$, let $R_v$ be an order of $B_v:=B\otimes_F F_v$ with
discriminant $N_v O_v$ such that $R_v\cap
K_v= O_{K_v}$. Such an order exists and unique up
to conjugacy of $K_v^\times$. Recall that a Gross-Prasad test vector $f\in \pi$ for the pair $\pi$
and $\chi$ is a pure tensor $f=\otimes_v f_v$ defined as follows (see \cite{GP}).
\begin{enumerate}[(1)]
\item If $v$ is finite with $\ord_v(N_v)\leq 1$ or $K_v/F_v$ is unramified, then
$\pi_v^{R_v^\times}$ is of dimension one and $f_v\in
\pi_v^{R_v^\times}$ is a non-zero vector.
\item If $v$ is finite with $\ord_v(N_v)\geq 2$ and $K_v/F_v$ is ramified,
then the space
$$(\pi_v\otimes \chi_v)^{K_v^\times}
=\{f_v\in \pi_v:\pi_v(t)f_v=\chi_v^{-1}(t)
f_v, \ \forall \ t\in K_v^\times\}$$
 is one-dimensional by Theorem \ref{TS2}.
The vector $f_v$ is any non-zero element in this space.
\item If $v$ is real, let $f_v$ is still any non-zero element of the one-dimensional space $(\pi_v\otimes \chi_v)^{K_v^\times}$.
\end{enumerate}
Thus a Gross-Prasad test vector for $(\pi, \chi)$ is unique up to scalar.

Let $\pi^\JL$ be the Jacquet-Langlands correspondence of $\pi$ on $\GL_2(\BA)$. The Hilbert newform $f'\in \pi^\JL$ is the unique form of level $U_1(N)$ such that  $\SO_2(\BR)\subset \GL_2(F_v)$ acts by the character $\matrixx{\cos\theta}{\sin \theta}{-\sin\theta}{\cos\theta}\mapsto e^{2\pi i k_v \theta}$ for each $v|\infty$, and such that $$L(s, \pi^\JL)=2^{[F:\BQ]} \cdot |d|^{s-\frac{1}{2}} \cdot \int_{F^\times \bs \BA^\times} f'\matrixx{a}{}{}{1} |a|^{s-\frac{1}{2}} d^\times a,$$
where the measure $d^\times a$ is chosen such that
$$\Res_{s=1} \int_{|a|\leq 1, a\in F^\times \bs \BA^\times} |a|^{s-1} d^\times a=\Res_{s=1} L(s, 1_F).$$

\begin{thm}[Explicit Waldspurger Formula] \label{thm4.2}
Assume the
conditions (a), (b) and (c). Let $f'$ be the newform of
$\pi^\JL$ and  $f$  a Gross-Prasad test vector. Then
$$\begin{aligned}
&\frac{1}{(f, f)_\Pet}\cdot
\Big|\sum_{t\in \wh{K}^\times/K^\times \wh{F}^\times
\wh{ O}_K^\times}
f(t)\chi(t)\Big|^2 \\
=&\ \kappa^2 \cdot [ O_K^\times: O_F^\times]^2\cdot 2^{[F:\BQ]}\cdot \frac{L^{((N,
D)\infty)}(1/2, \pi, \chi)}{(f',
f')_\Pet|Dd^2|^{1/2}}\cdot\prod_{v|\infty}(4\pi)^{-(k_v+1)}\Gamma(k_v)\\
&\cdot \prod_{v|N  \mathrm{\ inert}} (1-q_v^{-1})(1+q_v^{-1})^{-1}\cdot
\prod_{v||N \mathrm{\ ramified}} 2(1+q_v^{-1})^{-1}\cdot \prod_{v^2|N \mathrm{\ ramified}} 2(1-q_v^{-1}),
\end{aligned}$$
where $\kappa=1$ or $2$ is the order of the kernel of the natural morphism from the ideal class group of $F$ to that of $K$, and $(N, D)$ is the set of places $v$ such that both $\ord_v(N)$ and $\ord_v(D)$ are positive.
\end{thm}

Here as before, $(f',f')_\Pet$ and $(f, f)_\Pet$ denote the Peterson norms with respect to the Tamagawa measures on $Z\bs \GL_2$ and $Z\bs G$ respectively.
To deduce the explicit formula, we first calculate the local factors.

\begin{prop} \label{prop4.7}
Let $(\pi, \chi, f)$ as above. Let $e_v$ be the ramification index of
$E_v$ over $F_v$.  Then we have
$$|Dd|_v^{-1/2} \ \beta(f_v)=\begin{cases}
\ds e_v(1-q_v^{-e_v})\frac{L(1, \pi_v, \ad)}{L(1/2, \pi_v,
\chi_v)}, \qquad &\text{if $v|N$ non-split},\\
\ds \frac{L(1,1_v)}{L(2,1_v)}, \qquad &\text{if $v|| N$ split},\\
\ds \frac{L(1,1_v)L(1,\pi_v,\ad)}{L(2,1_v)}, \qquad
&\text{if $v^2|N$ split}, \\
1, &\text{otherwise}.
\end{cases}$$
\end{prop}

\begin{remark}
For each place $v\nmid \infty$, note that the central
character $\omega_v$ of $\pi_v$ is unramified and then
$\omega_v=\mu_v^2$ for some unramified character $\mu$. So we
may assume that $\pi_v$ is of trivial central character. Let
$\iota_i, i=1, 2$ be two embeddings of $K_v$ in $B(F_v)$ then they
are conjugate by an element $\gamma\in G(F_v)$. Let $R_1$ be the
order above and $f_1$ a test vectors, then $\gamma R_1\gamma^{-1}$
is an order under $\iota_2$ and $f_2=\pi_v(\gamma)(f_1)$ is a test
vector. We have that $\beta_v(f_1)=\beta_v(f_2)$.
\end{remark}

\begin{proof}
  We are in the local situation, omitting the subscript $v$,  let $K$ denote the quadratic extension of local field $F$. Denote $n = \ord_v(N)$. Reduce to compute the toric integral
  \[\beta^0 = \int_{F^\times \bs K^\times} \frac{\langle \pi(t)f,f \rangle}{\langle f,f \rangle} \chi(t)dt.\]
  If $n >0$ and $K$ is nonsplit, then $f$ is $\chi^{-1}$-eigen and
  $\beta^0 = \vol(F^\times \bs K^\times)$.  For the other cases, the order $R$ in the definition of $V(\pi,\chi)$ is an Eichler order of
discriminant $n$. We fix the following embedding of $K$ so that $R = R_0(n) :=  \begin{pmatrix}
     O &  O \\
    \mathfrak p^n &  O
  \end{pmatrix}
$ and we can take
the test vector as the new vector $W_0$. If $K = F^2$ is
split, embed $K$ into $M_2(F)$ by  $(a,b) \lto \matrixx{a}{}{}{b}$.
 If $K$ is  a field,  take $\tau \in  O_K$ such that $ O_K =  O[\tau]$ and such that if $K/F$ is ramified then
$\tau$ is a uniformizer. Let $\tr \tau, \RN \tau\in F$ denote the trace and norm of $\tau$, respectively.   Embed $K$ into $B$ by
$$a+b\tau \lto
  \begin{pmatrix}
    a+b\tr \tau & b \RN\tau \\
    -b & a
  \end{pmatrix}.$$

Assume $K = F^2$. We write $K^\times = F^\times K^1$ with the image of $K^1$ in $\GL_2(F)$ equal to
$\begin{pmatrix}
  * & \\
    & 1
\end{pmatrix}$. Denote by $\chi_1$ the restriction of $\chi$ to $K^1$, then
\[
  \begin{aligned}
  \beta^0 &= (W_0, W_0)^{-1}
  \iint_{(F^\times)^2} W_0
  \left[
    \begin{pmatrix}
      ab & \\
         & 1
    \end{pmatrix}
  \right]
  \ov{W_0
    \left[
      \begin{pmatrix}
	b & \\
	  & 1
      \end{pmatrix}
    \right]}\chi_1(a) d^\times b d^\times a \\
   &= (W_0,W_0)^{-1} |Z(1/2,W_0,\chi_1)|^2.
  \end{aligned}
\]
Now $Z(1/2,W_0,\chi_1) = \chi_1(d)^{-1}L(1/2,\pi\otimes\chi_1)$ and
\[\beta^0 =  (W_0,W_0)^{-1} L(1/2,\pi,\chi).\]
Now consider the case that $K$ is a field and $\pi$ is unramified. Let
    $$ \Psi(g) := \frac{(\pi(g)W_0,W_0)}{(W_0,W_0)}, \qquad  g\in \GL_2(F).$$
     Then
  $$\beta^0
    =\frac{\vol(K^\times/F^\times)}{\# K^\times/F^\times O_K^\times}
  \sum_{t \in K^\times/F^\times O_K^\times} \Psi(t)\chi(t).$$
  If $K/F$ is unramified, then
    \[\beta^0 = \vol(K^\times/F^\times) = |d|^{1/2}\]
    while if $K/F$ is ramified,
    \[\beta^0 = |Dd|^{1/2} (1 + \Psi(\tau)\chi(\tau)).\]
    Using MacDonald formula for the matrix coefficient $\Psi(\tau)$,
    we obtain
    \[\beta(f) = |Dd|^{1/2}.\]
\end{proof}

\begin{proof}[Proof of Theorem \ref{thm4.2}]
The proof is just a calculation of the right-hand side of the formula in Theorem \ref{thm4.1}.
Apply the expression of $\beta(f_v)$ in Proposition \ref{prop4.7}. We also need the a special value formula for $L(1,\pi,\ad)$, which will be proved in Theorem \ref{adjoint} in the next subsection.
Then the formula gives
$$\begin{aligned}
&\frac{\Big||Dd|^{-1/2}P_\chi(f)\Big|^2}{(f,
f)_\Pet} \\
=&\frac{L^{(\infty)}(1/2, \pi, \chi)}{(f',
f')_\Pet|Dd^2|^{1/2}}\left(\frac{2}{\pi}\right)^{[F:\BQ]}(4\pi)^{-\sum_v
k_v}
\prod_v \Gamma(k_v)\\
&\qquad\qquad\qquad\qquad \cdot \frac{\prod_{v\nmid
\infty} \beta_v(f_v)|Dd|_v^{-1/2}}{L_N(1, \pi, \ad)
\prod_{v|N}(1+q_v^{-1})\prod_{v\| N} (1-q_v^{-2})}\\
=&\frac{L^{(\infty)}(1/2, \pi, \chi)}{(f',
f')_\Pet|Dd^2|^{1/2}}\prod_{v|\infty}\left(2^{1-2k_v}\pi^{-(k_v+1)}\Gamma(k_v)\right)\\
&\qquad \cdot
\prod_{v||N}\frac{\beta_v(f_v)|Dd|_v^{-1/2}}{(1+q_v^{-1})}\cdot
\prod_{v^2|N}\frac{\beta_v(f_v)|Dd|_v^{-1/2}}{L_v(1, \pi_v,
\ad)(1+q_v^{-1})}\cdot \prod_{v\nmid N\infty}
\beta_v(f_v)|Dd|_v^{-1/2}\\
=&\frac{L^{(\infty)}(1/2, \pi, \chi)}{(f',
f')_\Pet|Dd^2|^{1/2}}\prod_{v|\infty}\left(2^{1-2k_v}\pi^{-(k_v+1)}\Gamma(k_v)\right)
\cdot \prod_{v| c} \beta_v(f_v)|Dd|_v^{-1/2}\\
&\qquad \cdot \prod_{v||N\mathrm{\ nonsplit}}\frac{\beta_v(f_v)|Dd|_v^{-1/2}}{(1+q_v^{-1})}\cdot
\prod_{v^2|N\mathrm{\ nonsplit}}\frac{\beta_v(f_v)|Dd|_v^{-1/2}}{L_v(1,
\pi_v, \ad)(1+q_v^{-1})},\\
=&\frac{L^{(\infty)}(1/2, \pi, \chi)}{(f',
f')_\Pet|Dd^2|^{1/2}}\prod_{v|\infty}\left(2^{1-2k_v}\pi^{-(k_v+1)}\Gamma(k_v)\right)\\
&\cdot \prod_{v|N\mathrm{\ inert}} (1-q_v^{-1})(1+q_v^{-1})^{-1}\cdot
\prod_{v||N\mathrm{\ ramified}} (1+q_v^{-1})^{-1}\cdot
\prod_{v^2|N \mathrm{\ ramified}} (1-q_v^{-1})\\
&\cdot \prod_{v|N \mathrm{\ ramified}}2 L(1/2, \pi_v, \chi_v)^{-1}.
\end{aligned}$$
It follows that
$$\begin{aligned}
&\frac{\Big||Dd|^{-1/2}P_\chi(f)\Big|^2}{(f,
f)_\Pet}=\frac{L^{((N, D)\infty)}(1/2, \pi, \chi)}{(f',
f')_\Pet|Dd^2|^{1/2}}\left(\prod_{v|\infty}2^{1-2k_v}\pi^{-(k_v+1)}\Gamma(k_v)\right)\\
&\qquad \cdot \prod_{v|N\mathrm{\ inert}} (1-q_v^{-1})(1+q_v^{-1})^{-1}\cdot
\prod_{v||N\mathrm{\ ramified}} 2(1+q_v^{-1})^{-1}\cdot \prod_{v^2|N\mathrm{\ ramified}} 2(1-q_v^{-1}).
\end{aligned}$$
To finish the proof, we need the following simple result.

\begin{lem} Let $F$ be a totally real field, $K$ a totally
imaginary quadratic extension over $F$, and $\eta$ the associated
quadratic character of $\BA^\times$. Then
$$\frac{2L(1,
\eta)|Dd|^{-1/2}}{\# (\wh{K}^\times/K^\times \wh{F}^\times
\wh{ O}_K^\times)}=
\kappa^{-1}\cdot [ O_K^\times: O_F^\times]^{-1}\cdot 2^{[F:\BQ]}.$$
Here $\kappa=1$ or $2$ is the cardinality of the kernel of natural morphism from the ideal class group of $F$ to that of $K$.
\end{lem}
\begin{proof}
It follows from the  exact sequence
$$1\ra (\wh{F}^\times\cap K^\times \wh{ O}_K^\times)/F^\times
\wh{ O}_F^\times \ra \wh{F}^\times/F^\times \wh{ O}_F^\times\ra  \wh{K}^\times/K^\times
\wh{ O}_K^\times \ra \wh{K}^\times/K^\times
\wh{F}^\times \wh{ O}_K^\times \ra 1,$$
that
$$\# \wh{K}^\times/K^\times \wh{F}^\times \wh{ O}_K^\times=\frac{h_K}{h_F}\cdot \kappa,$$ where we use the fact that
$\wh{K}^\times/K^\times\wh{ O}_K^\times$ is isomorphic to the ideal
class group of $K$ and similarly for $F$. By the ideal class number formula:
$$L(1, \eta)|Dd|^{-1/2}=L(0, \eta)=\frac{h_K}{h_F} \cdot
[ O_K^\times: O_F^\times]^{-1}\cdot 2^{[F:\BQ]-1}.$$ Thus we have
that
$$\frac{L(1,
\eta)|Dd|^{-1/2}}{\# \wh{K}^\times/K^\times \wh{F}^\times
\wh{ O}_K^\times}=  \kappa^{-1}\cdot [ O_K^\times
: O_F^\times]^{-1}\cdot 2^{[F:\BQ]-1}.$$
\end{proof}

Go back to proof of Theorem \ref{thm4.2}. Note that
$$P_\chi(f)=\frac{2L(1, \eta)}{\#\wh{K}^\times/K^\times
\wh{F}^\times\wh{ O}_K^\times}\sum_{t\in \wh{K}^\times/K^\times
\wh{F}^\times \wh{ O}_K^\times} f(t)\chi(t).$$ We have that
$$|Dd|^{-1/2}P_\chi(f)=2^{[F:\BQ]}\kappa^{-1} [ O_K^\times: O_F^\times]^{-1}\sum_{t\in
\wh{K}^\times/K^\times \wh{F}^\times \wh{ O}_K^\times}
f(t)\chi(t).$$ Thus we obtain the desired explicit formula in
Theorem \ref{thm4.2}.
\end{proof}

\subsection{Explicit Gross-Zagier formula}
We first recall the main theorem of \cite{YZZ}.
Let $F$ be a totally real number field and $\BA$ its ring of adeles. Let $X$ be the
Shimura curve over $F$ associated to an incoherent quaternion algebra
$\BB$ over $\BA$ with ramification set $\Sigma$ (containing all
infinite places). Let $\xi$ be the Hodge bundle on $X$, and $J$ be its Jacobian.
Let $A$ be a simple abelian variety defined over $F$ parameterized
by $X$. Then
$$\pi_A:=\Hom_\xi^0(X, A)=\Hom^0(J, A),$$
is a representation $\BB^\times$ over $\BQ$ whose infinite
components are all trivial. It is known that
$M:=\End_{\BB^\times}(\pi_A)=\End^0(A)$ is a number field of degree
$\dim A$ over $\BQ$. Let $(\ ,\ ): \pi_A\times \pi_{A^\vee}\ra M$ be
the perfect $\BB^\times$-pairing given by $(f_1,
f_2)=\vol(X_U)^{-1}(f_{1, U}\circ f_{2, U}^\vee)$, where the
composition using the canonical isomorphism $J_U^\vee\cong J_U$ and
the volume using the measure $dx dy/(2\pi y^2)$ on $\CH$. Let
$$\pair{\ ,\ }_M: \ A(\bar{F})_\BQ \otimes_M A^\vee (\bar{F})_\BQ
\lra M\otimes_\BQ \BR$$ be the $M$-bilinear height pairing whose
trace to $\BR$ is the usual height pairing.

Let $K$ be a totally imaginary quadratic extension of $F$ with a fixed
embedding $K_\BA\hookrightarrow \BB$ over $\BA$. Let $\chi:
\wh{K}^\times/K^\times \ra L^\times$ be a Hecke
character of a finite order valued in a number field $L\supset M$, and also viewed as a Galois character via the class field theory. Assume that
\begin{itemize}
\item $\omega_{\pi_A} \cdot \chi|_{\BA^\times}=1.$
\item $\epsilon(1/2, \pi_v, \chi_v)\cdot \chi_v\eta_v(-1) =\inv(\BB_v)$ for each place $v$ of $F$.
\end{itemize}
Then the global root number is $-1$.
By Theorem \ref{TS1}, these conditions implies that
$(\pi_A\otimes\chi)_{K_\BA^\times}$ is one-dimensional.

Denote
$$A(\chi)=(A(K^{\ab})_\BQ\otimes_ML)^{\Gal(K^{\ab}/K)}.$$
Let $h_0\in \CH$ be the unique fixed point of $K^\times$. It defines a point $P=([h_0,1]_U)_U\in X$. Define the period map $P_\chi:\pi\to A(\chi)$ by
$$P_\chi(f)=\int_{t\in \wh{K}^\times/K^\times \wh{F}^\times}
f(P)^{\sigma_t}\otimes_M \chi(t) dt,$$
where we use the Haar
measure such that the volume of $\wh{K}^\times/K^\times
\wh{F}^\times$ is $2L(1, \eta)$. The Gross-Zagier formula of
Yuan-Zhang-Zhang is as follows.

\begin{thm}[Yuan-Zhang-Zhang \cite{YZZ}]
For any pure tensors $f_1\in \pi_A$ and $f_2\in \pi_{A^\vee}$ with $(f_1, f_2)\neq 0$,
$$\frac{\pair{P_\chi(f_1), P_{\chi{-1}}(f_2)}_L}{(f_1, f_2)}=\frac{L'(1/2, \pi_A,
\chi)}{L(1, \pi_A, \ad) L(2, 1_F)^{-1}}\cdot \beta(f_1\otimes f_2)$$
as an identity in $L\otimes_\BQ \BC$. Here $\pair{\ ,\ }_L:
A(\chi)\times A(\chi^{-1})\ra L\otimes_\BQ \BR$ is the L-linear
N\'{e}ron--Tate height pairing  induced by the $M$-linear N\'{e}ron--Tate
height pairing $\pair{\cdot ,\cdot }_M$ above.
\end{thm}

Note that we can define Gross-Prasad test vectors as in the last subsection.
Then the explicit version of the formula is as follows.

\begin{thm}[Explicit Formula of Gross-Zagier] \label{EGZ}
Assume that $\chi$ is an unramified  character of finite order.
Let $f\in \pi_A$ be a Gross-Prasad test vector. Then
$$\begin{aligned}
&\frac{1}{(f, f)}\cdot \wh{h}\left(\sum_{t\in \wh{K}^\times/\wh{F}^\times
K^\times \wh{ O}^\times_K} f(P)^t \chi(t)\right)\\
=&\ \kappa^2\cdot [ O_K^\times:  O_F^\times]^2\cdot 2^{[F:\BQ]+1}\cdot
\frac{{L'}^{((N, D)\infty)}(1/2, \pi, \chi)}{(f',
f')_\Pet|Dd^2|^{1/2}
(4\pi)^{3[F:\BQ]}}\\
&\cdot \prod_{v|N\mathrm{\ inert}} (1-q_v^{-1})(1+q_v^{-1})^{-1}\cdot
\prod_{v||N\mathrm{\ ramified}} 2(1+q_v^{-1})^{-1}\cdot
\prod_{v^2|N \mathrm{\ ramified}} 2(1-q_v^{-1}).
\end{aligned}$$
Here $\kappa=1$ or $2$ is the order of the morphism from the ideal class group of $F$ to that of $K$, and $(N, D)$ denotes the set of finite places $v$ of $F$ such that both $\ord_v(N)$ and $\ord_v(D)$ are positive.
\end{thm}

The deduction of the theorem is almost the same as that of Theorem \ref{thm4.2}, so we omit it here. One can obtain the original Gross--Zagier formula under the Heegner hypothesis from the above formula.

\subsection{Special value formula of adjoint L-function}

In the proof of Theorem \ref{thm4.2} and Theorem \ref{EGZ}, we have used the following formula.

\begin{thm} \label{adjoint}
Let $F$ be a totally real field and $d_F$ the absolute discriminant of $F$.
Let $\sigma$ be a unitary cuspidal automorphic representation of
$\GL_2(\BA)$, $N\subset  O_F$ its conductor, and  $f$ the newform
in $\sigma$. Assume that $\sigma_v$ is discrete series of weight
$k_v$ for every $v|\infty$. Then
$$\frac{L^{S'}(1, \sigma, \ad) }{|d_F|^{1/2}\cdot
(f,f)_\Pet \cdot L(2, 1_F)}=2^{[F:\BQ]-1+\sum_{v|\infty}k_v}  \prod_{v|N}
(1+q_v^{-1}),$$
where $S'$ is the set of finite places $v$ of $F$
with conductor $n(\sigma_v)\geq 2$.
Equivalently,
$$\frac{L^{(N\infty)}(1, \sigma, \ad)}{|d_F|^{1/2}\cdot (f, f)_\Pet \cdot \zeta_F(2)
}=\frac{ (4\pi)^{\sum_v k_v}}{2\prod_v
\Gamma(k_v)}\cdot \prod_{v|N} (1+q_v^{-1})\prod_{v\| N}
(1-q_v^{-2}).$$
\end{thm}

This formula can be found in the literature (probably under slightly different assumptions). We sketch a proof here for the readers.
Set $G=\GL_2$ over $F$. Let $N$ the unipotent subgroup of $G$ consisting of matrices $\matrixx{1}{a}{}{1}$ in $G$,  and $U=\prod_v U_v$
be a maximal compact subgroup of $G(\BA)$.
We follow \cite[\S 1.6]{YZZ} to normalize the non-trivial additive character $\psi: F\bs \BA\ra \BC^\times$ and Haar measures on $\BA$, $\BA^\times$ and
$G(\BA)$ and their local components.

The proof of Theorem \ref{adjoint} starts with the residue of an Eisenstein series.
For any $\Phi\in \CS(\BA^2)$, define the Eisenstein series
$$E(s, g, \Phi):=\sum_{\gamma \in P(F)\bs G(F)} P(s, \gamma g,
\Phi),$$
where
$$P(s, g, \Phi)=|\det g|^s \int_{\BA^\times} \Phi([0, b]g)|b|^{2s}
d^\times b.$$

\begin{lem} \label{residue}
The Eisenstein series $E(s, g, \Phi)$ has meromorphic continuation to the whole $s$-plane with
only possible poles at $s=1, 0$. In particular,
$$\Res_{s=1}E(s, g, \Phi)
=\frac{1}{2} \wh{\Phi}(0)\cdot \Res_{s=1} L(s, 1_F).$$
\end{lem}
\begin{proof}

By the Poisson summation formula,
$$\begin{aligned}
E(s, g, \Phi)=&\ |\det g|^s \int_{F^\times \bs \BA^\times}
  \left(\sum_{\xi\in
F^2\setminus \{0\}} \Phi(a\xi g)\right) |a|^{2s}d^\times a\\
=&\ |\det g|^s\int_{|a|\geq 1}
\left(\sum_{\xi\in F^2\setminus \{0\}} \Phi(a\xi g)\right) |a|^{2s}d^\times a\\
&+|\det g|^{s-1} \int_{|a|\geq 1}\left(\sum_{\xi\in F^2\setminus
\{0\}} \wh{\Phi}(g^{-1}\xi {}^ta)\right) |a|^{2-2s}d^\times
a\\
&+|\det g|^{s-1}\wh{\Phi}(0)\int_{|a|\leq 1} |a|^{2s-2} d^\times a\\
&-|\det g|^s\Phi(0)\int_{|a|\leq 1} |a|^{2s}d^\times a.
\end{aligned}$$
Thus $E(s, g, \Phi)$ has meromorphic conti
Furthermore,
$$\Res_{s=1}E(s, g, \Phi)=\wh{\Phi}(0)\cdot \lim_{s\ra 1} (s-1)\int_{|a|\leq 1}
|a|^{2s-2} d^\times a
=\frac{1}{2} \wh{\Phi}(0)\cdot \Res_{s=1} L(s, 1_F).$$
\end{proof}

Let $\sigma$ be as in Theorem \ref{adjoint}.
Take any $f_1, f_2\in \sigma$.  Let $W_1, W_2\in \CW(\sigma, \psi)$ be the Whittaker functions
associated to them. Namely, for $i=1, 2$,
$$W_i(g)=\int_{N(F)\bs N(\BA)} f_i(ng) \ov{\psi(n)} dn,$$
where the Haar measure on $N(\BA)$ is the one on $\BA$ via the isomorphism $N(\BA)\cong \BA$.
As in \cite{JC}, consider the integral
$$Z(s, f_1, f_2, \Phi):=\int_{G(F)\bs G(\BA)/Z(\BA)} f_1(g) \ov{f_2(g)} E(s, g, \Phi)dg.$$
By unfolding the Eisenstein series, we obtain
$$Z(s, f_1, f_2, \Phi)
=\int_{N(\BA)\bs G(\BA)} |\det g|^s W_1(g)
\ov{W_2(g)}\Phi([0, 1]g)
dg.$$
It is a product of local factors.
The theorem will be obtained as the residue at $s=1$ of this expression.

For each place $v$ of $F$
and $\Phi_v\in \CS(F_v^2)$, denote the local factor
$$Z(s, W_{1, v}, W_{2, v}, \Phi_v)
:=\int_{N(F_v)\bs G(F_v)} |\det g|^s W_{1,v}(g) \ov{W_{2,v}(g)}\Phi_v([0,
1]g) dg,$$which has meromorphic continuation to the whole $s$-plane.
Moreover, for any $v\nmid \infty$, the fractional ideal of
$\BC[q_v^s, q_v^{-s}]$ generated by all $Z(s, W_{1, v}, W_{2, v}, \Phi_v)$
with $W_{i, v}\in \CW(\sigma_v, \psi_v)$ and $\Phi_v \in \CS(F_v^2)$ is a the same as that generated by $L(s, \sigma_v\times \wt{\sigma}_v)$.

To take the residue, we need to compute $Z(1, W_{1, v}, W_{2, v}, \Phi_v)$.
The following result assert that it is essentially the inner product on $\CW(\sigma_v, \psi_v)$ given by
$$\pair{W_{1, v}, W_{2, v}}=\int_{F_v^\times}W_{1,
v}\matrixx{a}{}{}{1}\ov{W_{2, v}\matrixx{a}{}{}{1}}d^\times a.$$

\begin{lem}\label{lemJC1}
For each $v$,
$$Z(1, W_{1, v}, W_{2, v}, \Phi_v)=\wh\Phi_v(0)\cdot\pair{W_{1, v}, W_{2, v}}.$$
\end{lem}

\begin{proof}
This is a result of \cite{JC}.
For any place $v$ of $F$, let $d'k$ be the Haar measure on $U_v$ determined by the following measure identity on $G(F_v)$:
 $$dg=|b|dx d^\times a d^\times b d'k, \quad\
g=a\matrixx{1}{x}{}{1}\matrixx{1}{}{}{b} k\in
G(F_v).$$
By \cite[p. 51]{JC},
$$Z(1, W_{1, v}, W_{2, v}, \Phi_v)=\int_{F_v^\times} W_{1, v}\matrixx{a}{}{}{1}\ov{W_{2,v}\matrixx{a}{}{}{1}}d^\times a\cdot \iint_{F_v^\times\times U_v}
\Phi_v([0, b]k) |b|^2 d^\times b d'k.$$
By \cite[Lemma 2.3]{JC},
$$\iint_{F_v^\times \times U_v} \Phi([0, b]k) |b|^2 d^\times b dk
=\wh{\Phi}_v (0).$$
The result follows.
\end{proof}

Now we are ready to finish the proof of Theorem \ref{adjoint}.
Let
$\Phi=\otimes_v \Phi_v\in \CS(\BA^2)$ be any element with $\wh{\Phi}(0)\neq 0$,  and let $f_1, f_2$ be pure
tensors. Take  the residues at $s=1$ on the two sides of
$$Z(s, f_1, f_2, \Phi)=\prod_v Z(s, W_{1, v}, W_{2, v}, \Phi_v).$$
Applying Lemmas \ref{lemJC1},  we have
$$(f_1, f_2)_\Pet \cdot \Res_{s=1} E(s, g, \Phi)=
\wh{\Phi}(0) \cdot \Res_{s=1}L(s, \sigma\times
\wt{\sigma})\cdot  \prod_v\frac{\pair{W_{1, v}, W_{2, v}}}{L(1,
\sigma_v \times \wt{\sigma}_v)}.$$
We will see that the product on the right-hand side converges absolutely.
Applying Lemma \ref{residue}, we have
$$\frac{L(1, \sigma, \ad)}{(f_1, f_2)_\Pet}=\frac{1}{2}\prod_v
\frac{L(1, \sigma_v \times \wt{\sigma}_v)}{\pair{W_{1, v}, W_{2, v}}}.$$
Let $f_1=f_2=f$ be the newform and $W^\circ=\otimes_v W_v^\circ$ the
corresponding new vector.
Then
$$\frac{L(1, \sigma, \ad)}{(f, f)_\Pet}=\frac{1}{2}\prod_v
\frac{L(1, \sigma_v \times \wt{\sigma}_v)}{\pair{W_v^\circ, W_v^\circ}}.$$
The proof is complete by the following result on the local factor.

\begin{lem}
$$\frac{L(1, \sigma_v \times \wt{\sigma}_v) L(2, 1_{F_v})^{-1}
|d_v|^{1/2}}{\pair{W_v^\circ, W_v^\circ}}=
\begin{cases}
1, \qquad &\text{$n(\sigma_v)=0, v\nmid \infty$},\\
1+q_v^{-1}, \qquad &\text{$n(\sigma_v)=1, v\nmid \infty$,}\\
(1+q_v^{-1})L(1, \sigma_v, \ad), \ &\text{$n(\sigma_v)\geq 2, v\nmid
\infty$},\\
2^{k_v+1}, \quad &\text{$v|\infty$}.
\end{cases}$$
\end{lem}
\begin{proof}
The result follows directly from the explicit
form of the Kirillov model for the new vector. Note that the new vector $W_v^\circ$ for $v|\infty$ is equal to
$$W_v^\circ(g)=|y|^{k/2}  e^{2\pi
i(x+iy)}e^{ik\theta}1_{\BR_+^\times}(\det g),$$ where
$g=a\matrixx{y}{x}{}{1}k_\theta\in \GL_2(\BR)$, which matches with
$\frac{1}{2}L(s, \sigma_\infty)$ so that the corresponding Hilbert form
is the (normalized) newform.
\end{proof}


\begin{thebibliography}{99}
\addtocontents{Bibliography}

\bibitem{BS}
 B. J. Birch, N. M. Stephens, {\em The parity of the rank of the Mordell-Weil group.} Topology
5, 295-299 (1966).

\bibitem{BCDT}
C. Breuil; B. Conrad; F. Diamond; R. Taylor, {\em On the modularity of elliptic curves over
$\BQ$: wild 3-adic exercises},  J. Amer. Math. Soc. 14 (2001), 843-939.



\bibitem{Coates-Wiles} J. Coates and A. Wiles, {\em On the conjecture of Birch and Swinnerton-Dyer}, Invent. Math. 39 (1977), 233-251.


\bibitem{G2} B. Gross,  {\em Local Orders, Root Numbers, and Modular
Curves}, American Journal of Mathematics 110 (1988), 1153-1182.

\bibitem{GP} B. Gross and D. Prasad, {\em
 Test vectors for linear forms}, Math. Ann. 291,
(1991), 343-355.

\bibitem{GZ} B. Gross and D. Zagier,  {\em Heegner points and derivatives of
$L$-series.} Invent. Math. 84 (1986), no. 2, 225--320.

\bibitem{Heegner} K. Heegner. {\em Diophantische analysis und modulfunktionen}. Math. Z. 56, 227-253 (1952).

\bibitem{Heath-Brown}  D.R. Heath-Brown {\em The size of Selmer
group for the congruent number problem, II}, Invent. math. 118,
331-370 (1994).


\bibitem{JC} H. Jacquet and Chen Nan, {\em Positivity of quadratic
base change L-functions}, Bull Soc. math. France 129 (1), 2001, p.
33-90.

\bibitem{JP}
B. W. Jones and G. Pall, {\em Regular and semi-regular positive ternary
quadratic forms.}  Acta Mathematica, 70: 165--191, 1939.

\bibitem{Kobayashi} S. Kobayshi, {\em The $p$-adic Gross-Zagier
formula for elliptic curves at supersingular primes},  Invent. Math. 191 (2013), no. 3, 527-629.

\bibitem{K}
V. A. Kolyvagain, {\em Euler system}, The Grothendieck Festschrift.
Prog. in ath., Boston, Birkhauser (1990).

\bibitem{K1}
V. A. Kolyvagain, {\em Finiteness of $E(\BQ)$ and $\Sha(E,\BQ)$ for a
subclass of Weil curves}, Math.USSR Izvestiya,Vol. 32 (1989), No. 3.


\bibitem{Monsky}  P. Monsky, {\em Mock Heegner Points and Congruent Numbers}, Math. Z. 204, 45-68 (1990).

\bibitem{Monsky3}  P. Monsky, Appendix to D.R. Heath-Brown {\em The size of Selmer
group for the congruent number problem, II}, Invent. math. 118,
331-370 (1994).


\bibitem{Perrin-Riou} B. Perrin-Riou, {\em  Points de Heegner et d\'eriv\'ees de fonctions
L p-adiques}, Invent. Math. 89 (1987), no. 3, pp. 455-510.

\bibitem{Re}
L.  R\'edei: Arithmetischer Beweis des Satzes \"uber die Anzahl der durch vier teilbaren Invarianten der absoluten Klassengruppe im quadratischen Zahlk\"orper. J. Reine Angew. Math. 171, 55-60 (1934).

\bibitem{Rubin}
K. Rubin, {\em Tate-Shafarevich groups and L-functions of elliptic curves with complex multiplication},  Invent. Math. 89 (1987), no. 3, 527-559.

\bibitem{Rubin2}
K. Rubin, {\em The``main conjectures'' of Iwasawa theory for imaginary quadratic fields. }Invent. Math.103 (1991), no.1, 25-68.

\bibitem{Sa} H. Saito, {\em On Tunnell's formula for characters of $\mathrm{GL}(2)$.} Compositio Math. 85 (1993), no. 1, 99--108.

\bibitem{Tian} Y. Tian, {\em Congruent Numbers and Heegner Points}, Cambridge J. of Math, Vol. 2.1.  117-161, 2014.

\bibitem{Tian1} Y. Tian, {\em Congruent numbers with many prime
factors}, PNAS, Vol 109, no. 52. 21256-21258.

\bibitem{Tu} J. Tunnell, {\em Local $\epsilon$-factors and characters of}
$\mathrm{GL}(2)$. Amer. J. Math. 105 (1983), no. 6, 1277--1307.

\bibitem{Tu2}J Tunnell, {\em  A Classical Diophantine Problem and Modular forms of weight 3/2}.
Inventiones Math. 72. (1983), 323-33.

\bibitem{TW} R. Taylor; A. Wiles,  {\em Ring-theoretic properties of certain Hecke algebras}, Ann. Math. 141 (1995), 553-572.

\bibitem{Wa} J. Waldspurger,  {\em Sur les valeurs de certaines fonctions
L automorphes en leur centre de sym\'etrie.}  Compositio Math. 54
(1985), no. 2, 173--242.

\bibitem{Wi} A. Wiles, {\em Modular elliptic curves and Fermat?s last theorem},
 Ann. Math. 142 (1995), 443-551.

\bibitem{YZZ}  X. Yuan, S. Zhang, and W. Zhang,  {\em The Gross-Zagier formula on Shimura
Curves}, Annals of Mathematics Studies Number 184, 2012.






\end{thebibliography}
\end{document}